\documentclass[11pt]{article}

\usepackage{latexsym}
\usepackage{amssymb}
\usepackage{amsthm}
\usepackage{amscd}
\usepackage{amsmath}
\usepackage{mathrsfs}
\usepackage{graphicx}
\usepackage{shuffle}
\usepackage[colorlinks=true, pdfstartview=FitV, linkcolor=blue, citecolor=blue, urlcolor=blue]{hyperref}
\usepackage{mathtools}
\usepackage[bbgreekl]{mathbbol}
\usepackage{mathdots}
\usepackage{ytableau}
\usepackage[enableskew,vcentermath]{youngtab}
\usepackage{tikz-cd}

\usepackage{mathtools}
\usepackage{amssymb}

\usepackage[colorinlistoftodos]{todonotes}
\usetikzlibrary{chains,scopes}
\usetikzlibrary{shapes.geometric,positioning}

\DeclareSymbolFontAlphabet{\mathbb}{AMSb}
\DeclareSymbolFontAlphabet{\mathbbl}{bbold}

\def\O{\mathsf{O}}

\definecolor{darkred}{rgb}{0.7,0,0} 
\newcommand{\defn}[1]{{\color{darkred}\emph{#1}}} 

\numberwithin{equation}{section}
\allowdisplaybreaks[1]

\theoremstyle{definition}
\newtheorem* {theorem*}{Theorem}
\newtheorem* {proposition*}{Proposition}
\newtheorem* {corollary*}{Corollary}
\newtheorem* {conjecture*}{Conjecture}
\newtheorem{theorem}{Theorem}[section]

\theoremstyle{definition}

\newtheorem* {example*}{Example}

\newtheorem{lemma}[theorem]{Lemma}
\theoremstyle{definition}
\newtheorem{definition}[theorem]{Definition}
\theoremstyle{definition}

\newtheorem{proposition}[theorem]{Proposition}
\newtheorem{corollary}[theorem]{Corollary}

\newtheorem{remark}[theorem]{Remark}
\newtheorem*{remark*}{Remark}
\theoremstyle{definition}
\newtheorem {example}[theorem]{Example}
\theoremstyle{definition}

\theoremstyle{definition}

\theoremstyle{definition}

\def\modu{\ (\mathrm{mod}\ }

\def\({\left(}
\def\){\right)}

\newcommand{\CC}{\mathbb{C}}

\newcommand{\cR}{\mathcal{R}}

\newcommand{\cC}{\mathcal{C}}
\newcommand{\cD}{\mathcal{D}}

\def\cX{\mathcal{X}}
\def\cY{\mathcal{Y}}

\def\CC{\mathbb{C}}

\def\ZZ{\mathbb{Z}}

\def\sl{\mathfrak{sl}}

\def\ch{\mathsf{ch}}

\newcommand{\g}{\mathfrak{g}}

\def\fk{\mathfrak}

\def\barr{\begin{array}}
\def\earr{\end{array}}
\def\ba{\begin{aligned}}
\def\ea{\end{aligned}}
\def\be{\begin{equation}}
\def\ee{\end{equation}}

\def\quand{\quad\text{and}\quad}

\newcommand{\gl}{\mathfrak{gl}}

\def\I{\mathcal{I}}

\def\hs{\hspace{0.5mm}}

\def\ben{\begin{enumerate}}
\def\een{\end{enumerate}}

\def\bei{\begin{itemize}}
\def\eei{\end{itemize}}

\def\hs{\hspace{0.5mm}}

\def\D{\hat D}

\def\Des{\mathrm{Des}}

\def\e{\textbf{e}}

\newcommand{\std}{\operatorname{std}}

\newcommand{\cB}{\mathcal{B}}

\def\iR{\cR^\O}

\def\sC{\mathscr{C}}

\def\q{\mathfrak{q}}

\def\Tab{\textsf{Tab}_{EG}}

\newcommand{\row}{\operatorname{row}}

\newcommand{\Sym}{\textsf{Sym}}

\newcommand{\SymP}{\Sym_{{P}}}
\newcommand{\SymQ}{\Sym_{{Q}}}

\newcommand{\weight}{\operatorname{wt}}

\def\cW{\mathcal{W}}
\newcommand{\Incr}{\operatorname{Incr}}

\def\row{\textsf{row}}

\def\path{\mathsf{path}}

\def\dash{\text{ --- }}

\def\whSym
{\mathfrak{m}\textsf{Sym}}

\def\whQSym
{\mathfrak{m}\textsf{QSym}}

\def\D{\textsf{D}}
\def\SD{\textsf{SD}}

\def\NN{\mathbb{N}}

\def\Tab{\textsf{Tab}}

\def\PHM{P^\O_{\textsf{HM}}}
\def\QHM{Q^\O_{\textsf{HM}}}

\def\PO{P_{\textsf{EG}}^\O}
\def\QO{Q_{\textsf{EG}}^\O}

\newcommand{\ytab}[1]{
\ytableausetup{boxsize = .5cm,aligntableaux=center}
{\small\begin{ytableau}  #1  \end{ytableau}}
}

\def\fks{{\fk s}}
\def\shword{\textsf{shword}}

\def\pair{\textsf{pair}}

\def\T{\mathsf{T}}

\def\BB{\mathbb{B}}

\def\concat{\mathsf{concat}}

\def\marked{\mathsf{marked}}
\def\unprime{\mathsf{unprime}}

\def\ORF{\rf^\O_n}

\def\iR{\cR_{\mathsf{inv}}}

\def\bar{\overline}

\def\fkd{\mathfrak{d}}

\def\ShTab{\mathsf{ShTab}}

\def\primes{\mathsf{primes}}

\def\ck{\textsf{ock}}

\def\double{\mathsf{double}}

\def\Incr{\mathsf{Incr}}

\def\diag{\text{\emph{diag}}}

\def\W{W^{\mathsf{BC}}}
\def\isim{\mathbin{\hat \equiv}}

\usepackage{listings}
\lstdefinelanguage{Sage}[]{Python}
{morekeywords={False,sage,True},sensitive=true}
\definecolor{dblackcolor}{rgb}{0.0,0.0,0.0}
\definecolor{dbluecolor}{rgb}{0.01,0.02,0.7}
\definecolor{dgreencolor}{rgb}{0.2,0.4,0.0}
\definecolor{dgraycolor}{rgb}{0.30,0.3,0.30}

\def\std{\mathsf{standardize}}

\def\qq{\q^+}
\def\SShTab{\mathsf{ShTab}^+}

\def\ORF{\mathsf{Incr}^+}
\def\unpaired{\mathsf{unpaired}}

\def\Thighest{T^{\mathsf{highest}}}
\def\Tlowest{T^{\mathsf{lowest}}}
\def\TTlowest{  \hat T^{\mathsf{lowest}}}

\def\sF{\mathscr{F}}

\usepackage{bbm}
\def\One{\mathbbm{1}}

\def\fkg{\mathfrak{g}}
\def\CCCq{\CC(\hspace{-0.5mm}(q)\hspace{-0.5mm})}

\usepackage{fullpage}
\usepackage[nomarkers,figuresonly,nofiglist]{endfloat}

\begin{document}
\title{Highest weight crystals for Schur $Q$-functions}

\author{
Eric MARBERG\thanks{
Department of Mathematics, Hong Kong University of Science and Technology, {\tt emarberg@ust.hk}.
}
\and
Kam Hung TONG\thanks{
Department of Mathematics, Hong Kong University of Science and Technology, {\tt khtongad@connect.ust.hk}.
}
}
\date{}

\maketitle

\begin{abstract}
Work of Grantcharov et al. develops a theory of abstract crystals for the queer Lie superalgebra $\mathfrak{q}_n$. Such $\mathfrak{q}_n$-crystals form a monoidal category in which the connected normal objects have unique highest weight elements and characters that are Schur $P$-polynomials. This article studies a modified form of this category, whose connected normal objects again have unique highest weight elements but now possess characters that are Schur $Q$-polynomials. The crystals in this category have some interesting features not present for ordinary $\mathfrak{q}_n$-crystals. For example, there is an extra crystal operator, a different tensor product, and an action of the hyperoctahedral group exchanging highest and lowest weight elements. There are natural examples of $\mathfrak{q}_n$-crystal structures on certain families of shifted tableaux and factorized reduced words. We describe extended forms of these structures that give similar examples in our new category.
\end{abstract}

\setcounter{tocdepth}{2}

\section{Introduction}

\subsection{Overview}

 
\defn{Crystals} are an abstraction for the \defn{crystal bases} of 
quantum group representations.
Invented by Kashiwara \cite{Kashiwara1990,Kashiwara1991} and Lusztig \cite{Lusztig1990a,Lusztig1990b} in the 1990s,
crystals may be viewed concretely as
directed acyclic graphs with labeled edges, along with a map assigning weight vectors  to each vertex, satisfying certain axioms.
Isomorphisms of crystals correspond to weight-preserving graph isomorphisms,
while subcrystals correspond to unions of weakly connected graph components.

For each finite-dimensional Lie superalgebra $\fkg$ there is a category of (abstract) $\fkg$-crystals.
The structure of $\fkg$
imposes different requirements for the weight map and edge labels.
These categories have some common features. There is always a direct sum operation $\oplus$ for crystals
corresponding to the disjoint union of directed graphs.
There is also a more subtle notion of a crystal tensor product $\otimes$.
There is a character map $\ch$ assigning to each finite crystal its weight-generating function.
Finally, there is a \defn{standard crystal} $\BB$ corresponding to the vector representation of an associated quantum group $U_q(\fkg)$.

These ingredients are enough to define a full subcategory of \defn{normal $\g$-crystals}:
this consists of the $\g$-crystals whose connected components are each isomorphic to a subcrystal
of $\BB^{\otimes m} $ for some $m\geq 0$.
Such crystals form the smallest monoidal subcategory containing the standard crystal that is closed under isomorphisms, direct sums, and passage to subcrystals.

Defined in this way, the normal $\g$-crystals are typically
 the abstract $\g$-crystals that correspond directly to crystal bases of finite-dimensional integrable $U_q(\g)$-modules.
 This connection implies some desirable properties:
 for example, that each connected normal crystal has a unique \defn{highest weight element} whose weight determines the crystal's isomorphism class.
  In such cases, the character map usually identifies the split Grothendieck group of the category of normal $\g$-crystals
 with a familiar algebra of symmetric polynomials. 
 
 The next section reviews how this works in two cases that have been well-studied, when $\g=\gl_n$ is the complex general linear Lie algebra
 and when $\g=\q_n$ is the \defn{queer Lie superalgebra}.
 Section~\ref{intro3-sect} outlines our main results, which establish similar formal properties of a new category of what we call \defn{$\qq_n$-crystals}.
 It is desirable  to find proofs of crystal properties using only the relevant combinatorial axioms rather than any connection to quantum groups, and this will be
 our approach throughout.

 \subsection{Crystals for Schur functions and Schur $P$-functions}
 
 Let $n$ be a positive integer. 
When $\g=\gl_n$ the edges in each crystal graph are labeled by indices in 
$\{1,2,\dots,n-1\}$ and the weight map takes values in $\ZZ^n$. The standard $\gl_n$-crystal is
 \be\label{standard-gl-eq}  
    \begin{tikzpicture}[xscale=1.8, yscale=1,>=latex,baseline=(z.base)]
    \node at (0,0) (z) {};
      \node at (0,0) (T0) {$\boxed{1}$};
      \node at (1,0) (T1) {$\boxed{2}$};
      \node at (2,0) (T2) {$\boxed{3}$};
      \node at (3,0) (T3) {${\cdots}$};
      \node at (4,0) (T4) {$\boxed{n}$};
      \draw[->,thick]  (T0) -- (T1) node[midway,above,scale=0.75] {$1$};
      \draw[->,thick]  (T1) -- (T2) node[midway,above,scale=0.75] {$2$};
      \draw[->,thick]  (T2) -- (T3) node[midway,above,scale=0.75] {$3$};
      \draw[->,thick]  (T3) -- (T4) node[midway,above,scale=0.75] {$n-1$};
     \end{tikzpicture}
  \ee
  where the weight of $\boxed{i}$ is the standard basis vector $\e_i \in \ZZ^n$. 
  We review the precise definition of \defn{$\gl_n$-crystals}
  and their tensor product in Section~\ref{gl-sect}. 
  
A $\gl_n$-crystal is \defn{normal} if its connected components are each isomorphic to a subcrystal of a tensor power of the standard $\gl_n$-crystal.
A remarkable property of normal $\gl_n$-crystals is that they are characterized by a set of local conditions known as the \defn{Stembridge axioms} \cite{Stembridge2003}.
For this reason such crystals are sometimes called \defn{Stembridge crystals}\footnote{In Bump and Schilling's book \cite{BumpSchilling}, however, \defn{Stembridge crystals} refer to
arbitrary \defn{twists} of what we call normal crystals, where a twist is obtained by translating the weight map by a fixed multiple of $(1,1,\dots,1) \in \ZZ^n$.}.

  A vector $\lambda \in \ZZ^n$ is a \defn{partition} if $\lambda_1 \geq \lambda_2 \geq \dots \geq \lambda_n \geq 0$.
  A \defn{$\gl_n$-highest weight element} of a $\gl_n$-crystal is any vertex in the associated crystal graph with no incoming edges.
 The claims in the following theorem are well-known, and serve as a prototype for subsequent results.
 
\begin{theorem}[{See \cite[Thms. 3.2 and 8.6]{BumpSchilling}}]
\label{gl-thm}
If $\cB$ is a connected normal $\gl_n$-crystal,
then $\cB$ has a unique $\gl_n$-highest weight element, whose 
 weight  $\lambda \in \ZZ^n$  is a partition such that $\ch(\cB)=s_\lambda(x_1,x_2,\dots,x_n)$.
For each partition $\lambda \in \ZZ^n$,
there is a connected normal $\gl_n$-crystal with highest weight $\lambda$,
and 
finite normal $\gl_n$-crystals are isomorphic if and only if they have the same character.
\end{theorem}


Here $s_\lambda(x_1,x_2,\dots,x_n)$ denotes the \defn{Schur polynomial} of a partition $\lambda$ in $n$ commuting variables.
The Schur polynomials indexed by partitions $\lambda \in \ZZ^n$ are a $\ZZ$-basis for the 
subring $\Sym(x_1,x_2,\dots,x_n)$
  of all symmetric polynomials in $\ZZ[x_1,x_2,\dots,x_n]$.
  
The \defn{split Grothendieck group} of an additive category $\sC$ is the abelian group
generated by the symbols $[A]$ for all objects $A \in \sC$, 
subject to the relations $[A]+[B]=[M]$ for all objects with $A \oplus B \cong M$.
  When $\sC$ is monoidal, this group is a ring with multiplication $[A][B] := [A\otimes B]$.
If $\cB$ and $\cC$ are finite crystals, then $\ch(\cB\otimes \cC) = \ch(\cB)\ch(\cC)$,
  so the following is immediate.
  
  \begin{corollary}
The map assigning a $\gl_n$-crystal to its character defines
  a ring isomorphism from the split Grothendieck group of the 
category of finite normal $\gl_n$-crystals to $\Sym(x_1,x_2,\dots,x_n)$.
  \end{corollary}

The general linear Lie algebra $\gl_n$ has two super-analogues given by $\gl_{m|n-m}$ (see \cite{BKK})
and the \defn{queer Lie superalgebra} $\q_n$.
Grantcharov et al. develop a theory of crystals for $\q_n$ in \cite{GJKK10,GJKKK,GJKKK15}.
In this theory,
the edges in each crystal graph are labeled by indices in   $\{ \bar 1, 1,2,\dots,n-1\}$ and the weight map takes values in $\NN^n$ where $\NN := \{0,1,2,\dots\}$.
The standard $\q_n$-crystal is formed by adding a single $\bar 1$-arrow to the standard $\gl_n$-crystal:
 \be\label{standard-q-eq}
    \begin{tikzpicture}[xscale=1.8, yscale=1,>=latex,baseline=(z.base)]
    \node at (0,0) (z) {};
      \node at (0,0) (T0) {$\boxed{1}$};
      \node at (1,0) (T1) {$\boxed{2}$};
      \node at (2,0) (T2) {$\boxed{3}$};
      \node at (3,0) (T3) {${\cdots}$};
      \node at (4,0) (T4) {$\boxed{n}$};
      \draw[->,thick,dashed,color=darkred]  (T0.15) -- (T1.165) node[midway,above,scale=0.75] {$\overline 1$};
      \draw[->,thick]  (T0.345) -- (T1.195) node[midway,below,scale=0.75] {$1$};
      \draw[->,thick]  (T1) -- (T2) node[midway,above,scale=0.75] {$2$};
      \draw[->,thick]  (T2) -- (T3) node[midway,above,scale=0.75] {$3$};
      \draw[->,thick]  (T3) -- (T4) node[midway,above,scale=0.75] {$n-1$};
     \end{tikzpicture}
\ee
For the precise definitions of \defn{$\q_n$-crystals} 
  and their tensor product, see Section~\ref{qn-sect}.
Besides in \cite{GJKK10,GJKKK,GJKKK15},
these crystals have been studied in \cite{AssafOguz,ChoiKwon,GHPS,Hiroshima2018,Hiroshima,Marberg2019b}, for example.

\defn{Normal $\q_n$-crystals} are defined in terms of tensor powers of the standard $\q_n$-crystal in the same way as in the $\gl_n$-case.
The notion of \defn{$\q_n$-highest weight elements} for $\q_n$-crystals is slightly different:
these are again the vertices with no incoming edges, but now in an \defn{extended crystal graph}
  involving additional arrows with labels in $\{\overline{n-1},\dots,\bar 2,\bar 1, 1,2,\dots,n-1\}$; 
  see Definition~\ref{q-highest-def}.  

 A partition $\lambda \in \NN^n$ is \defn{strict} if it has no repeated nonzero entries.
 The following $\q_n$-analogue of Theorem~\ref{gl-thm}
contains several results in \cite{GJKKK}, especially
\cite[Thm. 2.5 and Cor. 4.6]{GJKKK}.

\begin{theorem}[See \cite{GJKKK}]
\label{main-q-thm}
If $\cB$ is a connected normal $\q_n$-crystal,
then $\cB$ has a unique $\q_n$-highest weight element, whose 
 weight  $\lambda \in \NN^n$  is a strict partition such that $\ch(\cB)=P_\lambda(x_1,x_2,\dots,x_n)$.
For each strict partition $\lambda \in \NN^n$,
there is a connected normal $\q_n$-crystal with highest weight $\lambda$,
and 
finite normal $\q_n$-crystals are isomorphic if and only if they have the same character.
\end{theorem}

In this statement, $P_\lambda(x_1,x_2,\dots,x_n)$ is 
the generating function for semistandard shifted tableaux known as a
  \defn{Schur $P$-polynomial} (see Section~\ref{sym-sect} for the definition).
  The fact that the characters of connected normal $\q_n$-crystals are Schur $P$-polynomials
is not explicitly stated in \cite{GJKKK} but can be deduced using \cite[Thm. 2.17]{Serrano}.
The family of Schur $P$-polynomials indexed by strict partitions $\lambda \in \NN^n$ form a basis for a subring $\SymP(x_1,x_2,\dots,x_n)\subset\Sym(x_1,x_2,\dots,x_n)$.

  \begin{corollary}
The map assigning a $\q_n$-crystal to its character defines
  a ring isomorphism from the split Grothendieck group of the 
 category of finite normal $\q_n$-crystals to $\SymP(x_1,x_2,\dots,x_n)$.  \end{corollary}

\subsection{Crystals for Schur $Q$-functions}\label{intro3-sect}

Our main results concern a generalization of the category of $\q_n$-crystals.
We call the objects of this new category \defn{$\qq_n$-crystals}.
Viewed as directed graphs, 
these crystals have 
edges labeled by indices in $\{ \bar 1,0,1,2,\dots,n-1\}$ and weights in $\NN^n$.
The \defn{standard $\qq_n$-crystal} is 
 \be\label{standard-qq-eq}
    \begin{tikzpicture}[xscale=1.8, yscale=1,>=latex,baseline=(z.base)]
    \node at (0,0.7) (z) {};
      \node at (0,0) (T0) {$\boxed{1'}$};
      \node at (1,0) (T1) {$\boxed{2'}$};
      \node at (2,0) (T2) {$\boxed{3'}$};
      \node at (3,0) (T3) {${\cdots}$};
      \node at (4,0) (T4) {$\boxed{n'}$};
      \node at (0,1.4) (U0) {$\boxed{1}$};
      \node at (1,1.4) (U1) {$\boxed{2}$};
      \node at (2,1.4) (U2) {$\boxed{3}$};
      \node at (3,1.4) (U3) {${\cdots}$};
      \node at (4,1.4) (U4) {$\boxed{n}$};
      \draw[->,thick,dashed,color=darkred]  (T0.15) -- (T1.165) node[midway,above,scale=0.75] {$\overline 1$};
      \draw[->,thick]  (T0.345) -- (T1.195) node[midway,below,scale=0.75] {$1$};
      \draw[->,thick]  (T1) -- (T2) node[midway,above,scale=0.75] {$2$};
      \draw[->,thick]  (T2) -- (T3) node[midway,above,scale=0.75] {$3$};
      \draw[->,thick]  (T3) -- (T4) node[midway,above,scale=0.75] {$n-1$};
      \draw[->,thick,dashed,color=darkred]  (U0.15) -- (U1.165) node[midway,above,scale=0.75] {$\overline 1$};
      \draw[->,thick]  (U0.345) -- (U1.195) node[midway,below,scale=0.75] {$1$};
      \draw[->,thick]  (U1) -- (U2) node[midway,above,scale=0.75] {$2$};
      \draw[->,thick]  (U2) -- (U3) node[midway,above,scale=0.75] {$3$};
      \draw[->,thick]  (U3) -- (U4) node[midway,above,scale=0.75] {$n-1$};
      \draw[->,thick,dotted,color=blue]  (U0) -- (T0) node[midway,left,scale=0.75] {$0$};
     \end{tikzpicture}
\ee
where both $\boxed{i}$ and $\boxed{i'}$ have weight $\e_i \in \ZZ^n$. This is isomorphic to the direct sum of
two copies of the standard $\q_n$-crystal, with one additional $0$-arrow.
The $\qq_n$-tensor product is slightly unusual and combines features of queer crystals and of $\gl_{m|n}$-crystals from \cite{BKK} in the degenerate case $m= n=1$. 
For the precise definitions, see Section~\ref{ext-sect}.

\defn{Normal $\qq_n$-crystals} are defined in terms of tensor powers of the standard $\qq_n$-crystal in the same way as in the $\gl_n$- and $\q_n$-cases.
The \defn{$\qq_n$-highest weight elements} of a $\qq_n$-crystal are again the source vertices in a certain extended crystal graph;
see Definition~\ref{qq-highest-def}.
Our main result is the following extension of Theorem~\ref{main-q-thm}, which combines Theorem~\ref{highest-wt-thm}, Corollary~\ref{any-normal-cor}, and Theorem~\ref{sshtab-upgrade}.

\begin{theorem}\label{main-thm}
If $\cB$ is a connected normal $\qq_n$-crystal,
then $\cB$ has a unique $\qq_n$-highest weight element, whose 
 weight  $\lambda \in \NN^n$  is a strict partition such that $\ch(\cB)=Q_\lambda(x_1,x_2,\dots,x_n)  $.
For each strict partition $\lambda \in \NN^n$,
there is a connected normal $\qq_n$-crystal with highest weight $\lambda$,
and 
finite normal $\qq_n$-crystals are isomorphic if and only if they have the same character.
\end{theorem}

Here $Q_\lambda(x_1,x_2,\dots,x_n)$ is the \defn{Schur $Q$-polynomial} of a strict partition $\lambda$,
which is defined to be $2^{\ell(\lambda)} P_\lambda(x_1,x_2,\dots,x_n)$ where $\ell(\lambda)$ is the number of nonzero parts
of $\lambda$.
As $\lambda$ ranges over strict partitions in $\NN^n$ these polynomials are a $\ZZ$-basis 
for another subring $\SymQ(x_1,x_2,\dots,x_n)$. 

  \begin{corollary}
The map assigning a $\qq_n$-crystal to its character defines a ring isomorphism from the split Grothendieck group of the 
  category of finite normal $\qq_n$-crystals to  $\SymQ(x_1,x_2,\dots,x_n)$.  
  \end{corollary}

As an application of Theorem~\ref{main-thm}, we can derive a new shifted Littlewood-Richardson rule for products of Schur $Q$-polynomials.
The classical shifted Littlewood-Richardson rule (see \cite[(8.17)(i)]{Macdonald} or \cite[Thm.~8.3]{Stembridge1989})
expands products of Schur $P$-functions as $\NN$-linear combinations of Schur $P$-functions.
This can be converted to a rule for Schur $Q$-functions by dividing by appropriate powers of two,
but then it is not obvious that the coefficients that appear are all integers. 
Using $\qq_n$-crystals lets us avoid this issue.

For each strict partition $\lambda\in \NN^n$, fix a connected normal $\qq_n$-crystal $\cB_\lambda$ with highest weight $\lambda$.
Using Theorem~\ref{main-thm} to decompose the character of $\cB_\lambda \otimes \cB_\mu$ implies the following:

\begin{corollary}
It holds that
$
Q_{\lambda}(x_1,x_2,\dots,x_n) Q_{\mu}(x_1,x_2,\dots,x_n) = \sum_{\nu} g^{\nu}_{\lambda \mu} Q_{\nu}(x_1,x_2,\dots,x_n)
$
for all strict partitions $\lambda,\mu \in \NN^n$,
where the sum is over all strict partitions $\nu \in \NN^n$ and $g^{\nu}_{\lambda \mu} \in \NN$  is the number of $\qq_n$-highest weight elements in $\cB_\lambda \otimes \cB_\mu$ of weight $\nu$.
\end{corollary}

Queer crystals may be used to show that certain power series are \defn{Schur $P$-positive} in the sense of being positive linear combinations of Schur $P$-functions (see \cite[Cors. 3.34 and 3.38]{Marberg2019b}, for example). A similar application of $\qq_n$ crystals is to demonstrate \defn{Schur $Q$-positivity} (see Corollary~\ref{q-pos-cor}).

The latter is a stronger property compared to Schur $P$-positivity, as is Theorem~\ref{main-thm} compared to Theorem~\ref{main-q-thm}.
Although there is a commutative diagram forgetful functors 
\[\begin{tikzcd}[row sep=small, column sep=small]
\{ \text{ $\qq_n$-crystals }\}  \arrow[rr] \arrow[dr] && \{ \text{ $\q_n$-crystals }\} \arrow[dl] 
\\
 & \{ \text{ $\gl_n$-crystals }\} 
\end{tikzcd}
\]
the horizontal arrow does not take normal $\qq_n$-crystals to normal $\q_n$-crystals.
This means that Theorem~\ref{main-q-thm} does not directly imply 
 similar properties of normal $\qq_n$-crystals. As such, extending Theorem~\ref{main-q-thm} to Theorem~\ref{main-thm}
is nontrivial.

An interesting feature of $\qq_n$-crystals concerns an action of the finite Coxeter group $\W_n$ of type $\mathsf{BC}_n$.
There is an action of the symmetric group $S_n$ on the vertices of normal $\gl_n$- and $\q_n$-crystals.
Under this action, the longest permutation $w_0 \in S_n$ interchanges highest and lowest weight elements; see Proposition~\ref{low-prop}.
This property does not hold for normal $\qq_n$-crystals. Instead, we show that there is an action of $\W_n$ on the vertices 
of normal $\qq_n$-crystals, and for this action the longest preimage of $w_0$ under the projection $\W_n \to S_n$
  interchanges highest and lowest weight elements; see Proposition~\ref{wplus-prop}.

Our strategy for proving Theorem~\ref{main-thm} has the following outline.
 In Section~\ref{crystal-incr-sect}, we describe a $\qq_n$-crystal structure on increasing factorizations of \defn{primed involution words}, which are certain analogues of reduced words for permutations. 
This generalizes a $\q_n$-crystal identified by Hiroshima in \cite{Hiroshima}. It is relatively easy to show that every connected normal $\qq_n$-crystal
 may be embedded in one of these objects; this is carried out later in Section~\ref{morphisms-subsect1}.
 
Next, we show in Section~\ref{crystal-tab-sect} how to extend a $\q_n$-crystal structure on semistandard shifted tableaux
 studied in \cite{AssafOguz,HPS,Hiroshima2018} to a $\qq_n$-crystal on a larger set.
Building on results in \cite{Hiroshima2018}, we are able to prove that each $\qq_n$-crystal of shifted tableaux of a fixed strict partition shape 
 is connected with unique highest and lowest weight elements; see Theorem~\ref{highest-wt-thm}. 
 
In Section~\ref{morphisms-subsect2}
 we show how to embed our $\qq_n$-crystals of increasing factorizations into $\qq_n$-crystals of shifted tableaux.
This requires some technical results from \cite{M2021a} about a shifted form of \defn{Edelman-Greene insertion}.
Combining these steps lets us deduce that each connected normal $\qq_n$-crystal occurs as a crystal of shifted tableaux and therefore has a Schur $Q$-polynomial as its character. 

Our final task in Section~\ref{last-sect} is to
show that all of our $\qq_n$-crystals of shifted tableaux are normal. We can prove this directly for crystals of one-row tableaux.
Each Schur $Q$-polynomial appears as constituent of some product of Schur $Q$-polynomials indexed by one-row partitions.
Using this fact, we deduce that each $\qq_n$-crystal of shifted tableaux occurs  as a full subcrystal of a tensor product of 
crystals of one-row tableaux, and is therefore normal; see Theorem~\ref{sshtab-upgrade}.

\subsection{Connections to representation theory}

We briefly summarize how crystals arise from representation theory.
The quantum group $U_q(\gl_n)$ may be defined as a bialgebra over the field of formal Laurent series $\CCCq$
with generators $e_i$ and $f_i$ indexed by $i \in \{1,2,\dots,n-1\}$. 
These generators give rise to operators $\tilde e_i$ and $\tilde f_i$ on 
  $U_q(\gl_n)$-modules $M$ that are  \defn{integrable} in the sense of \cite[\S1.2]{Kashiwara1991}. 
Each pair $e_i$ and $f_i$ generates a copy of $U_q(\sl_2)$ and $\tilde e_i$ and $\tilde f_i$ are defined in terms of the 
 $U_q(\sl_2)$-decomposition of $M$ \cite[\S2.2]{Kashiwara1991}.

Kashiwara's results in \cite{Kashiwara1990,Kashiwara1991} show that every integrable module $M$ has a \defn{crystal basis}, which consists of a 
pair $(L,B)$ where $L$ is a free $\CC[[q]]$-module with $M = \CCCq \otimes_{\CC[[q]]} L$
and $B \subset L$ is $\CC$-basis of $L/qL$, subject to several conditions involving $\tilde e_i$ and $\tilde f_i$ \cite[\S 2.3]{Kashiwara1991}. In particular, one must have $\tilde e_i (B) \subset B \sqcup \{0\}$ and $\tilde f_i (B) \subset B \sqcup \{0\}$ 
and if $b,c \in B$, then $\tilde f_i (b) = c$ if and only if $\tilde e_i (c) = b$.
This means that much of the information in a crystal basis may be recorded in the \defn{crystal graph}
on $B$ with labeled directed edges $b \xrightarrow {i} c$ whenever $\tilde f_i (b) = c$.
This graph gives an example of a normal $\gl_n$-crystal, and every finite normal $\gl_n$-crystal arises in this way.

 The quantum queer superalgebra $U_q(\q_n)$ is  another bialgebra over $\CCCq$, also 
with generators $e_i$ and $f_i$ indexed by $i \in \{1,2,\dots,n-1\}$ but now with one extra generator $k_{\bar 1}$ \cite[Def. 1.1]{GJKKK15}.
There is a semisimple category of integral modules for $U_q(\q_n)$ \cite[Def. 1.5]{GJKKK15} 
and on such modules $M$ 
the generators of $U_q(\q_n)$ give rise to certain operators $\tilde e_i$ and $\tilde f_i$ 
indexed by $i \in \{\bar 1, 1,2,\dots,n-1\}$ and an additional operator $\tilde k_{\bar 1}$  \cite[\S 2]{GJKKK15}.
 
In this context, a \defn{crystal basis} \cite[Def. 2.2]{GJKKK15} for $M$ also consists of a pair $(L,B)$ where 
$L$ is a free $\CC[[q]]$-module $L$ such that $M = \CCCq \otimes_{\CC[[q]]} L$.
The set $B$, however, is not a basis for $L/qL$ but instead a set of $\tilde k_{\bar 1}$-invariant subspaces that give 
a direct sum decomposition of $L/qL$.
It is again required that $\tilde e_i (B) \subset B \sqcup \{0\}$ and $\tilde f_i (B) \subset B \sqcup \{0\}$, so
for each subspace $b \in B$ the operators $\tilde e_i$ and $\tilde f_i$ must either restrict to the zero map modulo $q$
or an isomorphism $b \xrightarrow{\sim} c$ for some other $c \in B$.
For $b,c \in B$ one agains insists that  $\tilde f_i (b) = c$ if and only if $\tilde e_i (c) = b$.
The corresponding \defn{crystal graph} on $B$ with edges $b \xrightarrow{ i} c$ whenever $\tilde f_i(b) = c$
is a normal $\q_n$-crystal. Grantcharov et al. prove that every integral $U_q(\q_n)$-module
has a crystal basis and that every finite normal $\q_n$-crystal arises from such a basis \cite[Thm. 4.6]{GJKKK15}.

Our concept of normal $\qq_n$-crystals should correspond in a similar way to crystal bases $(L,B)$ for integrable $U_q(\q_n)$-modules,
but with the following additional information. Namely, one must also specify a refined direct sum decomposition of each 
subspace $b \in B$, which is compatible with $\tilde e_i$ and $\tilde f_i$, 
such the action of $\tilde k_{\bar 1}$ on the summands of this decomposition defines a $\gl_{1|1}$-crystal in the sense
of \cite[\S2.3]{BKK}.
We have not yet found a completely satisfactory way of characterizing the data that makes up this kind of \defn{extended crystal basis}
for integrable $U_q(\q_n)$-modules.
For both $\g=\gl_n$ and $\g=\q_n$,
the natural tensor product for integrable $U_q(\g)$-modules gives rise to a tensor product for crystal bases,
and this informs the definition of the relevant tensor product for $\g$-crystals.
We also do not yet fully understand how to motivate the tensor product for $\qq_n$-crystals described in Section~\ref{ext-sect}
from representation theory.
We hope clarify these points in future work.
%

\subsection{Comparison with Gillespie-Levinson-Purbhoo crystals}

On the way to proving Theorem~\ref{main-thm},
we construct a connected normal $\qq_n$-crystal on the set $\SShTab_n(\lambda)$ of semistandard shifted tableaux 
of a given strict partition shape $\lambda$ with all entries at most $n$. The character of this object  is
the Schur $Q$-polynomial $Q_\lambda(x_1,x_2,\dots,x_n)$.
In  \cite{GLP,GilLev} Gillespie, Levinson, Purbhoo study another crystal-like structure on semistandard shifted tableaux.
Their objects are also encoded as certain
directed acyclic graphs with labeled edges, and have characters that are
 Schur $Q$-polynomials. 
 
Several differences offset these formal similarities, and we do not know of any way to derive our crystal
 constructions from those in \cite{GLP,GilLev}  or vice versa. In particular:
\begin{itemize}
\item The vertices in Gillespie, Levinson, Purbhoo's crystal graph are
a proper subset of $\SShTab_n(\lambda)$, consisting of representatives for a certain equivalence relation; see \cite[Def. 2.6]{GLP}.

\item There are $2(n-1)$ edge labels for the crystal graphs in \cite{GLP,GilLev}
whereas the edge labels for our crystals $\SShTab_n(\lambda)$ come from the $(n+1)$-element set $\{\bar 1,0,1,\dots,n-1\}$.
 The crystal operators corresponding to these two sets of edges do not seem to be easily related.
 
\item  There is an axiomatic definition of Gillespie, Levinson, Purbhoo's crystal graphs in \cite{GilLev},
but no notion of a tensor product analogous to the tensor product for $\qq_n$-crystals.
\end{itemize}
Independent of this comparison, it is an interesting open problem to give the category of objects in \cite{GilLev} a monoidal structure
and to relate this to representation theory.

\subsection{Outline}

Here is a brief outline of the rest of this article.
Section~\ref{prelim-sect} explains some notational conventions and preliminaries on symmetric functions.
Section~\ref{abstract-sect} gives the precise definitions of the $\gl_n$-, $\q_n$- and $\qq_n$-crystals
discussed informally above.
In Sections~\ref{words-sect}, \ref{crystal-incr-sect}, and \ref{crystal-tab-sect}
we construct three families of $\qq_n$-crystals---on words, increasing factorizations, and shifted tableaux, respectively.
Then in Section~\ref{morphisms-sect} we describe several morphisms between these crystals, in order to prove Theorem~\ref{main-thm}.

\subsection*{Acknowledgements}

This work was partially supported by grants ECS 26305218 and GRF 16306120
from the Hong Kong Research Grants Council.
We thank Dan Bump for his helpful answers to several questions about crystals,
especially related to the group action discussed in Remark~\ref{weyl-remark}. 
We also thank Travis Scrimshaw for many useful discussions.

\section{Preliminaries}\label{prelim-sect}

Given integers $p,q \in \ZZ$
let $[p,q] := \{ i \in \ZZ : p \leq i \leq q\}$ and $[q] := [1,q]$.
Recall that $\NN := \{0,1,2,\dots\}$.
%
For   $i \in \ZZ$, we set $i' := i -\frac{1}{2}$ and $\ZZ' := \ZZ -\frac{1}{2}$, so that 
$ \ZZ\sqcup \ZZ' = \{ \dots < 0' < 0 < 1' < 1 < \dots \} = \frac{1}{2}\ZZ.$
We refer to elements of $\ZZ'$ as \defn{primed numbers}.
\defn{Removing the prime} for some $i \in \ZZ \sqcup \ZZ'$ means to replace $i$ with $\lceil i \rceil$.
\defn{Adding a prime} to a number  $i \in \ZZ\sqcup \ZZ'$ means to replace it with $\lceil i \rceil-\frac{1}{2}$.
Throughout, we fix a positive integer $n$ and
let $x_1,x_2,x_3,\dots$ be commuting variables.

\subsection{Shifted tableaux}\label{tab-sect}


Assume $\lambda = (\lambda_1 \geq \lambda_2 \geq \dots \geq 0)$ is a partition and $\mu = (\mu_1> \mu_2> \dots \geq 0)$ is a strict partition.
Let $\ell(\lambda) = |\{ i > 0: \lambda_i >0\}|$.
The \defn{diagram} of $\lambda$ is 
the set  $\D_\lambda:= \{ (i,j) : i\in  [\ell(\lambda)] \text{ and } j \in [\lambda_i]\}.$
The \defn{shifted diagram} of  $\mu$
is the set $\SD_\mu:= \{ (i,i+j-1) : (i,j) \in \D_\mu\}.$
A \defn{tableau} of shape $\lambda$  is a map $ \D_\lambda \to \ZZ$. 
A \defn{shifted tableau} of shape $\mu$ is a  map $ \SD_\mu \to \ZZ \sqcup \ZZ'$.

If $T$ is a (shifted) tableau, then we  write $(i,j) \in T$ to indicate that $(i,j)$ belong to the domain of $T$
and we let $T_{ij}$ denote the value assigned to this position.
We draw tableaux in French notation, so that row indices increase from bottom to top and column indices increase from left to right.
If
\be\label{tableau-ex}
S = \ytab{   3 & 3& 7 \\ 1 & 1 & 4 & 6}
\quand 
T = \ytab{ \none & 2' & 2 & 4' \\ 1' & 1 & 1 & 4'}  ,
\ee
then $S$ is a tableau and $T$ is a shifted tableau of shape $\lambda=(4,3)$,
and $S_{23} = 7$ while $T_{23} = 2$. The \defn{(main) diagonal} of a shifted tableau
is the set of positions $(i,j)$ in its domain with $i=j$.

A (shifted) tableau is \defn{semistandard} if its entries are all positive and its rows and columns are weakly increasing,
such that no primed entry is repeated in any row and no unprimed entry is repeated in any column.
The examples 
in \eqref{tableau-ex}
are both semistandard.
%
For   $n \in \NN$,
we write $\Tab_n(\lambda)$ for the set of semistandard tableaux of shape $\lambda$ with all entries in $[n]$,
  $\SShTab_n(\mu)$
 for the set of semistandard shifted tableaux of shape $\mu$ with all entries in $\{1'<1<\dots<n'<n\}$,
 and $ \ShTab_n(\mu)$ for the subset of elements in $ \SShTab_n(\mu)$
with no primed entries on the diagonal.

\subsection{Symmetric polynomials}\label{sym-sect}

Our main reference below is Macdonald's book \cite{Macdonald}.
If $T$ is a (shifted) tableau, then we set
$x^T := x_1^{a_1} x_2^{a_2} \cdots x_n^{a_n}$ where $a_k =  | \{ (i,j) \in T : T_{ij} \in \{k,k'\}\}|$.
The \defn{Schur polynomial}  in $n$ variables corresponding to a partition $\lambda $ is 
then  $
s_\lambda(x_1,x_2,\dots,x_n) := \sum_{T \in \Tab_n(\lambda)} x^T .$
As noted in the introduction, when $\lambda$ varies over all partitions in $\NN^n$ (i.e., over all partitions with at most $n$ parts),
these polynomials are a $\ZZ$-basis for the subring 
 of symmetric polynomials $\Sym(x_1,x_2,\dots,x_n)\subset \ZZ[x_1,x_2,\dots,x_n]$.
%
The \defn{Schur $P$- and $Q$-polynomials}  in $n$ variables indexed by a strict partition $\mu \in \NN^n$ are 
\be
P_\mu(x_1,x_2,\dots,x_n) := \sum_{T \in \ShTab_n(\mu)} x^T
\quand
Q_\mu(x_1,x_2,\dots,x_n) := \sum_{T \in \SShTab_n(\mu)} x^T .\ee
As noted in the introduction, it holds that $Q_\mu(x_1,x_2,\dots,x_n)  = 2^{\ell(\mu)}P_\mu (x_1,x_2,\dots,x_n) $
where $\ell(\mu) $ is the number of nonzero parts of $\mu$.
As $\mu$ varies over all strict partitions in $\NN^n$,
the Schur $Q$-polynomials and Schur $P$-polynomials are $\ZZ$-bases for respective subrings 
\[
\SymQ(x_1,x_2,\dots,x_n)\subset \SymP(x_1,x_2,\dots,x_n)\subset \Sym(x_1,x_2,\dots,x_n);
\]
see
\cite[Chapter III, (8.9)]{Macdonald}.
If $n=1$, then  $\Sym(x_1) = \SymP(x_1) = \ZZ[x_1]$ and $\SymQ(x_1) = 2\ZZ[x_1]$.
When $n\geq 2$ these subrings are characterized as 
\[\ba
\SymP(x_1,\dots,x_n) &= \left\{ f \in \Sym(x_1,\dots,x_n) : f(x_1,-x_1,x_3,\dots,x_n) \in \ZZ[x_3,\dots,x_n]\right\},
\\ 
\SymQ(x_1,\dots,x_n) &=\left \{ f \in \SymP(x_1,\dots,x_n) : f - f(0,x_2,\dots,x_n) \in 2x_1\ZZ[x_1,\dots,x_n]\right\};
\ea\]
see, for example, the discussion in \cite[\S3.3]{IkedaNaruse} with $\beta=0$.

One has
$s_\lambda(x_1,x_2,\dots,x_{n}) = s_{\lambda}(x_1,x_2,\dots,x_{n},0)
$ and $P_\mu(x_1,x_2,\dots,x_{n}) = P_{\mu}(x_1,x_2,\dots,x_{n},0)
$, 
so one can
define formal power series by 
$s_\lambda := \lim_{n\to \infty} s_{\lambda}(x_1,\dots,x_n)$,
$ P_\mu := \lim_{n\to \infty} P_{\mu}(x_1,\dots,x_n)$,
and 
$ Q_\mu := \lim_{n\to \infty} Q_{\mu}(x_1,\dots,x_n) = 2^{\ell(\mu)}P_\mu$,
since the coefficients of any fixed monomial in these sequences of polynomials are eventually constant as $n\to\infty$.
The resulting symmetric elements of $\ZZ[[x_1,x_2,\dots]]$ are the \defn{Schur functions}, \defn{Schur $P$-functions}, and \defn{Schur $Q$-functions}, respectively.

\section{Abstract crystals} \label{abstract-sect}

This section contains the precise definitions of
the abstract and normal $\gl_n$-, $\q_n$- and $\qq_n$-crystals 
discussed in the introduction.
Each of these structures will formally consist of a nonempty set $\cB$ with 
a \defn{weight map} $\weight : \cB \to \ZZ^n$ and a family of \defn{raising operators} $e_i : \cB \to \cB \sqcup \{0\}$
and  \defn{lowering operators} $f_i : \cB \to \cB \sqcup \{0\}$,
where $0 \notin \cB$ is an auxiliary element. When $\cB$ is finite, its \defn{character} is the Laurent polynomial
$\ch (\cB) := \sum_{b \in \cB} x_1^{\weight(b)_1}x_2^{\weight(b)_2}\cdots x_n^{\weight(b)_n}$.

The \defn{crystal graph} associated to this data has vertex set $\cB$ and labeled edges $b\xrightarrow{i} c$ whenever $f_i(b) = c \neq 0$.
This graph determines both the raising and lowering operators, since 
it will always be required for $b,c \in\cB$ that $f_i(b)=c$ if and only if $e_i(c) = b$.

A subset of $\cB$ that forms a weakly connected component in the crystal graph inherits its own crystal structure
and is called a \defn{full subcrystal}.
Within each family of crystals, an \defn{isomorphism}  
will mean a weight-preserving map that defines an isomorphism of the corresponding crystal graphs.

\subsection{Crystals for general linear Lie algebras}\label{gl-sect}

The definition of a \defn{$\gl_n$-crystal} explained below is fairly standard in the literature.
In presenting this material we follow the conventions of Bump and Schilling's book \cite{BumpSchilling}.

Let $\cB$ be a nonempty set with a function $
\weight : \cB\to \ZZ^n
$ and an auxiliary element $0 \notin \cB$. %
For each $i \in [n-1]$, assume that maps
$
e_i,f_i : \cB \to \cB \sqcup \{0\}$ are given.  
We define $\varepsilon_i, \varphi_i : \cB \to \NN\sqcup\{\infty\}$ by 
\be\label{var-eq}\varepsilon_i(b) := \max\left\{ k\in \NN : e_i^k(b)\neq0 \right \}
\quand \varphi_i(b) := \max\left\{ k\in \NN : f_i^k(b) \neq 0\right \}.\ee
We refer 
to 
the $\varepsilon_i$'s and $\varphi_i$'s as \defn{string lengths}.
The value of $\weight(b)$ is the \defn{weight} of $b \in \cB$.
Finally, write $\e_1,\e_2,\dots,\e_n$ for the standard basis of $\ZZ^n$.

\begin{definition}[{See \cite[\S2.2]{BumpSchilling}}]
\label{crystal-def}
The set $\cB$ is an \defn{(abstract) $\gl_n$-crystal} relative to the weight map $\weight$
and the operators $e_i$ and $f_i$
if 
for all $b,c \in \cB$ and $i \in [n-1]$ one has:
\ben
\item[(S1)]  It holds that $e_i(b) = c$ if and only if $f_i(c) = b$, in which case
$\weight(c)- \weight(b) = \e_i - \e_{i+1}.$

\item[(S2)]  Both $\varepsilon_i(b)$ and 
$\varphi_i(b) $ are finite and  $\varphi_i(b) - \varepsilon_i(b) = \weight(b)_i - \weight(b)_{i+1}$.
\een
\end{definition}

More precisely, this is the definition of a \defn{seminormal $\gl_n$-crystal} in \cite{BumpSchilling}.
The character of any finite $\gl_n$-crystal is a symmetric Laurent polynomial
\cite[\S2.6]{BumpSchilling}.

The notion of highest and lowest weight elements for $\gl_n$-crystals is straightforward. Namely,
if $\cB$ is a $\gl_n$-crystal, then
a \defn{$\gl_n$-highest} (respectively, \defn{$\gl_n$-lowest}) \defn{weight element} $b \in \cB$
is an element with  $e_i (b) = 0$ (respectively, $f_i(b) = 0$) for all $i \in [n-1]$.

An essential feature of each of category of crystals is the existence of a nontrivial tensor product.
If  $\cB$ and $\cC$ are nonempty sets, then let $\cB \otimes \cC := \{ b\otimes c : b\in \cB,\ c\in \cC\}$
be the set of formal tensors of elements of $\cB$ with elements of $\cC$.
The next definition follows the ``anti-Kashiwara'' convention.

\begin{theorem}[{See \cite[\S2.3]{BumpSchilling}}]\label{thmdef1}
Let $\cB$ and $\cC$ be $\gl_n$-crystals.
Then $\cB \otimes \cC$ has a unique $\gl_n$-crystal structure
with weight map  
$
\weight(b\otimes c) := \weight(b) + \weight(c)
$
and crystal operators  
\[
e_i(b\otimes c) := \begin{cases}
b \otimes e_i(c) &\text{if }\varepsilon_i(b) \leq \varphi_i(c) \\
e_i(b) \otimes c &\text{if }\varepsilon_i(b) > \varphi_i(c)
\end{cases}
\quand
f_i(b\otimes c) := \begin{cases}
b \otimes f_i(c) &\text{if }\varepsilon_i(b) < \varphi_i(c) \\
f_i(b) \otimes c &\text{if }\varepsilon_i(b) \geq \varphi_i(c)
\end{cases}
\]
for $i \in [n-1]$,  
where  it is understood that $b\otimes 0 = 0\otimes c = 0$.
Moreover, if  $\cD$ is another $\gl_n$-crystal,
then the bijection $(b\otimes c) \otimes d \mapsto b\otimes(c\otimes d)$ 
is a  $\gl_n$-crystal isomorphism $(\cB\otimes\cC) \otimes \cD \cong \cB\otimes(\cC\otimes \cD)$.
\end{theorem}

Let $\One$ be a $\gl_n$-crystal with  a single element, whose weight is $0 \in \ZZ^n$.
The \defn{standard $\gl_n$-crystal} is the crystal in \eqref{standard-gl-eq};
we denote this by $\BB_n$.
As in the introduction, a
 $\gl_n$-crystal is \defn{normal} 
if each of its full subcrystals is isomorphic to a full subcrystal of $\BB_n^{\otimes m}$ for some $m\in \NN$,
where  $\BB_n^0 := \One$.

\begin{remark}\label{bracketrulegln}
The following well-known \defn{signature rule} (discussed, for example, in  \cite[\S2.4]{BumpSchilling}) 
can be used to compute the crystal operators for $\BB_n^{\otimes m}$. Suppose 
$w = \boxed{w_1} \otimes \boxed{w_2} \otimes \dots \otimes \boxed{w_m} \in \BB_n^{\otimes m}$, 
and $i \in [n-1]$. Mark the entries $w_k = i$ by a right parenthesis ``)"  and entries $w_j = i+1$ 
by a left parenthesis ``(". The \defn{i-unpaired indices} in $w$ are the indices $j \in [m]$ with 
$w_j \in \{i, i+1\}$ that are not the positions of matching parentheses. 
Now let $k$ be the last $i$-unpaired index of $w$ with $w_k =i$. If no such index exists 
then $f_i(w) = 0$; otherwise $f_i(w)$ is formed from $w$ by changing $w_k$ to $i+1$. 
Similarly, let $j$ be the first $i$-unpaired index of $w$ with $w_j = i+1$. If no such index 
exists, then $e_i(w) = 0$; otherwise $e_i(w)$ is formed from $w$ by changing $w_j$ to $i$.
\end{remark}


\subsection{Crystals for queer Lie superalgebras}\label{qn-sect}

Suppose $n\geq 2$ and
$\cB$ is a $\gl_n$-crystal with maps $e_{\bar 1},f_{\bar 1}: \cB \to \cB\sqcup\{0\}$,
to be called the \defn{queer raising and lowering operators}.
Define   $\varepsilon_{\bar 1},\varphi_{\bar 1} : \cB \to \NN\sqcup\{\infty\}$
by the formulas in  \eqref{var-eq} with $i=\bar 1$.
Grantcharov et al. introduce the following abstract crystals in \cite[Def. 1.9]{GJKKK}:

\begin{definition}[See \cite{GJKKK,GJKKK15}]
\label{q-def}
The $\gl_n$-crystal $\cB$ is an \defn{(abstract) $\q_n$-crystal} relative to the operators $e_{\bar 1}$ and $f_{\bar 1}$
if the weight map satisfies $\weight(\cB) \subset \NN^n$ and for all $b,c \in \cB$ one has:
\ben
\item[(P1)] 
It holds that $e_{\overline 1}(b)=c$ if and only if $f_{\overline 1}(c) = b$,
in which case 
$\weight(c)- \weight(b) = \e_1 - \e_2$
as well as 
$\varepsilon_i(b)=\varepsilon_i(c)$
and
$\varphi_i(b) = \varphi_i(c)$ for all $i \in [3,n-1]$.

\item[(P2)] If $i \in [3,n-1]$, then $e_i$ and $f_i$ commute
 with $e_{\overline 1}$ and $f_{\overline 1}$.

\item[(P3)] If $\weight(b)_1=\weight(b)_2= 0$, then $(\varepsilon_{\overline 1} + \varphi_{\overline 1})(b) =0$, and otherwise $(\varepsilon_{\overline 1} + \varphi_{\overline 1})(b) =1$.
\een
\end{definition}

We typically consider $\q_n$-crystals when $n\geq 2$, but for convenience we 
 also define an \defn{(abstract) $\q_1$-crystal} to be a nonempty set $\cB$
with a weight map $\weight :\cB \to\NN$.

The definitions in \cite{GJKKK,GJKKK15} 
omit 
(P3), which
 implies 
that the character of any finite $\q_n$-crystal 
is in  $\SymP(x_1,x_2,\dots,x_n)$ \cite[Prop. 2.5]{Marberg2019b}.
In \cite{GJKKK,GJKKK15},  abstract $\q_n$-crystals are also required to be normal as $\gl_n$-crystals,
but it is common to omit this condition.

Our description of the tensor product for $\q_n$-crystals again follows the ``anti-Kashiwara'' convention, which is opposite to that of \cite[Thm. 1.8]{GJKKK}
and \cite[Thm. 2.7]{GJKKK15}.

\begin{theorem}[{See \cite{GJKKK,GJKKK15}}]
\label{q-tensor-thm}
Suppose
 $\cB$ and $\cC$ are $\q_n$-crystals.
Then the $\gl_n$-crystal
$\cB \otimes \cC$
has a unique $\q_n$-crystal structure 
whose queer crystal operators are given by 
\[
\ba
 f_{\overline 1}(b\otimes c) := \begin{cases} 
 b \otimes f_{\overline 1}(c)&\text{if }\weight(b)_1 = \weight(b)_2 = 0
 \\
  f_{\overline 1}(b) \otimes c
&\text{otherwise} 
 \end{cases}
  \\[-12pt]
\\
  e_{\overline 1}(b\otimes c) := \begin{cases} 
 b \otimes e_{\overline 1}(c)
&\text{if }\weight(b)_1 = \weight(b)_2 = 0
 \\
  e_{\overline 1}(b) \otimes c
&\text{otherwise} 
 \end{cases}
 \ea
\]
when $n\geq 2$.
Moreover, if  $\cD$ is another $\q_n$-crystal,
then the bijection $(b\otimes c) \otimes d \mapsto b\otimes(c\otimes d)$ 
is a  $\q_n$-crystal isomorphism $(\cB\otimes\cC) \otimes \cD \cong \cB\otimes(\cC\otimes \cD)$.
\end{theorem}

The \defn{standard $\q_n$-crystal} is the crystal \eqref{standard-q-eq} specified in the introduction;
we denote this object by $\BB_n$.
Crystal graphs of $\q_n$-crystals have edges labeled by indices in $\{\bar 1,1,2,\dots,n-1\}$.

\begin{example}\label{q3tensorq3}
The $\q_3$-crystal $\BB_3 \otimes \BB_3$ has crystal graph
  
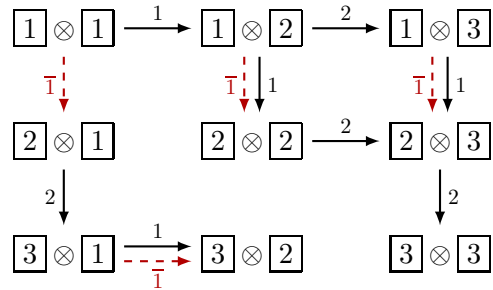
\begin{center}
    \begin{tikzpicture}[xscale=2.5, yscale=1.5,>=latex]
      \node at (0,2) (T31) {$\boxed{1} \otimes \boxed{1}$};
      \node at (1,2) (T32) {$\boxed{1} \otimes \boxed{2}$};
      \node at (2,2) (T33) {$\boxed{1}\otimes \boxed{3}$};
      \node at (0,1) (T21) {$\boxed{2} \otimes \boxed{1}$};
      \node at (1,1) (T22) {$\boxed{2} \otimes \boxed{2}$};
      \node at (2,1) (T23) {$\boxed{2}\otimes \boxed{3}$};
      \node at (0,0) (T11) {$\boxed{3} \otimes \boxed{1}$};
      \node at (1,0) (T12) {$\boxed{3} \otimes \boxed{2}$};
      \node at (2,0) (T13) {$\boxed{3}\otimes \boxed{3}$};
      \draw[->,thick]  (T31) -- (T32) node[midway,above,scale=0.75] {$1$};
      \draw[->,thick]  (T32) -- (T33) node[midway,above,scale=0.75] {$2$};
      \draw[->,thick,dashed,color=darkred]  (T31) -- (T21) node[midway,left,scale=0.75] {$\overline 1$};
      \draw[->,thick]  (T32.285) -- (T22.75) node[midway,right,scale=0.75] {$1$};
      \draw[->,thick,dashed,color=darkred]  (T32.255) -- (T22.105) node[midway,left,scale=0.75] {$\overline{1}$};
      \draw[->,thick]  (T33.285) -- (T23.75) node[midway,right,scale=0.75] {$1$};
      \draw[->,thick,dashed,color=darkred]  (T33.255) -- (T23.105) node[midway,left,scale=0.75] {$\overline{1}$};
      \draw[->,thick]  (T22) -- (T23) node[midway,above,scale=0.75] {$2$};
      \draw[->,thick]  (T21) -- (T11) node[midway,left,scale=0.75] {$2$};
      \draw[->,thick]  (T23) -- (T13) node[midway,right,scale=0.75] {$2$};
      \draw[->,thick]  (T11.7) -- (T12.173) node[midway,above,scale=0.75] {$1$};
     \draw[->,thick,dashed,color=darkred]  (T11.353) -- (T12.187) node[midway,below,scale=0.75] {$\overline{1}$};
     \end{tikzpicture}
  \end{center}
and weight map $\weight(\boxed{i}\otimes \boxed{j}) = \e_i  + \e_j$.
\end{example}

The $1$-element $\gl_n$-crystal $\One$ may be regarded as a $\q_n$-crystal.
A
 $\q_n$-crystal is \defn{normal} 
if each of its full $\q_n$-subcrystals is isomorphic to a full $\q_n$-subcrystal of $\BB_n^{\otimes m}$ for some $m\in \NN$,
where  $\BB_n^0 := \One$.

It remains to give the formal definition of the \defn{$\q_n$-highest weight elements} that are mentioned in Theorem~\ref{main-q-thm}.
This is more complicated than in the $\gl_n$-case and involves the following operators. 
Let $\cB$ be a $\gl_n$-crystal.
For each $i \in [n-1]$ define a map $\sigma_i : \cB \to \cB$ by
\be\label{weyl-action-eq} \sigma_i (b) :=
\begin{cases} e_i^{-k} (b)&\text{if $k\leq0$}\\
f_i^{k} (b)&\text{if $k\geq 0$}
\end{cases}
\quad\text{where }k := \varphi_i(b) - \varepsilon_i(b).
\ee
When we erase all arrows except those of the form $\xrightarrow{\ i \ }$,
the crystal graph of $\cB$ becomes a disjoint union of paths called \defn{$i$-strings},
 and $\sigma_i$ reverses the order of the elements in each $i$-string.
One has $\sigma_i^2(b) = b$, and $\weight(\sigma_i(b))$ 
is obtained from $\weight(b) $ by interchanging $\weight(b)_i$  and $\weight(b)_{i+1}$.

\begin{remark}\label{weyl-remark}
If $\cB$ is a normal $\gl_n$-crystal, then there is a unique group action of $S_n$ on the set $\cB$ in which the transposition $(i,i+1)$ acts as $\sigma_i$
\cite[Thm. 11.14]{BumpSchilling}.
In general there may fail to be such a group action. For example, 
there is a unique  $\gl_3$-crystal with crystal graph
 
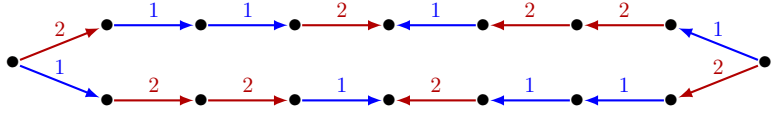
\begin{center}
    \begin{tikzpicture}[xscale=1.25, yscale=0.5,>=latex]
      \node[inner sep=0,outer sep=0] at (-4,0) (1) {$\bullet$};
      \node[inner sep=0,outer sep=0] at (-3,-1) (3) {$\bullet$};
      \node[inner sep=0,outer sep=0] at (-3,1) (5) {$\bullet$};
      \node[inner sep=0,outer sep=0] at (-2,-1) (7) {$\bullet$};
      \node[inner sep=0,outer sep=0] at (-2,1) (9) {$\bullet$};
      \node[inner sep=0,outer sep=0] at (-1,-1) (11) {$\bullet$};
      \node[inner sep=0,outer sep=0] at (-1,1) (13) {$\bullet$};
      \node[inner sep=0,outer sep=0] at (0,-1) (15) {$\bullet$};
      \node[inner sep=0,outer sep=0] at (0,1) (16) {$\bullet$};
      \node[inner sep=0,outer sep=0] at (1,1) (14) {$\bullet$};
      \node[inner sep=0,outer sep=0] at (1,-1) (12) {$\bullet$};
      \node[inner sep=0,outer sep=0] at (2,1) (10) {$\bullet$};
      \node[inner sep=0,outer sep=0] at (2,-1) (8) {$\bullet$};
      \node[inner sep=0,outer sep=0] at (3,1) (6) {$\bullet$};
      \node[inner sep=0,outer sep=0] at (3,-1) (4) {$\bullet$};
      \node[inner sep=0,outer sep=0] at (4,0) (2) {$\bullet$};
      \draw[->,thick,color=blue]  (1) -- (3) node[midway,above,scale=0.75] {$ 1$};
      \draw[->,thick,color=blue]  (11) -- (15) node[midway,above,scale=0.75] {$ 1$};
      \draw[->,thick,color=blue]  (5) -- (9) node[midway,above,scale=0.75] {$ 1$};
      \draw[->,thick,color=blue]  (9) -- (13) node[midway,above,scale=0.75] {$ 1$};
      \draw[->,thick,color=blue]  (2) -- (6) node[midway,above,scale=0.75] {$ 1$};
      \draw[->,thick,color=blue]  (4) -- (8) node[midway,above,scale=0.75] {$ 1$};
      \draw[->,thick,color=blue]  (8) -- (12) node[midway,above,scale=0.75] {$ 1$};
      \draw[->,thick,color=blue]  (14) -- (16) node[midway,above,scale=0.75] {$ 1$};
      \draw[->,thick,color=darkred]  (1) -- (5) node[midway,above,scale=0.75] {$ 2$};
      \draw[->,thick,color=darkred]  (3) -- (7) node[midway,above,scale=0.75] {$ 2$};
      \draw[->,thick,color=darkred]  (7) -- (11) node[midway,above,scale=0.75] {$ 2$};
      \draw[->,thick,color=darkred]  (13) -- (16) node[midway,above,scale=0.75] {$ 2$};
      \draw[->,thick,color=darkred]  (2) -- (4) node[midway,above,scale=0.75] {$ 2$};
      \draw[->,thick,color=darkred]  (6) -- (10) node[midway,above,scale=0.75] {$ 2$};
      \draw[->,thick,color=darkred]  (10) -- (14) node[midway,above,scale=0.75] {$ 2$};
      \draw[->,thick,color=darkred]  (12) -- (15) node[midway,above,scale=0.75] {$ 2$};
     \end{tikzpicture}
  \end{center}
  whose two highest weight elements both have weight $(2,1,0)$,
but on this crystal $\sigma_1\sigma_2\sigma_1 \neq \sigma_2\sigma_1\sigma_2$.
%
This subtlety  is sometimes overlooked in discussions
 involving the $\sigma_i$ maps. 
 \end{remark}

Assume  $\cB$ is a $\q_n$-crystal. 
Define $e_{\bar i}: \cB \to \cB \sqcup\{0\}$ and $  f_{\bar i} : \cB \to \cB \sqcup\{0\}$ for $i\in[2,n-1]$ by
\be
\label{bar-i-eq} 
\ba
e_{\bar i} &:=  (\sigma_{i-1}   \sigma_i)    \cdots  (\sigma_2 \sigma_3) (\sigma_1 \sigma_2) e_{\bar 1} (\sigma_2 \sigma _1) (\sigma_3 \sigma_2) \cdots (\sigma_i   \sigma_{i-1}) = \sigma_{i-1}   \sigma_i   e_{\overline{i-1}}   \sigma_i   \sigma_{i-1},
\\
f_{\bar i} &:=  (\sigma_{i-1}   \sigma_i)    \cdots (\sigma_2 \sigma_3) (\sigma_1 \sigma_2) f_{\bar 1}( \sigma_2 \sigma _1) (\sigma_3 \sigma_2) \cdots (\sigma_i   \sigma_{i-1}) =  \sigma_{i-1}   \sigma_i   f_{\overline{i-1}}   \sigma_i   \sigma_{i-1} ,
\ea
\ee
using the convention that $\sigma_i(0) = 0$. 

\begin{example} 
In Example~\ref{q3tensorq3}, the operator $f_{\bar 2}$ acts as
  \begin{center}
    \begin{tikzpicture}[xscale=2.5, yscale=1.5,>=latex]
      \node at (0,2) (T31) {$\boxed{1} \otimes \boxed{1}$};
      \node at (1,2) (T32) {$\boxed{1} \otimes \boxed{2}$};
      \node at (2,2) (T33) {$\boxed{1}\otimes \boxed{3}$};
      \node at (0,1) (T21) {$\boxed{2} \otimes \boxed{1}$};
      \node at (1,1) (T22) {$\boxed{2} \otimes \boxed{2}$};
      \node at (2,1) (T23) {$\boxed{2}\otimes \boxed{3}$};
      \node at (0,0) (T11) {$\boxed{3} \otimes \boxed{1}$};
      \node at (1,0) (T12) {$\boxed{3} \otimes \boxed{2}$};
      \node at (2,0) (T13) {$\boxed{3}\otimes \boxed{3}$};
      %
      %
      \draw[->,thick]  (T32) -- (T33) node[midway,above,scale=0.75] {$\overline 2$};
     %
      %
      \draw[->,thick]  (T22) -- (T12) node[midway,left,scale=0.75] {$\overline 2$};
      \draw[->,thick]  (T21) -- (T11) node[midway,left,scale=0.75] {$\overline 2$};
      \draw[->,thick]  (T23) -- (T13) node[midway,left,scale=0.75] {$\overline2$};
      %
     \end{tikzpicture}
  \end{center}
which means that $f_{\bar 2}(b) = f_2(b)$ for all crystal elements $b \neq \boxed{2}\otimes \boxed{2}$.
\end{example}

Define $\sigma_{w_0} : \cB \to \cB$ by
\be\label{w0-eq} \sigma_{w_0} := (\sigma_1) (\sigma_2\sigma_1)(\sigma_3\sigma_2\sigma_1) \cdots  (\sigma_{n-1} \cdots \sigma_2\sigma_1).\ee
Each $\sigma_i$ is invertible so $\sigma_{w_0}$ is also invertible, and if 
$ \weight(b) = (\alpha_1,\alpha_2,\dots,\alpha_n)\in \ZZ^n$, then one can check that
$\weight(\sigma_{w_0}(b)) = (\alpha_n,\dots,\alpha_2,\alpha_1)$.
If $\cB$ is normal as a $\gl_n$-crystal, then $\sigma_{w_0}$ gives the action of the
reverse permutation $w_0 := n\cdots 321 \in S_n$ on $\cB$ and $\sigma_{w_0} =\sigma_{w_0}^{-1}$.
For $i \in [n-1]$ define 
 \be e_{\bar i'} := \sigma_{w_0}  f_{\overline{n-i}}  \sigma_{w_0}^{-1}
 \quand
 f_{\bar i'} := \sigma_{w_0}   e_{\overline{n-i}}   \sigma_{w_0}^{-1}.\ee

Denote the indexing sets for these maps by 
 $I := [n-1]$,
$\overline{I} := \{\bar i : i \in I\}$, and $\overline{I}' := \{\bar i' : i \in I\}$.
We will see in a moment the following definition is equivalent to \cite[Def. 1.12]{GJKKK}.
\begin{definition}
\label{q-highest-def}
Suppose $\cB$ is a $\q_n$-crystal.
A \defn{$\q_n$-highest} (respectively, \defn{$\q_n$-lowest}) \defn{weight element} $b \in \cB$
is an element with  $e_i (b) = 0$ for all $i \in I\sqcup \overline I$  (respectively, $f_i(b)  = 0$ for all $i \in I\sqcup \overline{I}')$.
\end{definition}

The unique $\q_n$-highest and $\q_n$-lowest weight elements in $\BB_n$ (respectively, $\BB_3\otimes \BB_3$)
are $\boxed{1}$ and $\boxed{n}$ (respectively, $\boxed{1}\otimes \boxed{1}$ and $\boxed{3}\otimes \boxed{3}$).
Let $\prec$ be the transitive closure of the relation on $\ZZ^n$ that has $v \prec v + \e_i - \e_{i+1}$ for all $i \in [n-1]$.
This relation is a strict partial order. 

\begin{proposition}
Suppose $\cB$ is a $\q_n$-crystal. If $\cB$ has a unique $\q_n$-highest weight element, 
then this element is also 
 the unique element $b \in \cB$ with $\{ c \in \cB : \weight(b) \prec \weight(c)\} = \varnothing$,
 as well as
the unique element $b \in \cB$ with $e_i(b)= 0$ for all $i \in I\sqcup \overline{I}\sqcup \overline{I}'$.
\end{proposition}

\begin{proof}
Let $X := \{ b\in \cB : \text{no }c \in \cB \text{ has }\weight(b) \prec \weight(c)\}$, 
$Y := \{ b \in \cB : e_i(b) = 0\text{ if }i \in  I \sqcup\overline{I} \sqcup \overline{I}'\}$,
and $Z:=\{ b \in \cB : e_i(b) = 0\text{ if }i \in  I \sqcup \overline{I}  \}$.
Since the operators $e_i$, $e_{\overline{i}}$, and $e_{\overline{i}'}$ 
change the weight of an element by 
adding $\e_i-\e_{i+1}$ when they do not act as zero, we have $X \subset Y \subset Z$.

All weights for elements of $\q_n$-crystals are in $\NN^n$,
so if $b_0,b_1,b_2,\dots,b_k \in \cB$ are such $\weight(b_j) - \weight(b_{j-1}) \in \{ \e_i - \e_{i+1} : i \in [n-1]\}$ for all $j \in [k]$ 
then we must have $k \leq \weight(b_0)_1 +2 \weight(b_0)_2+3 \weight(b_0)_3 + \dots + n\weight(b_0)_n$.
Hence any $b \in \cB$ must have $b\prec c$ for some $c \in X$, so the set $X$ is nonempty if $\cB$ is nonempty.
This means that if $\cB$ has a unique $\q_n$-highest weight element $b$ then $\varnothing \neq X \subset Y \subset Z = \{b\}$
so $X=Y=Z=\{b\}$.
\end{proof}

\begin{proposition}\label{low-prop}
Suppose $\cB$ is a normal $\gl_n$-crystal. Then 
$b\in\cB$ is a $\gl_n$-lowest weight element if and only if $\sigma_{w_0}(b) \in \cB$ is a $\gl_n$-highest weight element.
If $\cB$ is also a $\q_n$-crystal, then 
$b\in\cB$ is a $\q_n$-lowest weight element if and only if $\sigma_{w_0}(b) \in \cB$ is a $\q_n$-highest weight element.
\end{proposition}

The characterization of 
$\q_n$-lowest weight elements in this proposition is \cite[Def. 1.12]{GJKKK}.

\begin{proof}
A connected normal $\gl_n$-crystal with highest weight $\lambda=(\lambda_1,\lambda_2,\dots,\lambda_n) \in \NN^n$
may be identified with the set of semistandard tableaux of shape $\lambda$ with all entries at most $n$; see \cite[Chapter 8]{BumpSchilling}.
This set has only one element $T_\lambda$ of weight $\lambda$ and one element $U_\lambda$ of weight $(\lambda_n,\dots,\lambda_2,\lambda_1)$.
As
$(\lambda_n,\dots,\lambda_2,\lambda_1)$ is also the weight of $\sigma_{w_0}(T_\lambda)$,
we must have $\sigma_{w_0}(T_\lambda) = U_\lambda$.

The tableaux $T_\lambda$ and $U_\lambda$ are interchanged by the \defn{Lusztig involution}, 
which also swaps highest and lowest weight elements; see \cite[Exercises 5.1 and 5.2]{BumpSchilling} or \cite[\S2.4]{Lenart}.
Since $T_\lambda$ is the unique highest weight element \cite[Thm. 3.2]{BumpSchilling}, 
 $U_\lambda$ is therefore the unique lowest weight element.
  The involution $\sigma_{w_0}$ therefore interchanges highest and lowest weight elements in normal $\gl_n$-crystals.

When $\cB$ is a $\q_n$-crystal that is normal as a $\gl_n$-crystal, $\sigma_{w_0} f_{\bar i'}  = e_{\overline{n-i}}\sigma_{w_0}$ for $i \in [n-1]$ by definition,
so
 $f_i(b) =0$ for all $i \in I\sqcup \overline{I}'$ if and only if $e_i \sigma_{w_0}(b) = 0$ for all $i\in I\sqcup \overline{I}$.
\end{proof}

\subsection{Extended queer supercrystals}\label{ext-sect}

In this section $n$ is allowed to be any positive integer.
Suppose $\cB$ is a $\q_n$-crystal with additional maps $e_{0},f_{0}: \cB \to \cB\sqcup\{0\}$.
Define   $\varepsilon_{0},\varphi_{0} : \cB \to \NN\sqcup\{\infty\}$
by the formula \eqref{var-eq} with $i=0$.
The following extension of Definition~\ref{q-def} is our primary subject.

\begin{definition}
\label{qq-def}
The $\q_n$-crystal $\cB$ is an \defn{(abstract) $\qq_n$-crystal}
relative to the operators $e_0$ and $f_0$ if for all $b,c \in \cB$ one has:
\ben
\item[(Q1)]  It holds that  $e_{0}(b)=c$ if and only if $f_{0}(c) = b$, 
in which case $\weight(b) = \weight(c)$ as well as
$\varepsilon_i(b)=\varepsilon_i(c)$
and
$\varphi_i(b) = \varphi_i(c)$ for all $i \in [n-1]$ and also for $i=\bar 1$ if $n\geq 2$.

\item[(Q2)] If $i \in [2,n-1]$, then $e_i$ and $f_i$ commute with  $e_{0}$ and $f_{0}$.

\item[(Q3)]
If $\weight(b)_1= 0$, then $(\varepsilon_{0} + \varphi_{0})(b) =0$, and otherwise $(\varepsilon_{0} + \varphi_{0})(b) =1$.

\een
\end{definition}

 Here is a first link between $\qq_n$-crystals and Schur $Q$-functions:

\begin{proposition}
If $\cB$ is a finite $\qq_n$-crystal,
then $\ch (\cB)\in\SymQ(x_1,x_2,\dots,x_n)$.
 \end{proposition}

\begin{proof}
Let $\cB$ be a finite $\qq_n$-crystal.
Since $\qq_n$-crystals are $\q_n$-crystals, we know that
$\ch(\cB)\in \SymP(x_1,x_2,\dots,x_n)$ \cite[Prop. 2.5]{Marberg2019b}.
As noted in Section~\ref{sym-sect}, an element $f(x_1,x_2,\dots,x_n) \in  \SymP(x_1,x_2,\dots,x_n)$ belongs to $ \SymQ(x_1,x_2,\dots,x_n)$ 
if  
$f- f(0,x_2,\dots,x_n) $ is divisible by $ 2x_1$; this holds even if $n=1$.
The difference
$\ch(\cB)(x_1,x_2,\dots,x_n)- \ch(\cB)(0,x_2,\dots,x_n)$ is the sum of
$x^{\weight(b)}$ over all $b \in \cB$ with $\weight(b)_1> 0$.
As all such elements $b$ satisfy $\varepsilon_0(b) + \varphi_0(b) = 1$ by (Q3),
this equals $2x_1 \sum_{b} x^{\weight(b)-\e_1}$
where the sum is over all $b \in \cB$ with $\weight(b)_1>0$ and $e_0(b)\neq 0$,
as needed.
\end{proof}

The following result gives 
a  tensor product for $\qq_n$-crystals. This
is more complicated than for $\q_n$-crystals,
but will turn out to have several desirable properties.

\begin{theorem}\label{qq-thmdef}
Let $\cB$ and $\cC$ be $\qq_n$-crystals.
Then the $\gl_n$-crystal $\cB\otimes \cC$
has a unique $\qq_n$-crystal structure
in which
$e_0$ and $f_0$ are given by 
\[\ba
  e_{0}(b\otimes c) := \begin{cases} 
 e_0(b) \otimes c&\text{if }\weight(b)_1 \neq 0
 \\
b \otimes e_0(c) & \text{if }\weight(b)_1  = 0
 \end{cases}
\quand
 f_{0}(b\otimes c) := \begin{cases} 
 f_0(b) \otimes c&\text{if }\weight(b)_1 \neq 0
 \\
b \otimes f_0(c) & \text{if }\weight(b)_1 = 0
 \end{cases}
\ea
\]
and in which (when $n\geq 2$)  $e_{\bar 1}$ and $f_{\bar 1}$ are given by  
\[
\ba
  e_{\overline 1}(b\otimes c) &:= \begin{cases} 
 b \otimes e_{\overline 1}(c)
 &\text{if }\weight(b)_1 = \weight(b)_2 = 0
 \\
 f_0  e_{\bar 1} (b) \otimes e_0(c) 
 &\text{if $\weight(b)_1 = 0$, $f_0e_{\bar 1}(b) \neq 0$, and $e_0(c)\neq 0$}
\\
 e_0  e_{\bar 1} (b) \otimes f_0(c) 
 &\text{if $\weight(b)_1 = 0$, $e_0e_{\bar 1}(b)\neq 0$, and $f_0(c) \neq 0$}
 \\
  e_{\overline 1}(b) \otimes c
&\text{otherwise}
 \end{cases}
  \\[-12pt]
\\
 f_{\overline 1}(b\otimes c) &:= \begin{cases} 
 b \otimes f_{\overline 1}(c)&\text{if }\weight(b)_1 = \weight(b)_2 = 0
 \\
  f_{\overline 1}f_0(b) \otimes e_0(c)
 &\text{if }\weight(b)_1 = 1, \text{ $f_{\overline 1}f_0(b)\neq 0 $, and $e_0(c)\neq 0$}
\\
  f_{\overline 1}e_0(b) \otimes f_0(c)
 &\text{if }\weight(b)_1 = 1, \text{ $f_{\overline 1}e_0(b) \neq 0$, and $f_0(c) \neq 0$}
  \\
    f_{\overline 1}(b) \otimes c &\text{otherwise}
 \end{cases}
 \ea
\]
where it is understood that $b\otimes 0 = 0\otimes c = 0$.
Moreover, if  $\cD$ is another $\qq_n$-crystal,
then the bijection $(b\otimes c) \otimes d \mapsto b\otimes(c\otimes d)$ 
is a  $\qq_n$-crystal isomorphism $(\cB\otimes\cC) \otimes \cD \cong \cB\otimes(\cC\otimes \cD)$.
\end{theorem}

\begin{proof}
When $e_0$ or $f_0$ do not act as zero, they
do not change the values of 
$\varepsilon_{i}$ or $\varphi_{i}$ for any $i \neq 0$ by property (Q1).
From this observation, the conditions in Definitions~\ref{q-def} and \ref{qq-def}
are straightforward to derive from the $\qq_n$-crystal axioms for $\cB$ and $\cC$,
so $\cB \otimes \cC$ is 
a $\qq_n$-crystal.

Now suppose $\cB$, $\cC$, and $\cD$ are $\qq_n$-crystals.
The natural bijection
$(\cB\otimes\cC) \otimes \cD \xrightarrow{\sim} \cB\otimes(\cC\otimes \cD)$
commutes with the $\gl_n$-crystal operators and also with $e_0$ and $f_0$, while preserving the weight map.
It remains to check that this map commutes with $e_{\bar 1}$ and $f_{\bar 1}$. This requires a somewhat involved case analysis.
Fix $b \in\cB$, $c \in \cC$, and $d \in \cD$.
We check that $e_{\bar 1}(b \otimes (c\otimes d)) = e_{\bar 1}((b \otimes c)\otimes d)$:
\begin{itemize}
\item[(a)] Assume that $\weight(b)_1 = \weight(b)_2 = 0$.
If $\weight(c)_1 = \weight(c)_2 = 0$, then 
\[
e_{\bar 1}(b \otimes (c\otimes d)) = e_{\bar 1}((b \otimes c)\otimes d)=
b\otimes c\otimes e_{\bar 1}(d).
\]
If $\weight(c)_1 = 0$, $f_0e_{\bar 1}(c) \neq 0$, and $e_0(d)\neq 0$, then
\[
e_{\bar 1}(b \otimes (c\otimes d)) = e_{\bar 1}((b \otimes c)\otimes d)=
b\otimes f_0 e_{\bar 1}(c)\otimes e_{0}(d)
\]
since in this case we have  $\weight(b\otimes c)_1 = 0$ and
$f_0e_{\bar 1}(b\otimes c)= b\otimes f_0e_{\bar 1}(c) \neq 0$.
It follows similarly that
if 
$\weight(b)_1 = 0$, $e_0e_{\bar 1}(b)\neq 0$, and $f_0(c) \neq 0$, then
\[
e_{\bar 1}(b \otimes (c\otimes d)) = e_{\bar 1}((b \otimes c)\otimes d)=
b\otimes e_0 e_{\bar 1}(c)\otimes f_{0}(d),
\]
and that in the remaining case 
$
e_{\bar 1}(b \otimes (c\otimes d)) = e_{\bar 1}((b \otimes c)\otimes d)=
b\otimes e_{\bar 1}(c)\otimes d.
$

\item[(b)] Assume that $\weight(b)_1 = 0$, $f_0e_{\bar 1}(b) \neq 0$, and $e_0(c\otimes d)\neq 0$.
If  $e_0(c\otimes d) = e_0(c) \otimes d$, then we must have 
$e_0(c) \neq 0$ and $\weight(b\otimes c)_1 = \weight(c)_1 \neq 0$, so $e_{\bar 1}((b \otimes c)\otimes d) = e_{\bar 1}(b \otimes c)\otimes d$
and thus
\[
e_{\bar 1}(b \otimes (c\otimes d)) = e_{\bar 1}((b \otimes c)\otimes d)=
f_0e_{\bar 1}(b)\otimes e_0(c)\otimes d.
\]
If $e_0(c\otimes d) = c\otimes e_0(d)$, then  
$\weight(b\otimes c)_1 =0$, $f_0e_{\bar 1}(b\otimes c) = f_0e_{\bar 1}(b)\otimes c \neq 0$, and $e_0(d) \neq 0$,
so
\[
e_{\bar 1}(b \otimes (c\otimes d)) = e_{\bar 1}((b \otimes c)\otimes d)=
f_0e_{\bar 1}(b)\otimes c\otimes e_0(d).
\]

\item[(c)] Next assume that
$\weight(b)_1 = 0$, $e_0e_{\bar 1}(b)\neq 0$, and $f_0(c\otimes d) \neq 0$.
This case is similar to the previous one.
If $f_0(c\otimes d) = f_0(c) \otimes d$, then one checks that
\[
e_{\bar 1}(b \otimes (c\otimes d)) = e_{\bar 1}((b \otimes c)\otimes d)=
e_0e_{\bar 1}(b)\otimes f_0(c)\otimes d
\]
and if $f_0(c\otimes d) = c \otimes f_0(d)$, then one checks that
\[
e_{\bar 1}(b \otimes (c\otimes d)) = e_{\bar 1}((b \otimes c)\otimes d)=
e_0e_{\bar 1}(b)\otimes c\otimes f_0(d).
\]

\item[(d)] Finally suppose  that $\weight(b)_1 \neq 0$ or $\weight(b)_2\neq 0$,
and that if $\weight(b)_1 = 0$, then we have (1) $f_0e_{\bar 1}(b) =0$ or $e_0(c\otimes d)= 0$ 
and also (2)  $e_0e_{\bar 1}(b) =0$ or $f_0(c\otimes d)= 0$.
We claim that
\[
e_{\bar 1}(b \otimes (c\otimes d)) =
 e_{\bar 1}((b \otimes c)\otimes d)= e_{\bar 1}(b) \otimes c\otimes d.\]
The first and last terms are equal by assumption.
The second equality holds if $\weight(b\otimes c)_1 \neq 0$ 
as then we must have $ e_{\bar 1}(b \otimes c) = e_{\bar 1}(b) \otimes c$
since if $\weight(b)_1 = 0$, then $\weight(c)_1 \neq 0$,
which means that $e_0(c\otimes d)= 0$ $\Leftrightarrow$ $e_0(c)= 0$  
and
$f_0(c\otimes d)= 0$ $\Leftrightarrow$ $f_0(c)= 0$.
 
 Assume $\weight(b\otimes c)_1 = 0$, so that $\weight(b)_1 = \weight(c)_1 = 0$ and $\weight(b\otimes c)_2  \geq \weight(b)_2 > 0$.
Then 
 \[
 e_{\bar 1}(b\otimes c) =e_{\bar 1}(b)\otimes c,
 \quad
  f_0e_{\bar 1}(b\otimes c) = f_0e_{\bar 1}(b)\otimes c,
  \quand
 e_0e_{\bar 1}(b\otimes c) = e_0e_{\bar 1}(b)\otimes c,\]
 while we also have
 $e_0(c\otimes d)= c\otimes e_0(d)$  and  $f_0(c\otimes d)= c\otimes f_0(d)$.
 It follows from (1) and (2) that
 $ e_{\bar 1}((b \otimes c)\otimes d) =  e_{\bar 1}(b \otimes c)\otimes d$ which equals $e_{\bar 1}(b)\otimes c\otimes d $ as needed.

\end{itemize} 
This shows that the $e_{\bar 1}$ 
 operator commutes with 
the natural bijection
$(\cB\otimes\cC) \otimes \cD \xrightarrow{\sim} \cB\otimes(\cC\otimes \cD)$.
A similar argument shows that $f_{\bar 1}$ also commutes with this map.
\end{proof}

The \defn{standard $\qq_n$-crystal} is the crystal described by \eqref{standard-qq-eq} in the introduction;
we denote this object by $\BB^+_n$.
Crystal graphs of $\qq_n$-crystals have edges labeled by indices in $\{\bar 1,0,1,\dots,n-1\}$.

\begin{example}\label{qq-bb-bb-ex}
The crystal graph of $\BB^+_2\otimes \BB^+_2$ is 
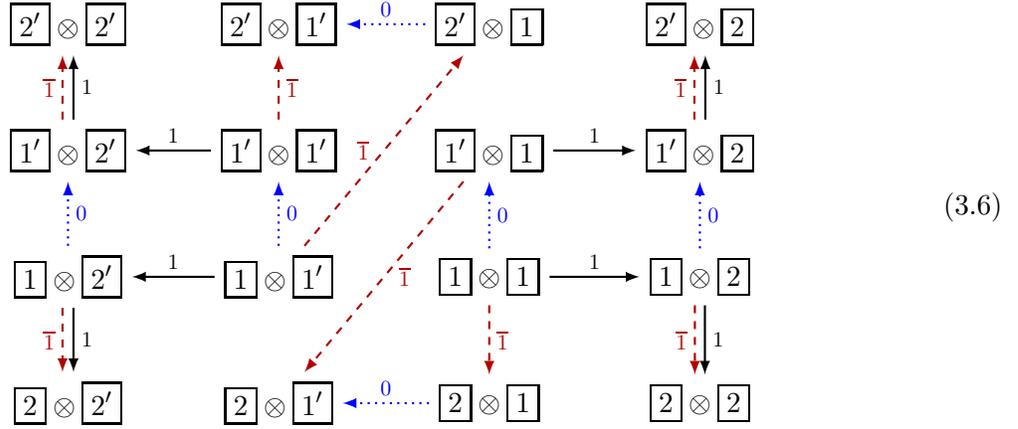
\be\label{bb2-eq}
    \begin{tikzpicture}[xscale=2.0, yscale=1.2,>=latex,baseline=(z.base)]
    \node at (0,2.1) (z) {};
      \node at (0,0) (22') {$\boxed{2}\otimes \boxed{2'}$};
     \node at (0,1.4) (12') {$\boxed{1}\otimes \boxed{2'}$};
      \node at (0,2.8) (1'2') {$\boxed{1'}\otimes \boxed{2'}$};
      \node at (0,4.2) (2'2') {$\boxed{2'}\otimes \boxed{2'}$};
      \node at (1.4,0) (21') {$\boxed{2}\otimes \boxed{1'}$};
     \node at (1.4,1.4) (11') {$\boxed{1}\otimes \boxed{1'}$};
      \node at (1.4,2.8) (1'1') {$\boxed{1'}\otimes \boxed{1'}$};
      \node at (1.4,4.2) (2'1') {$\boxed{2'}\otimes \boxed{1'}$};
      \node at (2.8,0) (21) {$\boxed{2}\otimes \boxed{1}$};
     \node at (2.8,1.4) (11) {$\boxed{1}\otimes \boxed{1}$};
      \node at (2.8,2.8) (1'1) {$\boxed{1'}\otimes \boxed{1}$};
      \node at (2.8,4.2) (2'1) {$\boxed{2'}\otimes \boxed{1}$};
      \node at (4.2,0) (22) {$\boxed{2}\otimes \boxed{2}$};
     \node at (4.2,1.4) (12) {$\boxed{1}\otimes \boxed{2}$};
      \node at (4.2,2.8) (1'2) {$\boxed{1'}\otimes \boxed{2}$};
      \node at (4.2,4.2) (2'2) {$\boxed{2'}\otimes \boxed{2}$};
      \draw[->,thick]  (12.280) -- (22.80) node[midway,right,scale=0.75] {$ 1$};
      \draw[->,thick,dashed,color=darkred]  (12.260) -- (22.100) node[midway,left,scale=0.75] {$\overline 1$};
      \draw[->,thick,dotted,color=blue]  (12) -- (1'2) node[midway,right,scale=0.75] {$ 0$};
      \draw[->,thick]  (11) -- (12) node[midway,above,scale=0.75] {$ 1$};
      \draw[->,thick,dotted,color=blue]  (11) -- (1'1) node[midway,right,scale=0.75] {$ 0$};
      \draw[->,thick,dashed,color=darkred]  (11) -- (21) node[midway,right,scale=0.75] {$\overline 1$};
      \draw[->,thick,dotted,color=blue]  (21) -- (21') node[midway,above,scale=0.75] {$ 0$};
      \draw[->,thick,dotted,color=blue]  (2'1) -- (2'1') node[midway,above,scale=0.75] {$ 0$};      
      \draw[->,thick,dotted,color=blue]  (11') -- (1'1') node[midway,right,scale=0.75] {$ 0$}; 
      \draw[->,thick]  (11') -- (12') node[midway,above,scale=0.75] {$ 1$};
      \draw[->,thick,dashed,color=darkred]  (11') -- (2'1) node[midway,left,scale=0.75] {$\overline 1$\ \ };
      \draw[->,thick,dotted,color=blue]  (12') -- (1'2') node[midway,right,scale=0.75] {$ 0$};
      \draw[->,thick]  (12'.280) -- (22'.80) node[midway,right,scale=0.75] {$ 1$};
      \draw[->,thick,dashed,color=darkred]  (12'.260) -- (22'.100) node[midway,left,scale=0.75] {$\overline 1$};
      \draw[->,thick]  (1'2.80) -- (2'2.280) node[midway,right,scale=0.75] {$ 1$};
      \draw[->,thick,dashed,color=darkred]  (1'2.100) -- (2'2.260) node[midway,left,scale=0.75] {$\overline 1$};
      \draw[->,thick]  (1'2'.80) -- (2'2'.280) node[midway,right,scale=0.75] {$ 1$};
      \draw[->,thick,dashed,color=darkred]  (1'2'.100) -- (2'2'.260) node[midway,left,scale=0.75] {$\overline 1$};
     \draw[->,thick]  (1'1') -- (1'2') node[midway,above,scale=0.75] {$ 1$};
     \draw[->,thick,dashed,color=darkred]  (1'1') -- (2'1') node[midway,right,scale=0.75] {$ \overline 1$};
     \draw[->,thick]  (1'1) -- (1'2) node[midway,above,scale=0.75] {$ 1$};
      \draw[->,thick,dashed,color=darkred]  (1'1) -- (21') node[midway,right,scale=0.75] {\ $\overline 1$};
     \end{tikzpicture}
\ee
 There are two full $\qq_2$-subcrystals $\cB$ and $\cC$, which are isomorphic via the map that exchanges the two elements in the middle of the top row with
  the two elements in the middle of the bottom row,
 and reflects all other elements across the central vertical axis.
 The character of $\BB^+_2$ is the Schur $Q$-polynomial $Q_{(1)}(x_1,x_2) = 2x_1 + 2x_2$
 while $\ch(\cB) = \ch(\cC) =  2x_1^2 + 4x_1x_2  + 2x_2^2 = Q_{(2)}(x_1,x_2)$.
The crystal decomposition $\BB^+_2 \otimes \BB^+_2 = \cB \sqcup \cC$ lifts the Schur $Q$-function identity $Q_{(1)}Q_{(1)} = 2 Q_{(2)}$.
 \end{example}

\begin{remark}
We can use Example~\ref{qq-bb-bb-ex} to explain the origin of the tensor product rules in Theorem~\ref{qq-thmdef}.
The formulas for $e_0(b\otimes c)$ and $f_0(b\otimes c)$ are designed so that if only the $0$-arrows are retained in the crystal graph,
then any normal $\qq_n$-crystal becomes a \defn{$\gl_{1|1}$-crystal} in the sense of \cite[\S2.4]{BKK}; this idea is explained more fully in Remark~\ref{gl11-rmk}.

It remains to motivate the definitions of  $e_{\bar 1}(b\otimes c)$ and $f_{\bar 1}(b\otimes c)$.
Here, we are lead by three principles. First, we want these formulas to agree with the ones in Theorem~\ref{q-tensor-thm}
if $\{b,c\}\subset \BB_n \subset \BB^+_n$.
Second, when $i\neq0$ we want $e_i$ and $f_i$ to commute with the map $\unprime: \BB^+_n \otimes \BB^+_n \to \BB_n \otimes \BB_n$
that removes the primes from each factor of $\boxed{i_1}\otimes \boxed{i_2}$ (and sends $0\mapsto 0$).
Finally, we want the crystal graph of $\BB^+_2\otimes \BB^+_2$ to have two isomorphic connected components.

The first principle requires us to have 
$f_{\bar 1}(\boxed{1}\otimes \boxed{1}) = \boxed{2}\otimes \boxed{1}$ 
and
 $f_{\bar 1}(\boxed{1}\otimes \boxed{2}) = \boxed{2}\otimes \boxed{2}$,
 which already places the 8 elements on the right side of \eqref{bb2-eq} in one connected component.
Since a vertex in the crystal graph can only be the target node of one arrow with a given label, the other principles then force us to have 
 $f_{\bar 1}(\boxed{1'}\otimes \boxed{1}) = \boxed{2}\otimes \boxed{1'}$ 
and
 $f_{\bar 1}(\boxed{1'}\otimes \boxed{2}) = \boxed{2'}\otimes \boxed{2}$.
 The remaining $\bar 1$-arrows in \eqref{bb2-eq} are uniquely determined if we want there to be two isomorphic components.
 Expressing the resulting cases for $e_{\bar 1}(b\otimes c)$ and $f_{\bar 1}(b\otimes c)$, just for $b,c \in \BB^+_2$,
 solely in terms of the weight map and the values of $e_0$, $f_0$, $e_{\bar 1}$, and $f_{\bar 1}$ on each factor
leads to the formulas in Theorem~\ref{qq-thmdef}.
  
%
\end{remark}

The $1$-element $\gl_n$-crystal $\One$ may be regarded as a $\qq_n$-crystal.
Following the conventions in the previous sections, we define a $\qq_n$-crystal to be \defn{normal} 
if each of its full $\qq_n$-subcrystals is isomorphic to a full $\qq_n$-subcrystal of $(\BB^+_n)^{\otimes m}$ for some $m\in \NN$,
where $(\BB^+_n)^0  := \One$.

 
Since there is an isomorphism of $\gl_n$-crystals $\BB_n^+\cong \BB_n \sqcup \BB_n$,
a  normal $\qq_n$-crystal is normal as a $\gl_n$-crystal.
However, normal $\qq_n$-crystals
are not necessarily normal as $\q_n$-crystals. 
This means that results like Theorem~\ref{main-q-thm} do not directly imply 
 similar properties of normal $\qq_n$-crystals. 
 
 \begin{example}\label{abnormal-remark}
The following crystal graph shows a full $\qq_2$-subcrystal  of $\BB_2^+ \otimes \BB_2^+ \otimes \BB_2^+$:
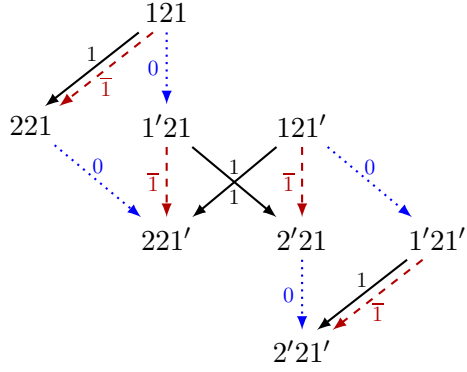
   \begin{center}
    \begin{tikzpicture}[xscale=1.8, yscale=1.5,>=latex]
\node at (0,3) (T121) {$121$};
\node at (-1,2) (T221) {$221$};
\node at (0,2) (T1'21) {$1'21$};
\node at (1,2) (T121') {$121'$};
\node at (0,1) (T221') {$221'$};
\node at (1,1) (T2'21) {$2'21$};
\node at (2,1) (T1'21') {$1'21'$};
\node at (1,0) (T2'21') {$2'21'$};
      \draw[->,thick,dashed,color=darkred]  (T121.235) -- (T221.35) node[midway,below,scale=0.75] {$\overline 1$};
      \draw[->,thick]  (T121.215) -- (T221.55) node[midway,above,scale=0.75] {$1$};
      \draw[->,thick,dotted,color=blue]  (T121) -- (T1'21) node[midway,left,scale=0.75] {$0$};
 \draw[->,thick,dotted,color=blue]  (T2'21) -- (T2'21') node[midway,left,scale=0.75] {$0$};
  \draw[->,thick,dotted,color=blue]  (T121') -- (T1'21') node[midway,above,scale=0.75] {$0$};
 \draw[->,thick]  (T1'21) -- (T2'21) node[midway,above,scale=0.75] {$1$};
\draw[->,thick,dashed,color=darkred]  (T1'21) -- (T221') node[midway,left,scale=0.75] {$\overline 1$};
\draw[->,thick,dashed,color=darkred]  (T121') -- (T2'21) node[midway,left,scale=0.75] {$\overline 1$};
  \draw[->,thick]  (T121') -- (T221') node[midway,below,scale=0.75] {$1$};
 \draw[->,thick,dotted,color=blue]  (T221) -- (T221') node[midway,above,scale=0.75] {$0$};
      \draw[->,thick,dashed,color=darkred]  (T1'21'.235) -- (T2'21'.35) node[midway,below,scale=0.75] {$\overline 1$};
      \draw[->,thick]  (T1'21'.215) -- (T2'21'.55) node[midway,above,scale=0.75] {$1$};
     \end{tikzpicture}
  \end{center}
Here we write $abc$  to denote the element $\boxed{a}\otimes \boxed{b}\otimes\boxed{c}\in \BB_2^+ \otimes \BB_2^+ \otimes \BB_2^+$.
Although this is a normal $\qq_2$-crystal, it is not a normal $\q_2$-crystal.
The elements $1'21$, $121'$, $221'$, and $2'21$ make up a full $\q_2$-subcrystal which
 is not normal, since it has
two $\q_2$-highest weight elements. 
\end{example}

To define highest and lowest weight elements for $\qq_n$-crystals, we need a few more operators.
Assume $\cB$ is a $\qq_n$-crystal. 
For each $i \in [n]$ let 
$ e_{0}^{[i]} : \cB \to \cB \sqcup\{0\}$ and $  f_{0}^{[i]}  : \cB \to \cB \sqcup\{0\}$
be the maps 
\be\label{ef0-eq} e_{0}^{[i]}  := \sigma_{i-1} \cdots \sigma_2 \sigma_1 e_0 \sigma_1 \sigma_2 \cdots \sigma_{i-1}
\quand
f_{0}^{[i]}  := \sigma_{i-1} \cdots \sigma_2 \sigma_1 f_0 \sigma_1 \sigma_2 \cdots \sigma_{i-1}
\ee
where  $\sigma_i$ is defined as in \eqref{weyl-action-eq}.
This means that $e_{0}^{[1]} = e_0$ and $f_{0}^{[1]} = f_0$.
\begin{example}
In Example \ref{qq-bb-bb-ex}, the operator $f_{0}^{[2]}$ acts as
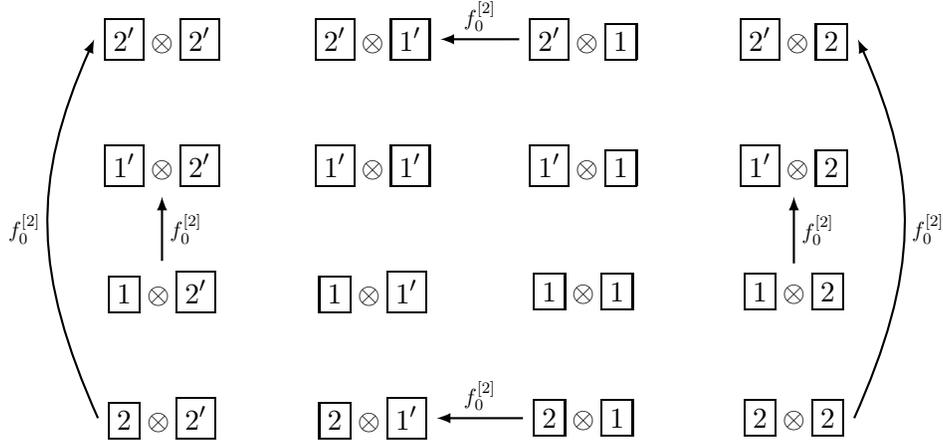
\begin{center}
    \begin{tikzpicture}[xscale=2.0, yscale=1.2,>=latex]
      \node at (0,0) (22') {$\boxed{2}\otimes \boxed{2'}$};
     \node at (0,1.4) (12') {$\boxed{1}\otimes \boxed{2'}$};
      \node at (0,2.8) (1'2') {$\boxed{1'}\otimes \boxed{2'}$};
      \node at (0,4.2) (2'2') {$\boxed{2'}\otimes \boxed{2'}$};
      \node at (1.4,0) (21') {$\boxed{2}\otimes \boxed{1'}$};
     \node at (1.4,1.4) (11') {$\boxed{1}\otimes \boxed{1'}$};
      \node at (1.4,2.8) (1'1') {$\boxed{1'}\otimes \boxed{1'}$};
      \node at (1.4,4.2) (2'1') {$\boxed{2'}\otimes \boxed{1'}$};
      \node at (2.8,0) (21) {$\boxed{2}\otimes \boxed{1}$};
     \node at (2.8,1.4) (11) {$\boxed{1}\otimes \boxed{1}$};
      \node at (2.8,2.8) (1'1) {$\boxed{1'}\otimes \boxed{1}$};
      \node at (2.8,4.2) (2'1) {$\boxed{2'}\otimes \boxed{1}$};
      \node at (4.2,0) (22) {$\boxed{2}\otimes \boxed{2}$};
     \node at (4.2,1.4) (12) {$\boxed{1}\otimes \boxed{2}$};
      \node at (4.2,2.8) (1'2) {$\boxed{1'}\otimes \boxed{2}$};
      \node at (4.2,4.2) (2'2) {$\boxed{2'}\otimes \boxed{2}$};
   \draw[->, thick] (22'.west) to [out=105, in = 255] node[midway,left,scale=0.75] {$ f_0^{[2]}$} (2'2'.west) ;
   \draw[->, thick] (22.east) to [out=75, in = 285] node[midway,right,scale=0.75] {$ f_0^{[2]}$} (2'2.east) ;
      \draw[->,thick]  (12) -- (1'2) node[midway,right,scale=0.75] {$ f_0^{[2]}$};
      \draw[->,thick]  (21) -- (21') node[midway,above,scale=0.75] {$ f_0^{[2]}$};
      \draw[->,thick]  (2'1) -- (2'1') node[midway,above,scale=0.75] {$ f_0^{[2]}$};      
      \draw[->,thick]  (12') -- (1'2') node[midway,right,scale=0.75] {$ f_0^{[2]}$};
     \end{tikzpicture}
  \end{center}
\end{example}

\begin{definition}\label{qq-highest-def}
Suppose $\cB$ is a $\qq_n$-crystal.
A \defn{$\qq_n$-highest weight element}
$b \in \cB$ is a 
$\q_n$-highest weight element
with $e_{0}^{[i]}(b)  = 0$ for all $i \in [n] $.
%
A \defn{$\qq_n$-lowest weight element}
$b \in \cB$ is a 
$\q_n$-lowest weight element
with $f_{0}^{[i]}(b)  = 0$ for all $i\in [n] $.
\end{definition}

 The unique $\qq_n$-highest and $\qq_n$-lowest weight elements in $\BB^+_n$
are $\boxed{1}$ and $\boxed{n'}$.
The $\qq_2$-highest weight elements in the crystal $\BB^+_2\otimes \BB^+_2$ 
from Example~\ref{qq-bb-bb-ex}
are $\boxed{1}\otimes \boxed{1}$
and $\boxed{1}\otimes \boxed{1'}$, while 
the $\qq_2$-lowest weight elements are $\boxed{2'}\otimes \boxed{2}$
and $\boxed{2'}\otimes \boxed{2'}$.

Theorem~\ref{main-thm}, which we prove in Section~\ref{morphisms-sect},
asserts that each connected normal $\qq_n$-crystal has a unique $\qq_n$-lowest weight element. 
The following statement is a corollary of this property.

\begin{corollary}\label{nu-cor}
Let $\cB$ be a normal $\qq_n$-crystal. Then there is a unique map $\nu_0 : \cB \to \NN$
such that $\nu_0(b) = 0$ if $b \in \cB$ is a lowest $\qq_n$-weight element and $\nu_0(e_i(b)) = \nu_0(b) + \delta_{i0}$ 
if 
$e_i(b) \neq 0$.
\end{corollary}

\begin{proof}
It suffices to show that there exists a map $\nu : \cB \to \NN$ with $\nu(e_i(b)) = \nu(b) + \delta_{i0}$ for all $b \in \cB$ and $i \in \{\bar 1,0,1,\dots,n-1\}$
with $e_i(b) \neq 0$.
When  $\cB=\BB^+_n$ this is given by setting $\nu(\boxed{a}) := 1$ and $\nu(\boxed{a'}) := 0$.
If $\cB $ and $\cC$ both have such a map, then $\nu(b\otimes c) := \nu(b) + \nu(c)$
is a map $\nu : \cB \otimes \cC \to \NN$ with the desired property. Therefore such a map $\nu$ exists for any normal $\qq_n$-crystal.
\end{proof}


We can use this result to motivate part of the $\qq_n$-crystal tensor product.

\begin{remark}\label{gl11-rmk}
Assume $\cB$ is a normal $\qq_n$-crystal. 
Let $\ZZ^{1|1} = \ZZ^{2}$ with the bilinear form 
$ \langle v,w\rangle_{1|1} = v_1w_1  - v_{2}w_{2}$
and define $\weight^{1|1} : \cB \to \ZZ^{1|1}$ by
\[
\weight^{1|1}(b) :=  (\nu_0(b)+\varepsilon_0(b)+\varphi_0(b))\e_1  -\nu_0(b) \e_2
\]
where $\nu_0 : \cB \to \NN$ is as in Corollary~\ref{nu-cor}.
For all $b \in \cB$ we have $\varepsilon_0(b)+\varphi_0(b)=\langle \weight^{1|1}(b) , \e_1-\e_2\rangle_{1|1} $,
and if  $e_0(b) = c\neq0$
then $\weight^{1|1}(c) - \weight^{1|1}(b) = \e_1-\e_2$.
These properties mean that $\cB$ is an \defn{(abstract) $\gl_{1|1}$-crystal} relative to $\weight^{1|1}$, $e_0$, $f_0$ in the sense
of \cite[\S2.3]{BKK}. If we ignore  the 
other operators, then Theorem~\ref{qq-thmdef} coincides with the tensor product 
for $\gl_{1|1}$-crystals in \cite[\S2.4]{BKK}.
\end{remark}

\section{Crystal operators on words}\label{words-sect}

It is useful to provide a model for the $\qq_n$-crystals $(\BB^+_n)^{\otimes m}$ when $m \in \NN$.
Define $\cW^+_n(m)$ to be the set of words of length $m$ with letters in $\{1'<1<2'<2<\dots<n'<n\}$,
so that $\cW^+_n(0) = \{\emptyset\}$.
The \defn{weight} of $w \in \cW^+_n(m)$ is  $\weight(w) := (a_1,a_2,\dots,a_n) \in \NN^n$
where $a_i$ is the number of letters of $w$ equal to $i'$ or $i$.
The map $\boxed{w_1} \otimes \boxed{w_2} \otimes \cdots \otimes \boxed{w_m} \mapsto w_1w_2\cdots w_m$
is a weight-preserving bijection $(\BB_n^+)^{\otimes m} \to \cW^+_n(m) $ 
which transfers a $\qq_n$-crystal structure to
$\cW^+_n(m)$. We describe this below.

\subsection{Formulas for crystal operators}

Let $w =w_1w_2\cdots w_m \in \cW^+_n(m)$ and $i \in [n-1]$.
%
Consider the word formed from $w$ by replacing each letter $w_k \in \{i',i\}$ by a right parenthesis ``)''
and each letter $w_j \in \{ i+1,i+1'\}$ by a left parenthesis ``(''.
The \defn{$i$-unpaired indices} in $w$ are the indices $ j \in [m] $
with $w_j \in \{i',i,i+1',i+1\}$
that are not the positions of matching parentheses in this word.
For example,  if $w=131'22'131'2 $ and $i=1$, then the word with parentheses is $)3)(()3)($ 
so the $i$-unpaired indices are $1$, $3$, and $9$.

The following 
description of the lowering and raising operators $f_i$ and $e_i$ for $\cW^+_n(m)$
 is a minor generalization of the signature rule in Remark~\ref{bracketrulegln}.


\begin{definition}
Let $k$ be the last $i$-unpaired index of $w$ with $w_k \in \{i',i\}$.
If no such index exists, then $f_i(w) :=0$. Otherwise form $f_i(w)$ from $w$ by adding $1$ to $w_k$. 
Similarly,
let $j$ be the first $i$-unpaired index of $w$ with $w_j \in \{i+1',i+1\}$.
If no such index exists, then $e_i(w) :=0$. Otherwise form $e_i(w)$ from $w$ by subtracting $1$ from $w_j$.
\end{definition}

For example,
$f_1(13{\color{red}1'}22'131'2 ) =13{\color{red}2'}22'131'2 $ and 
$ e_1(131'22'131'{\color{red}2} )= 131'22'131'{\color{red}1}$.

The definitions of $f_{0}(w)$ and $e_{0}(w)$ are next given as follows:

 \begin{definition}\label{words-ef0-def}
Let $j\in[m]$ be minimal with $w_j \in \{1',1\}$.
If no such $j$ exists or $w_j \neq 1$, then $f_0(w) := 0$. Otherwise $f_0(w)$ is formed from $w$ by changing $w_j$ to $1'$.
Similarly, if no such $j$ exists or $w_j \neq 1'$, then $e_0(w) := 0$. Otherwise $e_0(w)$ is formed by changing $w_j$ to $1$.
\end{definition}

For example $f_0(3{\color{red}1}21'1) = 3{\color{red}1'}21'1$ and $e_0(3{\color{red}1'}21'1) = 3{\color{red}1}21'1$.

The definitions of $f_{\bar 1}(w)$ and $e_{\bar 1}(w)$ 
require a few more cases to state in full.



\begin{definition}
Let $j \in [m]$ be minimal with $w_j\in\{1',1\}$.
If no such $j$ exists or  $w_i\in\{2,2'\}$ for some $i \in [j-1]$, then $f_{\bar 1}(w):=0$.
Otherwise let $k \in [j+1,m]$ be minimal with $w_k \in \{1',1\}$.
\begin{itemize}
\item If such $k$ exists, then $f_{\bar 1}(w)$ is formed from $w$ by changing $w_j$ to $w_k+1$ and $w_k$ to $w_j$.
\item Otherwise,  $f_{\bar 1}(w)$ is formed from $w$ by adding $1$ to $w_j$.

 \end{itemize}
\end{definition}

 Thus $f_{\bar 1}(3{\color{red}1'}21'1) = 3{\color{red}2'}21'1   $ and
$f_{\bar 1}(3{\color{red}1}2{\color{red}1'}1) = 3{\color{red}2'}2{\color{red}1}1 $
and
$f_{\bar 1}(3{\color{red}1'}2{\color{red}1}1) = 3{\color{red}2}2{\color{red}1'}1  $.

\begin{definition}\label{words-ebar-def}
Let $j \in [m]$ be minimal with $w_j\in\{2',2\}$.
If no such $j$ exists or $w_i \in \{1',1\}$ for some $i \in [j-1]$, then $e_{\bar 1}(w) :=0$.
Otherwise let $k \in [j+1,m]$ be minimal with $w_k \in \{1',1\}$.
\begin{itemize}
\item If such $k$ exists, then $e_{\bar 1}(w)$ is formed from $w$ by changing $w_j$ to $w_k$ and $w_k$ to $w_j-1$.
\item Otherwise,  $e_{\bar 1}(w)$ is formed from $w$ by subtracting $1$ from $w_j$.

  \end{itemize}
\end{definition}
 Thus $e_{\bar 1}(3{\color{red}2'}21'1 )= 3{\color{red}1'}21'1 $ and
$e_{\bar 1}( 3{\color{red}2'}2{\color{red}1}1 ) =  3{\color{red}1}2{\color{red}1'}1$
and
$e_{\bar 1}(3{\color{red}2}2{\color{red}1'}1 ) = 3{\color{red}1'}2{\color{red}1}1 $.

Checking the following is straightforward from Theorems~\ref{thmdef1} and \ref{qq-thmdef}:

\begin{proposition}\label{qq-words-prop}
Relative to   these operators  
 $ \cW^+_n(m)$ 
is a $\qq_n$-crystal isomorphic to $(\BB^+_n)^{\otimes m}$. 
\end{proposition}


\begin{corollary}
Suppose $\cB$ is a normal $\qq_n$-crystal and $b,c \in \cB$.
\ben
\item[(a)] If $\weight(b)_1\neq 0$, then $ e_{\bar 1} e_0(b) = e_0e_{\bar 1} (b) $ and $  e_{\bar 1} f_0(b) = f_0e_{\bar 1} (b)$.

\item[(b)] If   $e_{\bar 1}(b)=c$, then $\varepsilon_0(b) \leq \varepsilon_0(c)$ and $\varphi_0(b) \leq \varphi_0(c)$.
\een
\end{corollary}

\begin{proof}
It suffices to prove this when $\cB= \cW^+_n(m)$ for an arbitrary $m\in \NN$.
But in that case both properties are clear from the formulas for $e_{\bar 1}$, $e_0$, and $f_0$
in Definitions~\ref{words-ebar-def} and \ref{words-ef0-def}.
\end{proof}

%

Given $w \in \cW^+_n(m)$, write $\unprime(w)$ for the word formed by removing the primes from all letters.
For each $w \in \cW^+_n(m)$ and $i \in \{\bar 1,1,2,\dots,n-1\}$,
it is clear from the definitions above that
\be\label{first-unprime-eq} e_i ( \unprime(w)) = \unprime ( e_i(w))\quand  f_i ( \unprime(w))= \unprime ( f_i(w))\ee
under the convention that $\unprime(0) := 0$. 
The set $\cW_n(m) := \{ \unprime(w) : w \in \cW^+_n(m)\}$   is therefore a $\q_n$-subcrystal.
This is
 isomorphic to $\BB_n^{\otimes m}$
by 
  \cite[Remarks 2.3 and 2.4]{GHPS}.

If $S \subset [n]$, then   $\{w \in \cW^+_n(m) : w_i\in \ZZ'\text{ if and only if } i \in S\}$
is evidently a $\gl_n$-subcrystal of $\cW^+_n(m)$ and
it follows from \eqref{first-unprime-eq} that $\unprime$ defines a $\gl_n$-crystal isomorphism from this set to
 $\cW_n(m)$. We can reformulate this observation as the following property.
 
 We define a \defn{quasi-isomorphism} between $\gl_n$-, $\q_n$-, or $\qq_n$-crystals, respectively,
to be
a map $\psi : \cB \to \cC$  
such that 
 for each full subcrystal $\cX \subset \cB$, the image $\cY := \psi(\cX)$ is a full subcrystal of $ \cC$
and the restricted map $\psi: \cX \to \cY$ is a crystal isomorphism.

\begin{proposition}
The map $\unprime : \cW^+_n(m) \to \cW_n(m)$ is a quasi-isomorphism of $\gl_n$-crystals.
\end{proposition}

\subsection{Weyl group action}

On normal $\qq_n$-crystals, there is a action of hyperoctahedral group
extending  \eqref{weyl-action-eq}.
This is an interesting feature of $\qq_n$-crystals not present for $\q_n$-crystals.

Suppose $\cB$ is a $\qq_n$-crystal. Then the formula \eqref{weyl-action-eq} for $\sigma_i:\cB\to\cB$
makes sense when $i=0$, and gives
a self-inverse, weight-preserving bijection $\sigma_0 : \cB \to \cB$ satisfying
$ \sigma_0(b) = e_0(b)$ if $e_0(b)\neq 0$,
$ \sigma_0(b) = f_0(b)$ if $f_0(b)\neq 0$,
and
$ \sigma_0(b) = b$  otherwise.
%
Let $\W_n$ denote the group whose elements are the permutations $w$ of $\ZZ$
satisfying $w(-i) = -w(i)$ for all $i \in [n]$ and $w(i) = i$ for all $i>n$. 
This is the finite Coxeter group of type $\mathsf{BC}$ and rank $n$.
Its simple generators 
are given by $t_0 := (-1,1)$ and $t_i := (i,i+1)(-i,-i-1)$ for $i \in [n-1]$.

\begin{theorem}\label{bc-action-prop}
Suppose $\cB$ is a normal $\qq_n$-crystal.
Then there exists a unique action of $\W_n$ on $\cB$
in which $t_0$ and $t_i$ for $i \in [n-1]$
act as the operators $\sigma_0$ and $\sigma_i$, respectively.
\end{theorem}

\begin{proof}
It suffices to check that $\sigma_0,\sigma_1,\sigma_2,\dots, \sigma_{n-1}$ 
satisfy the braid relations for  type $\mathsf{BC}$.
Since $\cB$ is normal as a $\gl_n$-crystal, we already know 
from Remark~\ref{weyl-remark} that 
all relevant (type $\mathsf{A}$) braid relations among $\sigma_1,\sigma_2,\dots, \sigma_{n-1}$ hold.
It is also clear from axioms (Q1) and (Q2) in Definition~\ref{qq-def} that $\sigma_0$ commutes with $\sigma_i$ for all $i \in [2,n-1]$.
Thus it remains only to show that $\sigma_0\sigma_1\sigma_0\sigma_1 = \sigma_1\sigma_0\sigma_1\sigma_0$ as operators $\cB\to \cB$.
For this, we may assume that $\cB = \cW^+_n(m)$ for some $m\geq 0$.

Choose an element $w=w_1w_2\cdots w_m \in \cW^+_n(m)$. The subword of letters in $w$ at $1$-unpaired indices has the form $a_1a_2\cdots a_r b_1b_2\cdots b_s$ where each
$a_i \in \{1',1\}$ and each $b_i \in \{2',2\}$. One can check that $\sigma_1$
acts on $w$ by changing this subword to $c_1 c_2 \dots c_s d_1d_2 \cdots d_r$ where each $c_i \in \{1',1\}$ and
each $d_i \in \{2',2'\}$, with  primed letters occurring at the same locations as in $w$.
On the other hand, $\sigma_0$ acts on $w$ by simply toggling the prime on the first letter equal to $1'$ or $1$, fixing $w$ if there are no such letters.
For example,  $\sigma_1(31'251'22'2) = 31'251'11'2$ and $\sigma_0(31'251'22'2) =31251'22'2$.

If $w$ contains no letters equal to $1'$ or $1$, then 
$\sigma_0\sigma_1\sigma_0\sigma_1(w)$ and $ \sigma_1\sigma_0 \sigma_1\sigma_0(w)$
are both formed from $w$ by toggling the prime on the first letter equal to $2'$ or $2$, if one exists.
Likewise, if $r=s=0$, then $\sigma_0\sigma_1\sigma_0\sigma_1(w) = \sigma_1\sigma_0 \sigma_1\sigma_0(w) = w$.
Assume $w$ contains a letter equal to $1'$ or $1$, and let $j$ be the index of the first such letter.
Also assume $r+s>0$ so that $w$ has at least one $1$-unpaired index, and let $k$ be the first such index.
If $j < k$, then  $\sigma_0\sigma_1(w) = \sigma_1\sigma_0(w)$ so $\sigma_0\sigma_1\sigma_0\sigma_1(w) = \sigma_1\sigma_0 \sigma_1\sigma_0(w) =w$.
The same conclusion holds $j=k$ and $s > 0$. 

We can only have $k<j$ if $r=0$, and in this case $\sigma_0\sigma_1\sigma_0\sigma_1(w) $ and $\sigma_1\sigma_0 \sigma_1\sigma_0(w)$
are both formed from $w$ by toggling the primes on both $w_j$ and $w_k$.
Assume finally that $j=k$ and $s=0$.
If every letter of $w$ equal to $1'$ or $1$ occurs at a $1$-unpaired position, then $\sigma_0\sigma_1\sigma_0\sigma_1(w) = \sigma_1\sigma_0 \sigma_1\sigma_0(w) = \sigma_0(w)$.
Otherwise, let $l$ be the first index of a letter of $w$  equal to $1'$ or $1$ that is not $1$-unpaired.
Then $k<l$, and the words $\sigma_0\sigma_1\sigma_0\sigma_1(w) $ and $\sigma_1\sigma_0 \sigma_1\sigma_0(w)$
are both formed from $w$ by toggling the primes on $w_k$ and $w_l$.
Thus in all cases   $\sigma_0\sigma_1\sigma_0\sigma_1(w) =\sigma_1\sigma_0 \sigma_1\sigma_0(w)$ as needed.
\end{proof}

\section{Crystal operators on increasing factorizations}\label{crystal-incr-sect}

This section describes a $\qq_n$-crystal on factorizations of certain analogues of reduced words for involutions in symmetric groups.
This structure extends a $\q_n$-crystal studied in \cite{Hiroshima2018,Marberg2019b},
which is itself based on a $\gl_n$-crystal described in \cite{MorseSchilling}.

\subsection{Involution words}\label{iw-sect}

Let $S_\ZZ$ be the group of permutations of $\ZZ$ that fix all but finitely many points.
This is a Coxeter group with simple generators $s_i = (i,i+1)$ for $i \in \ZZ$.

 There is a unique associative operation $\circ : S_\ZZ \times S_\ZZ\to S_\ZZ$ 
such that $\pi \circ s_i = \pi$ if $\pi(i) > \pi(i+1)$ and $\pi\circ s_i = \pi s_i$ if $\pi(i)<\pi(i+1)$ for each $i \in \ZZ$ \cite[Thm. 7.1]{Humphreys}. 
A \defn{reduced word} for $\pi \in S_\ZZ$ is an integer sequence $a_1a_2\cdots a_n$ of shortest possible length such that $\pi=s_{a_1}s_{a_2}\cdots s_{a_n}$ (equivalently with $\pi=s_{a_1}\circ s_{a_2}\circ \cdots \circ s_{a_n}$).
 Let 
 $I_\ZZ := \{ \pi \in S_\ZZ : \pi=\pi^{-1}\}$. 
 If $z \in I_\ZZ$ and $i \in \ZZ$, then 
 \be
 s_i \circ z \circ s_i= 
 \begin{cases}  z &\text{if }z(i) > z(i+1) \\
 zs_i = s_iz &\text{if $z(i)=i$ and $z(i+1)=i+1$} \\
 s_izs_i &\text{otherwise}.
 \end{cases}	
 \ee
This implies that $I_\ZZ = \{ \pi^{-1} \circ \pi : \pi \in S_\ZZ\}$, so the following construction exists for any $z \in I_\ZZ$:
 
\begin{definition}\label{inv-def}
An \defn{involution word} for  $z \in I_\ZZ$ is an integer sequence $a_1a_2\cdots a_n$
of shortest possible length with
$z=s_{a_n}\circ \cdots   \circ s_{a_2}\circ s_{a_1} \circ 1 \circ s_{a_1}  \circ s_{a_2} \circ \cdots\circ s_{a_n}$.\footnote{Since $s_i \circ 1 \circ s_i = s_i$, if $n > 0$, then 
$s_{a_n}\circ \cdots   \circ s_{a_2}\circ s_{a_1} \circ 1 \circ s_{a_1}  \circ s_{a_2} \circ \cdots\circ s_{a_n} =
s_{a_n}\circ \cdots   \circ s_{a_2}\circ s_{a_1}    \circ s_{a_2} \circ \cdots\circ s_{a_n}$.}
Write $\iR(z)$ for the set of such words.
\end{definition}

If $z = (1,3)(2,4)$, then
$ \iR(z) = \{ 132,312\}
$
and if $z = (1,4)$, then
$ \iR(z) = \{
123, 231, 213, 321
\}.$
Involution words have been studied before in various forms, for example, in \cite{CJW,HMP2,HanssonHultman,HuZhang1,RichSpring}.

The following generalization of Definition~\ref{inv-def} is considered in \cite{M2021a,M2021b}.
A \defn{commutation} for  $a_1a_2\cdots a_n  \in \iR(z)$ is an index $i \in [n]$ such that
both $a_i$ and $1+a_{i}$ are fixed points of 
 the involution $s_{a_{i-1}}\circ \cdots \circ s_{a_2} \circ s_{a_1} \circ 1 \circ s_{a_1} \circ s_{a_2} \circ\cdots \circ s_{a_{i-1}}.$
 If $n>0$, then $i=1$ is always a commutation.
 
 \begin{definition}\label{inv-def2}
A \defn{primed involution word} for
$z\in I_\ZZ $ is a primed word whose unprimed form is in $\iR(z)$ and whose primed letters occur at commutations.
Write $\iR^+(z)$ for the set of such words.
\end{definition}
For example,  if $z = (1,3)(2,4)$, then $\iR^+(z) = \left\{ 
132, 13'2, 1'32, 1'3'2, 312, 31'2, 3'12, 3'1'2
\right\}$.
The number of commutations is the same for every involution word 
for $z \in I_\ZZ$, and given by
the \defn{absolute length}  $\ell_{\mathsf{abs}}(z) := |\{ i \in \ZZ : i < z(i)\}|$.
Therefore $|\iR^+(z)| = 2^{\ell_{\mathsf{abs}}(z)} |\iR(z)|$.

Definitions~\ref{inv-def} and \ref{inv-def2} can be formulated for arbitrary Coxeter systems (see \cite[\S5]{M2021b}),
but our applications only require the versions for permutations just given.
The next two propositions recall a few special properties of primed involution words 
that will be useful later. 

\begin{proposition}[{\cite[Prop. 8.2]{M2021b}}]
\label{isim-lem2}
 Suppose  $z \in I_\ZZ$ and $X,Y \in \ZZ$. No word in $\iR^+(z)$
contains any of the following as consecutive subwords:
\[  \ba XX,&& X'X, && XX',&& X'X',&& X'(X+1'),&& (X+1')X', \\  XY'X,&& X'Y'X, && X'YX',&& XY'X',&&\text{or}&& X'Y'X'.\ea\]
A word in $\iR^+(z)$ also cannot begin with 
$X(X+1'),$ $ (X+1)X',$ $ XYX,$ $ X'YX,$ or $ XYX'$,
and may only contain $XYX$, $X'YX$, or $XYX'$ as a consecutive (non-initial) subword if $|X-Y| = 1$.
\end{proposition}

Define $\isim$ to be the transitive closure of the symmetric relation on primed words that has 
$ aXYb \isim aYXb
$ for all primed words $a,b$ and $X,Y\in \ZZ\sqcup \ZZ'$
with $|\lceil X\rceil- \lceil Y\rceil|  > 1$, as well as 
$
aXYXb \isim aYXYb $ and $ aX'YXb\isim aYXY'b
$
for all primed words $a,b$ and  $X,Y \in \ZZ$ with $|X-Y|=1$, 
and finally  with
$ Xa \isim X'a$ and $ XYa\isim YXa 
$
for all primed words $a$ and  $X,Y \in \ZZ$.
For example,  
$ 1'232' \isim 1'3'23 \isim 13'23 \isim 3'123 \isim 3123 \isim 1323 \isim 1232
\isim 2132
\isim 2312 
\isim 3212\isim 3121.
$
\begin{proposition}[{\cite[Cor. 8.3]{M2021b}}]
\label{isim-lem}
Each set $\iR^+(z)$ for $z \in I_\ZZ$ is an equivalence class under $\isim$.
\end{proposition}

This generalizes \cite[Thm. 3.1]{HuZhang1}, which describes a similar relation spanning $\iR(z)$.

\subsection{Formulas for crystal operators}\label{incr-f-sect}

Fix an involution  $z \in I_\ZZ$ and define $\ORF_n(z)$ to be the set of tuples
$a=(a^1,a^2,\cdots ,a^n)$ 
where each $a^i$ is a strictly increasing (possibly empty) primed word such that
the concatenation 
$
\concat(a) := a^1a^2\cdots a^n
$ is in $ \iR^+(z)$. 
Define the weight of $a \in \ORF_n(z)$ to be 
\be\label{incr-wt-eq}\weight(a) := (\ell(a^1),\ell(a^2),\dots,\ell(a^n))  \in \NN^n.\ee
We will make $ \ORF_n(z)$ into a $\qq_n$-crystal with this weight map below.

The sequence of 
numbers 
$\hat c_i(z) := | \{  j \in \ZZ :  z(j) < i<j\text{ and } z(i) >  z(j) \}|$ for $i \in \ZZ$
make up the \defn{involution code} of $z$. One has $\hat c_i(z) = 0$ for all but finitely many $i \in \ZZ$.
The transpose of the partition sorting $(\dots,\hat c_1(z), \hat c_2(z), \hat c_3(z),\dots)$
is a strict partition, called the \defn{involution shape} of $z$ in \cite{HMP4} and denoted $\mu(z)$.
For example, if $z = (1,5)(2,3) \in I_\ZZ$, then $\mu(z) = (4,1)$; this also holds if
$z = (k+1,k+5)(k+2,k+3) \in I_\ZZ$ for any $k \in \ZZ$.
The following is useful to note:

\begin{proposition}\label{mu-prop} The set $\ORF_n(z)$ is nonempty if and only if $\mu(z)$ has at most $n$
nonzero parts. 
\end{proposition}

\begin{proof}
It is noted in \cite[Remark 3.16]{Marberg2019b} that $\ORF_n(z)$ is nonempty if and only if 
 $\max_{i \in \ZZ} \hat c_i(z) \leq n$.
\end{proof}

For the rest of this section we assume $\ell(\mu(z)) \leq n$ so that $ \ORF_n(z) \neq \varnothing$ and $\mu(z)\in \NN^n$.
The crystal operators on $ \ORF_n(z)$ are defined in terms of the following pairing procedure:

\begin{definition}\label{pair-def}
Suppose $v=v_1v_2\cdots v_p$ and $w=w_1w_2\cdots w_q$ are strictly increasing words
with letters in $\ZZ \sqcup \ZZ' $.
Form a set of paired letters $\pair(v,w)$ by iterating over the letters in $w$ from largest to smallest; at each iteration,
the current letter $w_j$ is paired with the smallest unpaired letter $v_i$ with $\lceil v_i\rceil >\lceil w_j\rceil $ (if such a letter exists)
and then $(v_i,w_j)$ is added to $\pair(v,w)$.
\end{definition}

If $v=1,3,4,5,8,10',11$ and $w=2',6,9,12,13$, then 
$\pair(v,w) = \{ (10',9), (8,6),(3,2')\}.$
In the sequence of definitions below, we 
fix $i \in [n-1]$ and $a=(a^1,a^2,\dots,a^n) \in \ORF_n(z)$.

\begin{definition}
\label{orf-def}
If every letter in $a^i$ is the first term of an element of $ \pair(a^i,a^{i+1})$,
then $f_i(a) := 0$.
Otherwise, let $x \in \ZZ\sqcup \ZZ'$ be the largest unpaired letter in $a^i$,  
let $y  \in \ZZ$ be the smallest integer not in $\unprime(a^{i+1})$ 
with $y\geq \lceil x\rceil $, and construct an $n$-tuple of strictly increasing words $f_i(a)$
by applying the following procedure to $a$:

\begin{itemize}
\item[(L1)] If  $x \in \ZZ'$,
then remove $x$ from $a^i$ and add $y'$ to $a^{i+1}$:
\[ 
 a=(\dots,1{\color{red}3'}459,347',\dots) \mapsto (\dots,1459,34{\color{red}5'}7',\dots)=f_i(a).\]

\item[(L2)]  If $x \in \ZZ$,
then remove $x$ from $a^i$ and add $y$ to $a^{i+1}$.
Then, for each integer $v \in [x,y-1]$
with $v+1 \in a^i$ and $v' \in a^{i+1}$, replace $v+1 \in a^i$ by $v+1'$ and $v' \in a^{i+1}$ by $v$:
\[ 
 a=(\dots,1{\color{red}3}4{\color{red}5}69,3{\color{red}4'}58,\dots) \mapsto (\dots,14{\color{red}5'}69,3{\color{red}4}5{\color{red}6}8,\dots)=f_i(a).\]
\end{itemize}
\end{definition}

\begin{definition}
\label{ore-def}
If every letter in $a^{i+1}$ is the second term of an element of $ \pair(a^i,a^{i+1})$,
then $e_i(a) := 0$.
Otherwise, let $y \in \ZZ\sqcup \ZZ'$ be the smallest unpaired letter in $a^{i+1}$,  
let $x  \in \ZZ$ be the largest integer not in $\unprime(a^{i})$ 
with $x\leq \lceil y\rceil $, and construct an $n$-tuple of strictly increasing words $e_i(a)$
by applying the following procedure to $a$:
\begin{itemize}
\item[(R1)]   If  $y \in \ZZ'$,
then remove $y$ from $a^{i+1}$ and add $x'$ to $a^{i}$:
\[ 
 a= (\dots,1459,34{\color{red}5'}7',\dots)\mapsto (\dots,1{\color{red}3'}459,347',\dots) =e_i(a).\]

\item[(R2)]  If $y \in \ZZ$,
then remove $y$ from $a^{i+1}$ and add $x$ to $a^{i}$.
Then, for each integer $v \in [x,y-1]$
with $v+1' \in a^i$ and $v \in a^{i+1}$, replace $v+1' \in a^i$ by $v+1$ and $v \in a^{i+1}$ by $v'$:
\[ 
 a= (\dots,14{\color{red}5'}69,3{\color{red}4}5{\color{red}6}8,\dots)\mapsto(\dots,1{\color{red}3}4{\color{red}5}69,3{\color{red}4'}58,\dots) =e_i(a).\]
\end{itemize}
\end{definition}


In Definitions~\ref{orf-def2} and \ref{ore-def2} we assume that $n\geq 2$.
\begin{definition}
\label{orf-def2}
If $a^1$ is empty or if the first letter of $a^1$ is not strictly smaller than every letter in $a^2$,
then  $f_{\bar 1}(a) := 0$.
If $a^1$ has at least two letters and the first two of these are not both primed or unprimed,
then reverse the primes on these letters and move the modified first letter of $a^1$ to the start of $a^2$.
 %
Otherwise, move the first letter of $a^1$ to the start of $a^2$.
\end{definition}

For example
$f_{\bar 1} ({\color{red}1'3}4,25,\dots) =  ({\color{red}3'}4,{\color{red}1}25,\dots)$
and
$f_{\bar 1} ({\color{red}1'}3'4,25,\dots)=  (3'4,{\color{red}1'}25,\dots)  $.

\begin{definition}
\label{ore-def2}
If $a^2$ is empty or if the first letter of $a^2$ is not strictly smaller than every letter in $a^1$,
then  $e_{\bar 1}(a) := 0$.
 If $a^1$ is nonempty and the first letters of $a^1$ and $a^2$ are not both primed or unprimed,
then reverse the primes on these letters and move the modified first letter of $a^2$ to the start of $a^1$.
%
Otherwise, move the first letter of $a^2$ to the start of $a^1$.
%
\end{definition}

For example
$e_{\bar 1} ({\color{red}3'}4,{\color{red}1}25,\dots) = ({\color{red}1'3}4,25,\dots) $
and
$e_{\bar 1} (3'4,{\color{red}1'}25,\dots) = ({\color{red}1'}3'4,25,\dots) $.

\begin{definition}
\label{orf-def3}
If $a^1$ is empty or begins with a primed letter, then $f_{0}(a):=0$.
Otherwise,  form $f_0(a)$ from $a$ by adding a prime to the first letter of $a^1$.
Similarly, if
 $a^1$ is empty or begins with an unprimed letter, then $e_{0}(a):=0$.
Otherwise,  form $e_0(a)$ from $a$ by removing the prime from the first letter of $a^1$.
Thus 
 $f_0({\color{red}1}3'4,25,\dots) = ({\color{red}1'}3'4,25,\dots) $
 and $ e_0({\color{red}1'}3'4,25,\dots) = ({\color{red}1}3'4,25,\dots) $.
\end{definition}


Given a factorization $a \in \ORF_n(z)$  let 
\be \unprime(a) := (\unprime(a^1), \unprime(a^2),\dots, \unprime(a^n)).\ee
Denote the set of increasing factorizations of (unprimed) involution words for $z$ by
 \be\label{incr-eq} \Incr_n(z) := \{ a \in \ORF_n(z) : a = \unprime(a) \} = \{ \unprime(a) : a \in \ORF_n(z)\}.\ee
Restricted to this set, the operators $e_i$ and $f_i$ defined above
for $i \in \{\bar 1,1,2,\dots,n-1\}$
 coincide with the ones in \cite[Thm. 3.1]{Hiroshima2018} 
and make $ \Incr_n(z)$
into a $\q_n$-crystal, 
which is normal by \cite[Cor. 3.33]{Marberg2019b}.
Our ultimate goal is to show that the larger set $\ORF_n(z)$ is likewise a normal $\qq_n$-crystal;  this will eventually be stated as Corollary~\ref{orf-upgrade}.
In this section we only prove one part of this claim:

\begin{proposition}\label{qq-orf-prop}
Assume $z \in I_\ZZ$ has $\ell(\mu(z)) \leq n$. 
Relative to the operators 
defined
above
 and  the weight map \eqref{incr-wt-eq},
the nonempty set 
$ \ORF_n(z)$ 
is a $\qq_n$-crystal.
\end{proposition}

For an example of the crystal $\ORF_n(z)$, see Figure~\ref{incr-fig}.
Our proof of Proposition~\ref{qq-orf-prop} is at the end of this section, following several lemmas. The first of these is clear from the definitions:

\begin{lemma}\label{ef-ick-lem}
If $a \in \ORF_n(z)$ and $i \in [n-1]$,
then it holds that
$e_i ( \unprime(a)) = \unprime ( e_i(a))$ and $ f_i ( \unprime(a))= \unprime ( f_i(a))$
under the convention that $\unprime(0) := 0$. 
When $n\geq 2$ the same identities hold  for $i=\bar 1$.
\end{lemma}

Since $\Incr_n(z)$ is a $\q_n$-crystal, this lemma mostly
reduces the proof of Proposition~\ref{qq-orf-prop} to showing that 
 $e_i$ and $f_i$  
are well-defined maps $\ORF_n(z) \to \ORF_n(z)\sqcup\{0\}$.
This is nontrivial because these operators can change the locations of the primed letters in a factorization.

In Lemmas~\ref{L5.3}, \ref{L5.4}, \ref{L5.6}, and \ref{L5.7}, we assume $a$ and $b$ are strictly increasing words with no primed letters
such that the concatenation $ab$ is a reduced word for some element of $S_\ZZ$.
If $x$ is a letter in $a$ that is not the first term of a pair in $\pair(a,b)$,
then we say that $x$ is \defn{unpaired} in $a$.
If $y$ is a letter in $b$ that is not the second term of a pair in $\pair(a,b)$,
then we say that $y$ is \defn{unpaired} in $b$.

\begin{lemma}
\label{L5.3}
Suppose $x$ is the last unpaired letter in $a$.
Then $x-1 \notin b$ and there exists $q \in \NN$ such that $x+i \in a$ and $x-1+i \in b$ for all $i \in [q]$
while $x+q+1 \notin a$ and $x+q \notin b$.
\end{lemma}

This result is essentially \cite[Lem. 10.4]{BumpSchilling}, but since our notational conventions are 
quite different, we include a direct proof for completeness.

\begin{proof}

If we had $x-1 \in b$, then we would have $(x,x-1) \in \pair(a,b)$, contradicting  the assumption that $x$ is unpaired.
Let $q \in \NN$ be maximal such that $x+q \in a$. It suffices to show that $x+i \in b$ for $0\leq i <q$ but $x+q \notin b$.

If $x+1 \in a$, then $(x+1,y) \in \pair(a,b)$ for some $y \in b$, and since $x \in a$ is unpaired we must have $y=x$.
If we also have $x+2 \in a$, then
   $(x+1,y) \in \pair(a,b)$ for some $y \in b$, and as $x \in a$ is unpaired it must hold that $y=x+1$.
Repeating this argument shows that $x-1+i \in b$ for all $i \in [q]$. 
 
Finally suppose  $x+q \in b$. The letters in $ab$ that are between the subword $ x(x+1)(x+2)\cdots(x+q) $ in $a$ 
and the subword $ (x+1)(x+2)\cdots (x+q)$ in $b$ are each either greater than $x+q+1$ or less than $x - 1$.
Therefore $ab$ belongs to the same commutation class as a word containing
the consecutive subword
$ x(x+1)(x+2)\cdots(x+q) (x+1)(x+2)\cdots (x+q)$. But it is straightforward to check that the latter word is not reduced.
As $ab$ is a reduced word, we must have $x+q \notin b$.
\end{proof}

The proof of the following complementary result is similar. We omit the details.

\begin{lemma}
\label{L5.4}
Suppose $y$ is the first unpaired letter in $b$.
Then $y+1 \notin a$ and there exists $q \in \NN$ such that $y+1-i \in a$ and $y-i \in b$ for all $i \in [q]$
while $y-q \notin a$ and $y-q-1 \notin b$.
\end{lemma}

We continue to let $a$ and $b$ be strictly increasing words with no primed letters.

\begin{lemma}
\label{L5.6}
Assume that  $ab$ is a consecutive subword of an involution word $w$ for some  $z \in I_\ZZ$.
Suppose there is an unpaired letter in $a$ and let $x$ be the last such letter.
Let $q \in \NN$ be maximal such that $x+q \in a$. 
Let $\hat a$ be the subword of $a$ formed by removing $x$,
 let $\hat b$ be the strictly increasing word formed by adding $x+q$ to $b$,
and  let $\hat w$ be the word formed from $w$ by replacing the subword $ab$ by $\hat a \hat b$.
Define $i$ and $j$ to be the respective indices of $x \in a$ in $w$ and $x \in \hat b$ in $\hat w$.\footnote{
Note by Lemma~\ref{L5.3} that if $q=0$, then $x\notin b$ and if $q>0$, then the index of $x \in b$ in $w$ is $j+1$.}
Then: 
\ben
\item[(a)] No index in $\{i+1,i+2,\dots,i+q\}$ is a commutation in $w$.

\item[(b)] At most one index in $\{i\} \sqcup\{j+1,j+2,\dots,j+q\}$ is a commutation in $w$.
\een
 \ben
\item[(c)] One has $\hat w \in \iR(z)$, and if $i$ is a commutation in $w$,
then $j+q$ is a commutation in $\hat w$.

\item[(d)] If $p\in[q]$ and $j+p$ is a commutation in $w$, 
then $i-1+p$ is a commutation in $\hat w$.

\een

\end{lemma}

\begin{proof} 
To refer to primed numbers 
we introduce the notation $h^{\alpha} := h-\alpha/2$ for $h \in \ZZ$ and $\alpha \in \{0,1\}$.
Observe that in this notation,
the relation $\isim$ from Proposition~\ref{isim-lem}
gives $\cdots 1^\alpha 3^\beta\cdots \isim \cdots 3^\beta 1^\alpha\cdots$ and $\cdots 1^\alpha 2^\beta 1^\gamma\cdots \isim \cdots 2^\gamma 1^\beta 2^\alpha\cdots$
for example.

Without loss of generality we may assume that $x=1$.
Suppose $W$ is a primed involution word for $z$  such that $\unprime(W) = w$.
Let $A$ and $B$ be the (primed) consecutive subwords of $W$ 
corresponding to $a$ and $b$.
We know from Lemma~\ref{L5.3} that $A$ has a subword of the form 
\be\label{subword1}
1^{\alpha_1}2^{\alpha_2}3^{\alpha_3}\cdots (q+1)^{\alpha_{q+1}}
\ee
for some choice of $\alpha_1,\alpha_2,\dots,\alpha_{q+1}   \in \{0,1\}$
while $B$ has a subword of the form
\be\label{subword2}
1^{\beta_1}2^{\beta_2}\cdots q^{\beta_q}
\ee
for some choice of $\beta_1,\beta_2,\dots,\beta_q  \in \{0,1\}.$
Note that $i$ is a commutation in $w$ if and only if $\alpha_1 \neq 0$ for some choice of $W$,
with similar observations applying to  other indices.
Form $\hat A$ from $A$ and $\hat B$ from $B$ by changing
the subwords \eqref{subword1} and \eqref{subword2} to 
$
2^{\beta_1}3^{\beta_2}4^{\beta_3} \cdots(q+1)^{\beta_q} 
$
and
$ 1^{\alpha_2}2^{\alpha_3}3^{\alpha_4}\cdots q^{\alpha_{q+1}}(q+1)^{\alpha_1}
$
respectively.
Then form $\hat W$ from $W$ 
by replacing $AB$ by $\hat A \hat B$. Observe that $\unprime(\hat W) = \hat w$.

We first show that $W$ and $\hat W$ are equivalent under  $\isim$ from Proposition~\ref{isim-lem}. 
This will suffice to prove parts (c) and (d).
Let $L$ be the part of $A$ that comes after \eqref{subword1} and let $M$ be the part of $B$ that comes before \eqref{subword2}, and define $K$ and $N$ to be the primed words such that 
\[ W =  K \cdot 1^{\alpha_1}2^{\alpha_2}3^{\alpha_3}\cdots (q+1)^{\alpha_{q+1}}
\cdot L\cdot  M\cdot 
1^{\beta_1}2^{\beta_2}\cdots q^{\beta_q}\cdot N.\]
Then we similarly have 
$\hat W =  K \cdot 2^{\beta_1}3^{\beta_2}4^{\beta_3} \cdots(q+1)^{\beta_q} 
\cdot L\cdot  M\cdot 
1^{\alpha_2}2^{\alpha_3}3^{\alpha_4}\cdots q^{\alpha_{q+1}}(q+1)^{\alpha_1}\cdot N$
and it follows from Lemma~\ref{L5.3} that 
\[ \ba
W &\isim  K\cdot L \cdot 1^{\alpha_1}2^{\alpha_2}3^{\alpha_3}\cdots (q+1)^{\alpha_{q+1}} 
1^{\beta_1}2^{\beta_2}\cdots q^{\beta_q}\cdot M\cdot  N,
\\
 \hat W &\isim  K\cdot L \cdot 2^{\beta_1}3^{\beta_2}4^{\beta_3} \cdots(q+1)^{\beta_q} 
1^{\alpha_2}2^{\alpha_3}3^{\alpha_4}\cdots q^{\alpha_{q+1}}(q+1)^{\alpha_1}\cdot M\cdot  N.
\ea\]
It is easy to see that the words on the right can be transformed by a sequence of commutation relations of the form  $\cdots X^\alpha Y^\beta \cdots \isim \cdots Y^\beta X^\alpha \cdots$
to 
\be\label{ssubword1}
K\cdot L \cdot 1^{\alpha_1}2^{\alpha_2}1^{\beta_1}3^{\alpha_3}2^{\beta_2} 4^{\alpha_4} 3^{\beta_3} \cdots (q+1)^{\alpha_{q+1}}q^{\beta_q}\cdot M\cdot  N
\ee
and
\[
 K\cdot L \cdot 2^{\beta_1}1^{\alpha_2}3^{\beta_2} 2^{\alpha_3}4^{\beta_3}3^{\alpha_4}  \dots (q+1)^{\beta_q} q^{\alpha_{q+1}}(q+1)^{\alpha_1}\cdot M\cdot  N
\]
respectively. The first of these becomes the second
after applying a sequence of braid relations of the form
$\cdots X^\alpha Y^\beta X^\gamma\cdots \isim \cdots Y^\gamma X^\beta Y^\alpha\cdots$
so we have $W \isim \hat W$ as desired.

As mentioned above, this fact implies parts (c) and (d).
For parts (a) and (b), we note that for each $k\in [q]$, the word obtained after applying $k-1$ braid relations 
to \eqref{ssubword1}
has the form
\[
K\cdot L \cdot 2^{\beta_1} 1^{\alpha_2} 
 \cdots k ^{\beta_{k-1}}(k-1)^{\alpha_{k}} \cdot k^{\alpha_{1}} \cdot (k+1)^{\alpha_{k+1}} k^{\beta_{k}}\cdots (q+1)^{\alpha_{q+1}}q^{\beta_q}\cdot M\cdot  N
.\] This is a primed involution word by Proposition~\ref{isim-lem},
so to avoid the patterns forbidden in Proposition \ref{isim-lem2} we must have $\alpha_{k+1} = 0$
and $\alpha_1 + \beta_k \in \{ 0, 1\}$. As this applies to all $k \in [q]$ and all primed involution words $W$ with $w=\unprime(W)$,
we deduce that none of the indices $i+1,i+2,\dots,i+q$ are commutations in $w$
and that if $i$ is a commutation in $w$, then none of $ j+1,j+2,\dots,j+q$ is also a commutation.

The last thing to check is that at most one of  $ j+1,j+2,\dots,j+q$ is a commutation in $w$.
Writing $w_k$ for the $k$th letter of $w$, define
 $\delta_k(w) := s_{w_k}\circ \cdots \circ s_{w_2} \circ s_{w_1}  \circ 1\circ s_{w_1} \circ s_{w_2} \circ \cdots \circ s_{w_k}$.
 Suppose $q>0$ and $j+k$ is a commutation in $w$ for some minimal $k \in [q]$.
 Since we assume $x=1$, we have $w_{j+p} = p$ for all $p \in [q]$, so this means that $k$ and $k+1$ are both fixed points of $\delta_{j+k-1}(w)$. It is easy to check by induction that $(k,p+1)$ is then a cycle of $\delta_{j+p}(w)$ for each $p \in [k,q]$ and $j+p$ is not a commutation for any $p \in [k+1,q]$. This completes the proof of parts (a) and (b).
\end{proof}

There is again a complementary result with a symmetric proof,
whose details we omit.

\begin{lemma}
\label{L5.7}
Assume that  $ab$ is a consecutive subword of an involution word $w$ for some  $z \in I_\ZZ$.
Suppose there is an unpaired letter in $b$ and let $y$ be the first such letter.
Let $q \in \NN$ be maximal such that $y-q \in b$.
Let $\hat b$ be the subword of $b$ formed by removing $y$,
 let $\hat a$ be the strictly increasing word formed by adding $y-q$ to $a$,
 and let $\hat w$ be the word formed from $w$ by replacing the subword $ab$ by $\hat a \hat b$.
Define $i$ and $j$ to be the respective indices of $y \in \hat a$ in $\hat w$ and $y \in b$ in $ w$.\footnote{
Note by Lemma~\ref{L5.4} that if $q=0$, then $y\notin a$ and if $q>0$, then the index of $y \in b$ in $w$ is $i-1$.}
Then:
\ben
\item[(a)] No index in $\{ j-1,j-2,\dots,j-q\}$ is a commutation in $w$.

\item[(b)] At most one index in $\{i-1,i-2,\dots,i-q \}\sqcup \{j\}$ is a commutation in $w$.
\een
\ben
\item[(c)] One has $\hat w \in \iR(z)$, and if $j$ is a commutation in $w$, then $i-q$ is a commutation in $\hat w$.

\item[(d)] If $p \in [q]$ and  $i-p$ is a commutation in $w$, 
then $j+1-p$ is a commutation in $\hat w$.
\een
\end{lemma}

We may now prove Proposition~\ref{qq-orf-prop}.

\begin{proof}[Proof of Proposition~\ref{qq-orf-prop}]
First let $i \in [n-1]$. Everything that needs to be checked to conclude that 
the operator $f_i$ from Definition~\ref{orf-def} (respectively, $e_i$ from Definition~\ref{ore-def})
is a well-defined map $\ORF_n(z) \to \ORF_n(z)\sqcup \{0\}$ is immediate from 
Lemma~\ref{L5.6} (respectively, Lemma~\ref{L5.7}).

Now suppose $b,c \in \ORF_n(z)$. 
If $e_i(b)=c$, then 
$e_i(\unprime(b)) = \unprime(c)$ by  Lemma~\ref{ef-ick-lem} so
$f_i(\unprime(c)) = \unprime(b)$ since $\Incr_n(z)$ is a $\q_n$-crystal.
For this to hold, the last unpaired letter in $\unprime(c^i)$ must be the unique letter
not also present in $ \unprime(b^i)$, and given this observation it is clear from  
Definition~\ref{ore-def} that   $f_i(c) = b$. 
If $f_i(c) = b$, then it follows similarly that $e_i(b)=c$.
This confirms axiom (S1) in Definition~\ref{crystal-def}.
Since $\unprime : \ORF_n(z)\to\Incr_n(z)$ is a weight-preserving map, 
axiom (S2) holds by Lemma~\ref{ef-ick-lem}, so $ \ORF_n(z)$ is a $\gl_n$-crystal.

 If $n\geq 2$ and the words $b^1$ and $b^2$ are both nonempty,
 then  $\unprime(\min(b^1)) \neq \unprime(\min(b^2))$ since 
 otherwise $\unprime(\concat(b)) \in \iR(z)$ would be equivalent under $\isim$
 to a word starting with $XYX$ for some $X,Y \in \ZZ$, contradicting Propositions~\ref{isim-lem2} and \ref{isim-lem}. 
 Once we note this, checking that the operators $e_{\bar 1}$ and $f_{\bar 1}$
 in Definitions~\ref{orf-def2} and \ref{ore-def2} are 
 well-defined maps
 $\ORF_n(z) \to \ORF_n(z)\sqcup \{0\}$ satisfying $e_{\bar 1}(b) = c$ if and only if $f_{\bar 1}(c)=b$
 is straightforward using Propositions~\ref{isim-lem2} and \ref{isim-lem}.
 
 The other conditions in Definition~\ref{q-def} are clear, so  $ \ORF_n(z)$ is a $\q_n$-crystal.
The axioms in Definition~\ref{qq-def} are also mostly self-evident for $\ORF_n(z)$.
The only relevant property  that is not completely trivial from the definitions is the claim that $e_0$ and $f_0$ preserve
the string lengths $\varepsilon_{i}$ and $\varphi_{i}$ for $i \in \{ \bar 1 , 1\}$, 
but this follows from Lemma~\ref{ef-ick-lem}.
Thus $ \ORF_n(z)$ is a $\qq_n$-crystal.
\end{proof}

\begin{figure}[h]
\begin{center}
\input{qq3-incr-crystal.tex}
\end{center}
\caption{Crystal graph of $\qq_3$-crystal $\ORF_3(z)$ for $z=(1,3)(2,4)\in I_\ZZ$.
In this picture we draw styled edges without labels for clarity. Solid blue and red arrows are edges 
$b \xrightarrow{\ 1\ } c$ and $b \xrightarrow{\ 2\ } c$, respectively. Dotted green and dashed blue arrows 
are edges $b \xrightarrow{\ 0\ } c$ and $b \xrightarrow{\ \bar 1\ } c$, respectively.
}
\label{incr-fig}
\end{figure}

\subsection{Coxeter-Knuth operators}\label{ck-sect}

Continue to fix an element $z \in I_\ZZ$.
In this section we prove some additional facts about the crystal operators on $\ORF_n(z)$.
These properties will be used in Section~\ref{morphisms-sect}.

A set of integers $\{j,k\}$ is a \defn{cycle} of $z$ if $j\neq z(j)=k$.
If $i \in [m]$ is a commutation for   $w=w_1w_2\cdots w_m \in \iR(z)$,
then the unordered pair
$ \gamma_i(w) := s_{w_m}\cdots s_{w_{i+2}} s_{w_{i+1}}(\{ w_i, 1+w_{i}\}) $
is a cycle of $z$,
and the map $i \mapsto \gamma_i(w)$ is a bijection from the set of commutations $i \in [m]$
for $w$ to the set of cycles of $z$.
Each $w \in \iR^+(z)$ determines a corresponding set of \defn{marked cycles} 
\[ \marked(w) := \{ \gamma_i(w) : i \in [\ell(w)]\text{ and } w_i \in \ZZ'\}
\quad\text{where $\gamma_i(w) := \gamma_i(\unprime(w))$}.\]
If $z = (1,6)(2,5)(3,4)$, then $w=5'13'243541 \in \iR^+(z)$ and $\marked(w) = \{ \{3,4\}, \{1,6\}\}$, for example.
For $a \in \Incr^+_n(z)$ we define $\marked(a) := \marked(\concat(a))$.

 \begin{lemma}\label{marked-lem0}
 If $a,b \in \Incr^+_n(z)$,
  $\unprime(a) = \unprime(b)$, and $\marked(a) = \marked(b)$, then $a=b$.
 \end{lemma}
 
 \begin{proof}
If $a,b \in \Incr^+_n(z)$ and
  $\unprime(a) = \unprime(b)$, then $a=b$ if and only if the primes indices in $\concat(a)$ and $\concat(b)$
are the same, which happens precisely when $\marked(a) = \marked(b)$.
 \end{proof}
 
Let $\ck$ denote the operator that acts on 1- and 2-letter primed words 
by interchanging  
\[
X \leftrightarrow X',
\quad
XY \leftrightarrow YX,\quad 
X'Y \leftrightarrow Y'X,\quad 
XY' \leftrightarrow YX',\quand 
X'Y' \leftrightarrow Y'X'
\]
for all $X,Y \in \ZZ$. In addition, let $\ck$ act on 3-letter primed words as the involution interchanging
\[
XYX \leftrightarrow YXY,\quad 
X'YX \leftrightarrow YXY',\quad
ACB \leftrightarrow CAB,\quand 
BCA \leftrightarrow BAC
\]
for all $X,Y \in \ZZ$ and all $A,B,C \in \ZZ\sqcup \ZZ'$ with $\lceil A \rceil < \lceil B \rceil  <\lceil C \rceil $,
while fixing 
any 3-letter words not of these forms.
Now, given a primed word $w=w_1w_2w_3\cdots w_m$ and $i \in[m-2]$, we define
\[ 
\ba
\ck_{-1}(w) &:=\ck(w_1) w_2w_3\cdots w_m, \\
\ck_0(w) &:= \ck(w_1w_2)w_3\cdots w_m, \\
\ck_i(w) &:= w_1\cdots w_{i-1} \ck(w_iw_{i+1}w_{i+2}) w_{i+3}\cdots w_m,
\ea
\]
while setting $\ck_i(w):=w$ for $i \in \ZZ$ with $i+2 \notin[m]$.
For example, if $w = 45'7121'$, then
\[
\barr{rrr}
 \ck_{-1}(w) = 4'5'7121',\quad
&
\ck_{0}(w) = 54'7121',\quad
&
\ck_{1}(w) = 45'7121',
\\
\ck_{2}(w) = 45'1721',\quad
&
\ck_{3}(w) = 45'1721',\quad
&
\ck_{4}(w) = 45'72'12.
\earr \]
We call $\ck_i$ an \defn{(orthogonal) Coxeter-Knuth operator}.
This terminology comes from \cite{M2021a}, where these operators are related to a map called
\defn{orthogonal Edelman-Greene insertion}; see Section~\ref{morphisms-subsect2}.

\begin{lemma}\label{marked-lem}
If $w \in \iR^+(z)$
 and $i>0$, then $\marked(\ck_i(w)) = \marked(w)$.
 \end{lemma}
 
 \begin{proof}
Fix $w \in \iR^+(z)$ and $i>0$, and suppose $v:=\ck_i(w) \neq w$. Set $\gamma_j(w) := \emptyset$ if $j$ is not a commutation of $\unprime(w)$.
If $v_i=w_i$, then it is easy to check that $\gamma_{i+1}(v) = \gamma_{i+2}(w)$, $\gamma_{i+2}(v) = \gamma_{i+1}(w)$,
and $\gamma_j(v) =\gamma_j(w)$ for all $j \notin \{i+1,i+2\}$, so $\marked(v) = \marked(w)$.
The identity $\marked(v) = \marked(w)$ follows by a similar argument when $v_{i+2}=w_{i+2}$.

Assume $v_i \neq w_i$ and $v_{i+2} \neq w_{i+2}$.
Then $w_iw_{i+1}w_{i+2}$ must have the form $XYX$, $X'YX$, or $YXY'$ where $X\in \ZZ$ and $Y = X\pm 1$,
so
$\ck_i$ applied to $w$ acts on this subword as the relation $XYX \leftrightarrow YXY$ or $X'YX \leftrightarrow YXY'$.
Proposition~\ref{isim-lem2} implies that $\gamma_{i+1}(v) = \gamma_{i+1}(w) = \emptyset$ 
and that  $i$ and $i+2$ are not both commutations in $v$ or $w$.
To show that $\marked(v) = \marked(w)$ it remains to check that $\gamma_i(v) = \gamma_{i+2}(w)$
and $\gamma_{i+2}(v) = \gamma_{i}(w)$,
and this is straightforward.
 \end{proof}

Lemmas~\ref{ock-key-lem} and \ref{f0-incr-lem} are key technical results that will be needed in Section~\ref{morphisms-subsect2}.

\begin{lemma}\label{ock-key-lem}
Let $w=(w^1,w^2,\dots,w^n) \in \Incr^+_n(z)$ and $k \in[n-1]$.
Suppose $f_k(w) \neq 0$. Define 
\[ M := \ell(w^1) +\ell(w^2)+ \dots +\ell(w^{k-1}) + 1\quand N := \ell(w^1) + \ell(w^2) + \dots + \ell(w^{k+1}) .\]
Then there are indices $j_1,\dots,j_l \in [M,N-2]$ with 
$
\concat(f_k(w)) =  \ck_{j_l}  \cdots \ck_{j_1}(\concat(w)).
$
\end{lemma}

\begin{proof}
We set up our notation as in Lemma~\ref{L5.6} with $a:=\unprime(w^k)$ and $b :=\unprime(w^{k+1})$. Let $x$ be the last unpaired letter in $a$ and 
suppose $q\in \mathbb{N}$ is maximal such that $x+q \in a$. 
 Let $r$ be the number of letters greater than $x+q$ in $a$, and let $s$ be the number of letters less than $x$ in $b$. 
 
 Write $\sim_{MN}$ for the transitive closure  of the relation on primed words that has $v \sim_{MN} \ck_i(v)$ if $i \in [M,N-2]$.
The lemma is equivalent to the claim that $\concat(w) \sim_{MN} \concat(f_k(w))$.
We will prove this by induction on $r+s$.

Without loss of generality we may assume that $x=1$. 
As above, we  write elements of $\ZZ\sqcup\ZZ'$ in the form $h^\alpha := h - \frac{\alpha}{2}$ for $h \in \ZZ$ and $\alpha \in \{0,1\}$.
Then Lemma~\ref{L5.6}
implies that $ w^kw^{k+1}$ contains a consecutive subword of the form
\be\label{word1} 1^{\alpha_0}2 \hs 3 \cdots (q+1) (\dash r \dash )(\dash s\dash ) 1^{\alpha_1}2^{\alpha_2}3^{\alpha_3}\cdots q^{\alpha_q}\ee
where each $\alpha_i \in \{0,1\}$
and
where  $(\dash r\dash)$ and $(\dash s\dash)$ denote strictly increasing primed words of length $r$ and $s$
with all letters greater than $q+1$ and at most $0$, respectively.
Definition~\ref{orf-def} tells us that
$\concat(f_k(w))$ is formed from $\concat(w)$ by replacing the subword \eqref{word1} by 
\be\label{word2}
 2^{\alpha_1}3^{\alpha_2}4^{\alpha_3}\cdots (q+1)^{\alpha_q} (\dash r \dash )(\dash s\dash )  1\hs 2\hs 3\cdots q(q+1)^{\alpha_0}.
\ee
When $q=0$ we interpret \eqref{word1} as $1^{\alpha_0}(\dash r \dash )(\dash s\dash ) $ and \eqref{word2}
as $(\dash r \dash )(\dash s\dash ) 1^{\alpha_0}$.

First suppose $r=s=0$.   
Then it suffices to show that  \eqref{word1} and \eqref{word2} are equivalent
under the transitive closure $\sim$ of the relation on primed words that has $v \sim \ck_i(v)$ for any $i>0$.
This is trivial if $q=0$. If $q>0$, then it is easy to see
\eqref{word1} and \eqref{word2} are respectively equivalent to 
\[
1^{\alpha_0} 2 \hs 1^{\alpha_1}3\hs  2^{\alpha_2}4\hs  3^{\alpha_3}\dots (q+1)q^{\alpha_q}
\quand
 2^{\alpha_1}1 \hs 3^{\alpha_2}2 \hs 4^{\alpha_3}3 \dots (q+1)^{\alpha_q} q  (q+1)^{\alpha_0}.
\]
Applying $\ck_{2q-1}\cdots \ck_5\ck_3\ck_1$ transforms the left word to the right word,
so \eqref{word1} and \eqref{word2} belong to the same equivalence class under $\sim$ as desired.

Suppose next that $r>0$.
Let $h$ be the last letter of $w^k$.
As $x$ is the last unpaired letter in $a=\unprime(w^k)$, there is a letter $g >x+q$ in $w^{k+1}$ with $\unprime(g) < \unprime(h)$.
Let $g$ be the largest such letter.
Define $u^k$ to be the word formed by removing $h$ from $w^k$,
let $u^{k+1}$ be the subword of $w^{k+1}$ consisting of all 
letters less than $g$, and let $\tilde u^{k+1}$ be the subword of $w^{k+1}$ consisting of all letters at least $g$
so that $w^{k+1}= u^{k+1}\tilde u^{k+1}$. One can check that the tuple of increasing words
\[ u :=(w^1\dots,w^{k-1}, u^k, u^{k+1}, h, \tilde u^{k+1}, w^{k+2}, \dots, w^n )
\]
has $\concat(u) \sim_{MN} \concat(w)$ so $u \in \ORF_{n+2}(z)$.
Similarly define $v^k$ to be the word formed by removing $h$ from the $k$th term of $f_k(w)$,
and let $v^{k+1}$ and $\tilde v^{k+1}$ be the subwords of the $(k+1)$th term of $f_k(w)$ consisting of all 
letters less than $g$ and at least $g$, respectively. Then 
\[
v :=(w^1\dots,w^{k-1}, v^k, v^{k+1}, h, \tilde v^{k+1}, w^{k+1}, \dots, w^n )
 \]
 has $\concat(v) \sim_{MN} \concat(f_k(w))$ so $v \in \ORF_{n+2}(z)$.
Moreover, it is easy to see that $f_k(u) = v$,
so  by induction on $r+s$ we have $\concat(u) \sim_{M\tilde N} \concat(v)$
for $\tilde N := N - \ell(\tilde u^{k+1}) - 1 = N - \ell(\tilde v^{k+1}) - 1 $.
This means that $\concat(w) \sim_{MN} \concat(u)  \sim_{MN} \concat(v)  \sim_{MN} \concat(f_k(w))$.

Suppose instead that $s>0$. Our argument is similar to the previous case.
Let $g$ be the first letter of $w^{k+1}$.
Then there is a smallest letter $h \leq x-1$ in $w^{k}$ with $\unprime(g) < \unprime(h)$.
Define $u^{k+1}$ to be the word formed by removing $g$ from $w^{k+1}$,
and let $u^{k}$ and $\tilde u^k$ be the subwords of $w^{k}$ consisting of all 
letters at most $h$ and greater than $h$, respectively,
so that $w^{k}= u^{k}\tilde u^{k}$. Then  
\[ u :=(w^1\dots,w^{k-1}, u^k, g, \tilde u^k, u^{k+1}, w^{k+2}, \dots, w^n )
\]
has $\concat(u) \sim_{MN} \concat(w)$ so $u \in \ORF_{n+2}(z)$.
Similarly define $v^k$ to be the word formed by removing $g$ from the $(k+1)$th term of $f_k(w)$,
and let $v^{k}$ and $\tilde v^{k}$ be the subwords of the $k$th term of $f_k(w)$ consisting of all 
letters at most $h$ and greater than $h$, respectively. Then 
\[
v :=(w^1\dots,w^{k-1},  v^k, g, \tilde v^k, v^{k+1}, w^{k+1}, \dots, w^n )
 \]
likewise has $\concat(v) \sim_{MN} \concat(f_k(w))$ so $v \in \ORF_{n+2}(z)$.
We again have $f_k(u) = v$,
so   by induction  $\concat(u) \sim_{\tilde M N} \concat(v)$
for $\tilde M := M + \ell( u^{k}) + 1 = N +\ell( u^{k}) + 1$,
and this implies that $\concat(w) \sim_{MN} \concat(u)  \sim_{MN} \concat(v)  \sim_{MN} \concat(f_k(w))$.
 
 We conclude by induction on $r+s$ that $\concat(w) \sim_{MN} \concat(f_k(w))$ in all cases.
\end{proof}

The following is clear from Lemmas~\ref{marked-lem} and \ref{ock-key-lem}.
 
\begin{corollary}\label{ock-key-cor} If $a \in \Incr^+_n(z)$, $i \in [n-1]$,
and $f_i(a)\neq 0$, then $\marked(a) = \marked(f_i(a))$.
\end{corollary}


Recall that the subset $\Incr_n(z) \subset \Incr^+_n(z)$
defined in \eqref{incr-eq} 
 is $\q_n$-subcrystal.

\begin{corollary}\label{unprime-cor2}
The map $\unprime : \Incr^+_n(z) \to \Incr_n(z)$ is a quasi-isomorphism of $\gl_n$-crystals.
\end{corollary}

This map is not usually a quasi-isomorphism of $\q_n$-crystals.

\begin{proof}
If $S$ is any set of $2$-cycles of $z$, then $\cC_S := \{ a \in \Incr^+_n(z) : \marked(a) = S\}$
is a $\gl_n$-subcrystal by Corollary~\ref{ock-key-cor},
and $\unprime:\cC_S \to \Incr_n(z)$
is a $\gl_n$-crystal isomorphism  by Lemmas~\ref{ef-ick-lem} and \ref{marked-lem0}.
As $\Incr^+_n(z)$ is a disjoint union of such $\gl_n$-subcrystals $\cC_S$, the result follows.
\end{proof}

The \defn{descent set} of a primed word $w=w_1w_2\cdots w_m$ is $\Des(w) := \{ i \in [m-1] : w_i > w_{i+1}\}$.

\begin{lemma}\label{f0-incr-lem}
Suppose $a \in \Incr^+_n(z)$. Let $q := \weight(a)_1 = \ell(a^1)$ and 
\[w :=\begin{cases} \concat(a) &\text{if }q \leq 1 \\  \ck_{q-2}\cdots \ck_1\ck_0(\concat(a)) &\text{if }q\geq 2.\end{cases}\]
 If  $q = 0$ or if $\weight(a)_2 \neq 0$ and $q \in \Des(w)$, then $f_{\bar 1}(a) = 0$,
and otherwise   $\concat(f_{\bar 1}(a)) = w$.
\end{lemma}

\begin{proof}
If $q=   0$, then $f_{\bar 1}(a) = 0$ since $a^1$ is empty. 
Suppose  $q=1$ so that $w = \concat(a)$. If $\weight(a)_2 \neq 0$ and $1 \in \Des(w)$, then the first and only letter in $a^1$ is strictly larger than the first letter of $a^2$, 
so $f_{\bar 1}(a) = 0$. If $\weight(a)_2 =0$ or $1 \notin \Des(w)$, then $f_{\bar 1} (a)$ is formed from $a$ by moving the only letter in $a^1$ to the beginning of $a^2$, 
so $\concat(f_{\bar 1}(a)) = \concat(a) = w$ as claimed.

Now assume $q \geq 2$ so that $w=  \ck_{q-2}\cdots \ck_1\ck_0(\concat(a))$. 
Write $v_i$ for the $i$th letter of $a^1$.
If $v_1$ and $v_2$ are both primed or both unprimed, then let $\tilde v_1 :=v_1$ and $\tilde v_2:=v_2$,
and otherwise form $\tilde v_1$ and $\tilde v_2$ by reversing the primes on $v_1$ and $v_2$, respectively.
Applying $\ck_0$ to $\concat(a)$ replaces the first two letters  $v_1v_2$  by $\tilde v_2 \tilde v_1$. Each successive application of $\ck_j$ for $j= 1,2,\dots,q-2$ 
then transposes $\tilde v_1$ and the letter to its right. Thus $w$ is formed from $\concat(a)$ by changing
the first $q$ letters from $v_1v_2v_3\cdots v_q$ to $\tilde v_2v_3\cdots v_q \tilde v_1$.

Since $w$ is a primed involution word,
it follows from Proposition~\ref{isim-lem2} that $\lceil  v_1\rceil=\lceil \tilde v_1\rceil $ cannot be equal to the first letter of $\unprime(a^2a^3\cdots a^n)$.
Thus if $a^2$ is nonempty, then $q $ belongs to $ \Des(w)$
precisely when the first letter of $a^1$ is not strictly smaller than every letter in $a^2$.
Therefore if $\weight(a)_2 \neq 0$ and $q \in \Des(w)$, then $f_{\bar 1}(a) = 0$ as claimed.
If instead we have $\weight(a)_2 =0$ or $q \notin \Des(w)$,
then $f_{\bar 1}(a)$ is formed from $a$ by changing $a^1 =v_1v_2v_3\cdots v_q$ to $\tilde v_2v_3\cdots v_q$
and then adding $\tilde v_1$ to the start of $a^2$.
 Comparing this to the description of $w$ above shows that $\concat(f_{\bar 1}(a)) = w$.
\end{proof}

\section{Crystal operators on shifted tableaux}\label{crystal-tab-sect}

Continue to let $n$ be a positive integer.
Recall that if $\lambda$ is a strict partition, then $\SShTab_n(\lambda)$ is
the set of semistandard shifted tableaux of shape $\lambda$ with all entries at most $n$,
and $\ShTab_n(\lambda)$  is the subset of such tableaux with no primed diagonal entries.
Results in  \cite{AssafOguz,HPS,Hiroshima2018}
describe a $\q_n$-crystal structure on   $\ShTab_n(\lambda)$.
Here, we extend this to a $\qq_n$-crystal on the larger set $\SShTab_n(\lambda)$.


\subsection{Skew shifted tableaux}

If $\lambda$ and $\mu$ are strict partitions with $\SD_\mu \subset \SD_\lambda$,
then we write $\mu \subset \lambda$
and set $\SD_{\lambda/\mu} := \SD_\lambda \setminus \SD_\mu$.
A \defn{(semistandard) skew shifted tableau}
of shape $\lambda/\mu$ is a map $ \SD_{\lambda/\mu} \to \{ 1' < 1 < 2' < 2 < \dots\}$
such that entries are weakly increasing along rows and columns,
with no unprimed entries repeated in a column and no primed entries repeated in a row.

If $T$ is a skew shifted tableau and $i\leq j$ are positive integers,
then the set of positions 
$T^{-1}(\{ i' < i < \dots <j' <j\})$ is equal to $\SD_{\lambda/\mu}$
for some strict partitions $\mu \subset \lambda$.
We write $T|_{[i,j]}$ for the skew shifted tableau obtained by restricting $T$ to this subdomain.

A skew shifted tableau is a \defn{rim} if its domain  has no positions 
$(i_1,j_1), (i_2,j_2)$ 
with $i_1 < i_2 $ and $j_1 < j_2$. 
A rim whose domain is connected is a \defn{ribbon};
in French notation, this appears as
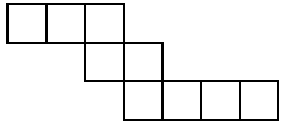
\[
\ytab{ \ & \ &  \
   \\ \none &  \none  & \ & \ 
      \\ \none &  \none  & \none & \ & \ & \ & \  }
\]
or some analogous shape.
If $T$ is a skew shifted tableau, then
the subtableau $T|_{[i,i]}$
is always a rim, and therefore a disjoint union of ribbons, which we call the \defn{$i$-ribbons} of $T$.


Let $T$ be a skew shifted tableau whose domain fits inside $[k]\times [k]$ for some positive integer $k$.
If $C_i$ is the sequence of primed entries in column $i$ of $T$, read in order,
and  $R_i$ is the sequence of unprimed entries in row $i$ of $T$, read in order,
then the \defn{shifted reading word} of $T$ 
is 
\[\shword(T) := \unprime(C_kR_k\cdots C_2R_2C_1R_1).\]
This does not depend on $k$. For example, if 
\[ T= \ytab{ \none & 3 & 5' & 7 \\ 1' & 2' & 4' & 6 & 8'&9},\]
then $\shword(T) = 845237169$.
The order of the letters in $\shword(T)$ defines  a total order on the boxes of $T$ which
we call the \defn{shifted reading word order}.
In the preceding example, this order is $(1,5) < (1,3) < (2,3) < (1,2) <(2,2)< (2,4) < (1,1) < (1,4) < (1,6)$.

The shifted reading word is the same as the \defn{hook reading word} appearing in \cite[Def. 3.4.]{AssafOguz}
and in earlier literature,
but different from the reading wording of a shifted tableau defined in \cite[\S4]{HPS}.
Note that toggling the primes on the diagonal of $T$ 
has no effect on $\shword(T)$.

 \subsection{Formulas for crystal operators}\label{tab-op-sect}
 
This section defines raising and lowering operators on skew shifted tableaux.
This will involve another pairing procedure, now on the boxes in a shifted tableau.
 Fix an integer $i>0$ and let $T$ be a skew shifted tableau.
Assume the domain of $T|_{[i,i+1]}$ has size $N$. Let $\alpha_1,\alpha_2,\dots,\alpha_N$
be the positions in this domain, ordered such that $\alpha_j$ contributes the $j$th letter of $\shword(T|_{[i,i+1]})$.

\begin{definition}\label{tab-pair-def}
Consider the word formed by replacing each $i$ in $\shword(T|_{[i,i+1]})$ by a right parenthesis ``)''
and each $i+1$ in $\shword(T|_{[i,i+1]})$ by a left parenthesis ``(''.
If $j$ and $k$ are the indices of a matching set of parentheses in this word, then we say that $\alpha_j$ and $\alpha_k$ 
are paired. Remove all paired positions from the sequence $(\alpha_1,\alpha_2,\dots,\alpha_N)$
and let 
$ \unpaired_i(T) $ 
denote the resulting subsequence.
\end{definition}

For example, suppose $i=1$ and $T|_{[1,2]}$ is the skew shifted tableau
\[ 
\ytab{ \none & \none  & 1 & 2 & 2 \\
\none & \none[\cdot] & \none[\cdot] & 1' & 1 & 2' & 2  \\ 
\none[\cdot] & \none[\cdot] & \none[\cdot] & \none[\cdot] & 1' & 1  & 1 & 2' & 2}.
\]
Then $\shword(T|_{[1,2]}) = 2 2 1 1 122 12 112$ and the corresponding ordering of the boxes in $T|_{[1,2]}$ is
\[ 
\ytab{ \none & \none  & 5 & 6 & 7 \\
\none & \none[\cdot] & \none[\cdot] & 4 & 8 & 2 & 9  \\ 
\none[\cdot] & \none[\cdot] & \none[\cdot] & \none[\cdot] & 3 & 10  & 11 & 1 & 12}.
\]
The paired positions are $
(\alpha_2,\alpha_3), $ $ (\alpha_1,\alpha_4),$ $ (\alpha_6,\alpha_{11}),$
$(\alpha_7,\alpha_{8}),$ and
$  (\alpha_9,\alpha_{10})$,
  and so we have \[\unpaired_i(T) = (\alpha_5,\alpha_{12}) = ((3,3), (1,9)).\]

Our descriptions of $f_i(T)$ and $e_i(T)$ for $i>0$
are closely modeled on \cite[Defs. 3.5 and 3.9]{AssafOguz}.
Below, when we refer to ``interchanging the primes''  on two elements of $  \ZZ\sqcup \ZZ' = \frac{1}{2}\ZZ$,
we mean the operation   that adds a prime to one number while removing the prime from the other if the two are not both primed or both unprimed,
and otherwise leaves the numbers unchanged.

\begin{definition}\label{shtab-f}
Consider the positions  $(x,y)$ in $\unpaired_i(T)$ with $ T_{xy}\in \{i',i\}$.
If there are no such positions, then set $f_i(T) = 0$.
Otherwise, let $(x,y)$ be the last such position
and construct a skew shifted tableau tableau $f_i(T)$ with the same shape as $T$
by the following procedure:
\ben
\item[(L1)] First assume $T_{xy} = i$. Then we form $f_i(T)$ from $T$
 as follows:
\item[(a)] If $T_{x,y+1} = i+1'$, then form $f_i(T)$ by changing $T_{xy}$ to $i+1'$ and $T_{x,y+1} $ to $i+1$:
\[ {\scriptsize
\ytableausetup{boxsize = 1.0cm,aligntableaux=center}
\begin{ytableau} 
\color{black}T_{x+1,y} & \none \\
\color{red}T_{xy} & \color{red} T_{x,y+1}
\end{ytableau} = \begin{ytableau} 
\color{black} ? & \none\\
\color{red} i & \color{red} i+1'
\end{ytableau}
\ \mapsto\ 
 \begin{ytableau} 
\color{black} ? & \none\\
\color{red} i+1' & \color{red} i+1
\end{ytableau}}.
\]

\item[(b)] If $T_{x,y+1} \neq i+1'$ and $ T_{x+1,y} \notin \{i+1',i+1\}$, then form $f_i(T)$ by changing $T_{xy}$ to $i+1$:
\[ {\scriptsize
\begin{ytableau} 
\color{black} T_{x+1,y} & \none \\
\color{red}T_{xy} & \color{black} T_{x,y+1}
\end{ytableau}  =
 \begin{ytableau} 
\color{black}\substack{\text{not}\\i+1 \\ \text{nor} \\ i+1'} & \none \\
\color{red}i & \color{black}\substack{\text{not}\\i+1'}
\end{ytableau}
\ \mapsto\ 
 \begin{ytableau} 
\color{black} \substack{\text{not}\\i+1 \\ \text{nor} \\ i+1'} & \none \\
\color{red}i+1 & \color{black}\substack{\text{not}\\i+1'}
\end{ytableau}}.\]
 
\item[(c)] If $T_{x,y+1} \neq i+1'$,  $ T_{x+1,y} \in \{i+1',i+1\}$, and the position $(\tilde x, \tilde y)$  farthest northwest in the
 $(i+1)$-ribbon containing $(x+1,y)$ 
 has $\tilde x \neq \tilde y$,
then form $f_i(T)$ by
changing $T_{xy}$ to $i+1'$ and $T_{\tilde x\tilde y}$ to $i+1$:
%
      \[ {\scriptsize
\begin{ytableau} 
\color{red}T_{\tilde x\tilde y}  &\none &  \none & \none \\
\color{black}\ddots&\color{black}\ddots &\none & \none \\
\none &\color{black}\ddots & \color{black} T_{x+1,y} & \none \\
 \none &\none  &\color{red} T_{xy} & \color{black} T_{x,y+1} 
\end{ytableau}  =
 \begin{ytableau} 
\color{red}i+1'  &  \none & \none & \none \\
\color{black}\ddots & \color{black}\ddots& \none & \none \\
\none  & \color{black}\ddots& \color{black}\substack{i+1 \\ \text{or} \\ i+1'} & \none \\
 \none & \none & \color{red}i & \color{black}\substack{\text{not}\\ i+1'}
\end{ytableau} 
\ \mapsto\ 
 \begin{ytableau} 
\color{red}i+1  &  \none & \none & \none \\
\color{black}\ddots & \color{black}\ddots& \none & \none \\
\none & \color{black}\ddots& \color{black}\substack{i+1 \\ \text{or} \\ i+1'} & \none \\
 \none & \none & \color{red} i+1' & \color{black}\substack{\text{not}\\ i+1'}
\end{ytableau} }.
\]
%

\item[(d)] If $T_{x,y+1} \neq i+1'$,  $ T_{x+1,y} \in \{i+1',i+1\}$,
and
 the position $(\tilde x, \tilde y)$  farthest northwest in the
 $(i+1)$-ribbon containing $(x+1,y)$ has $\tilde x = \tilde y$,
then form $f_i(T)$ by first changing $T_{xy}$ to $i+1'$
and then interchanging the primes on   $T_{\tilde x \tilde x}$ and $T_{\tilde x-1,\tilde x-1}$.
Thus we would have
 \[
\ba 
f_1\(\ytab{
\none & \color{red}{ 2} & 2 \\
\color{red}1' & 1 & \color{red}{ 1} &2 }\) &= \ytab{
\none &\color{red} 2' & 2 \\
\color{red}1& 1 & \color{red}2' & 2 }
\quad\text{while}\quad
f_1\(\ytab{
\none & { 2'} & 2 \\
1' & 1 & \color{red}{ 1} &2 }\) = \ytab{
\none & 2' & 2 \\
1'& 1 & \color{red}2' & 2 }
\ea
\]
where in both examples $(x,y)=(1,3)$ and $(\tilde x, \tilde y)=(2,2)$.
\een
The \defn{conjugate}  \cite[\S4]{HPS} of a skew shifted tableau 
  is given by transposing the locations of all boxes and then adding $\frac{1}{2}$ to all entries.
  If $T_{xy} = i'$ then $f_i(T)$  is formed by first conjugating $T$, then following the rules 
in cases L1(a)(b)(c) above (note that case L1(d) will not apply), and then applying the inverse of conjugation to the result.
This corresponds explicitly to the following operations:
\ben
\item[(L2)] Suppose $T_{xy} = i'$. Then we form $f_i(T)$ from $T$ as follows:
\item[(a)] If $T_{x+1,y} = i$, then form $f_i(T)$   by changing $T_{xy}$ to $i$ and $T_{x+1,y} $ to $i+1'$:
\[ {\scriptsize
\ytableausetup{boxsize = 1.0cm,aligntableaux=center}
\begin{ytableau} 
\color{red}T_{x+1,y} & \none \\
\color{red}T_{xy}  &\color{black} T_{x,y+1} 
\end{ytableau} = \begin{ytableau} 
\color{red}i & \none \\
\color{red}i' & \color{black}?
\end{ytableau}
\ \mapsto\ 
 \begin{ytableau} 
\color{red}i+1' & \none \\
\color{red}i & \color{black}?
\end{ytableau}}.
\]

 \item[(b)] If $T_{x+1,y} \neq i$ and $T_{x,y+1} \notin \{i, i+1'\}$, then form $f_i(T)$ by changing $T_{xy}$ to $i+1'$:
\[ 
{\scriptsize
\begin{ytableau} 
\color{black}T_{x+1,y} & \none \\
\color{red}T_{xy} &\color{black} T_{x,y+1}
\end{ytableau} = \begin{ytableau} 
\color{black}\substack{\text{not}\ i} & \none\\
\color{red}i' &\color{black} \substack{\text{not}\ i \\ \text{nor} \\ i+1'}
\end{ytableau}
\ \mapsto\ 
 \begin{ytableau} 
\color{black}\substack{\text{not}\ i}  & \none\\
\color{red}i+1' & \color{black}\substack{\text{not}\ i \\ \text{nor} \\ i+1'}
\end{ytableau}}.
\]

 \item[(c)] If $T_{x+1,y} \neq i$ and $T_{x,y+1} \in \{i, i+1'\}$,
 then form $f_i(T)$ by
 changing $T_{xy}$ to $i$ and $T_{\tilde x \tilde y}$ to $i+1'$,
  where $(\tilde x, \tilde y)$ 
  is  the first position
 in the $i$-ribbon containing $(x,y)$ 
that is southeast of $(x,y)$ with $T_{\tilde x \tilde y} = i$ and
 $T_{\tilde x, \tilde y+1} \notin \{i, i+1'\}$:
 \[ {\scriptsize
\begin{ytableau} 
\color{black}T_{x+1,y} & \none \\
\color{red}T_{xy} & \color{black}T_{x,y+1}  \\
 \color{black}\ddots & \color{black}\ddots & \none \\
\none &\color{black}\ddots& \color{red}T_{\tilde x \tilde y}
\end{ytableau} =\begin{ytableau} 
\color{black}\substack{\text{not}\ i}  & \none \\
\color{red}i' & \color{black}\substack{i \\ \text{or} \\ i+1'}  \\
\color{black}\ddots & \color{black}\ddots & \none \\
\none  &\color{black}\ddots & \color{red}i
\end{ytableau}
\ \mapsto\ 
\begin{ytableau} 
\color{black}\substack{\text{not}\ i}  & \none \\
\color{red}i & \color{black}\substack{i \\ \text{or} \\ i+1'}  \\
\color{black}\ddots & \color{black}\ddots \\
\none &\color{black}\ddots& \color{red}i+1'
\end{ytableau}}.
\]

\een
\end{definition}

\begin{remark}\label{shtab-f-remark}
When $T$ has no primed diagonal entries, the preceding definition coincides with the formula for $f_i(T)$ in 
\cite[Def. 3.5]{AssafOguz}. The latter looks somewhat different from the operator $f_i$ defined in \cite[\S4]{HPS}
but gives the same result by \cite[Prop. 3.19]{AssafOguz}.
For $T$ with no primed diagonal entries, Assaf and Oguz 
show in the proofs of \cite[Thms. 3.8 and 3.10]{AssafOguz}
that 
\begin{itemize}
\item[(i)] in case L1(a) it always holds that $x\neq y$, 

\item[(ii)] in case L1(c) it always holds that $T_{\tilde x \tilde y} = i+1'$, 

\item[(iii)] in case L2(a) it always holds that $x+1\neq y$,  and

\item[(iv)]  in case L2(c) there always exists a position $(\tilde x, \tilde y)$ as described.
\end{itemize}
Properties (i), (ii), and (iii) must also hold when $T$ has primed diagonal entries, 
since any counterexample would remain so on removing all primes from the diagonal.
The same reasoning applies to property (iv) since 
\begin{itemize}
\item[(v)] in case L2(c) it always holds that  $x \neq y$.
\end{itemize}
This follows since if $x=y$ in case L2(c), then 
we must have $T_{x,y+1} =i+1'$ for  $(x,y)$ to be the last position in $\unpaired_i(T)$ with $T_{xy} \in \{i',i\}$,
but then removing the prime from $T_{xy}$  would produce a counterexample to (i).
Finally, it is easy to check that
\begin{itemize}
\item[(vi)] in case L1(d) it must hold that $T_{\tilde x-1, \tilde x-1} \in \{i',i\}$ for $(x,y)$ to be in $\unpaired_i(T)$.
\end{itemize}
\end{remark}

Let $\primes(T)$ denote the total number of primed entries in $T$ 
and let $\primes_\diag(T)$ denote the number of primed diagonal entries in $T$.
Remark~\ref{shtab-f-remark} implies the following lemma.

\begin{lemma}\label{shtab-primed-lem}
Suppose $f_i(T) \neq 0$. 
Then $\primes_\diag(T) = \primes_\diag(f_i(T))$. If case L1(d) applies in Definition~\ref{shtab-f}, then 
$ \primes(f_i(T))=\primes(T)+1$. In all other 
cases $ \primes(f_i(T))=\primes(T)$ and the sets of primed diagonal positions in $f_i(T)$ and $T$ coincide.
\end{lemma}


Next, we define raising operators $e_i$. Recall that $i>0$  and  $T$ is a skew shifted tableau.

\begin{definition}\label{shtab-e}
Consider the positions  $(x,y)$ in $\unpaired_i(T)$ with $ T_{xy}  \in \{i+1',i+1\}$.
If there are no such positions, then set $e_i(T) = 0$.
Otherwise, let $(x,y)$ be the first such position
and construct a skew shifted tableau tableau $e_i(T)$ with the same shape as $T$
by the following procedure:
\ben
\item[(R1)] First assume $T_{xy} = i+1$. Then we form $e_i(T)$ from $T$
 as follows:
 \item[(a)] If $T_{x,y-1} = i+1'$, then form $e_i(T)$ by changing $T_{xy}$ to $i+1'$ and $T_{x,y-1} $ to $i$:
\[ {\scriptsize
\ytableausetup{boxsize = 1.0cm,aligntableaux=center}
\begin{ytableau} 
\color{red}T_{x,y-1} &\color{red} T_{xy} \\
\none & T_{x-1,y}
\end{ytableau} = \begin{ytableau} 
\color{red}i+1' &\color{red} i+1\\
\none & ?
\end{ytableau}
\ \mapsto\ 
 \begin{ytableau} 
\color{red}i&\color{red} i+1'\\
\none & ?
\end{ytableau}}.
\] 

\item[(b)] If $T_{x,y-1} \neq i+1'$ and $ T_{x-1,y} \notin\{i,i+1'\}$, then form $e_i(T)$ by changing $T_{xy}$ to $i$:
\[ {\scriptsize
\begin{ytableau} 
T_{x,y-1} & \color{red} T_{xy} \\
\none & T_{x-1,y}
\end{ytableau} = \begin{ytableau} 
\substack{\text{not} \\ i+1'} & \color{red} i+1\\
\none &  \substack{\text{not }i \\\text{nor} \\ i+1'} 
\end{ytableau}
\ \mapsto\ 
 \begin{ytableau} 
\substack{\text{not} \\ i+1'} & \color{red} i\\
\none &  \substack{\text{not }i \\\text{nor} \\ i+1'} 
\end{ytableau}}.
\]

\item[(c)] If $T_{x,y-1} \neq i+1'$ and $ T_{x-1,y} \in \{i,i+1'\}$,
then form $e_i(T)$ by 
changing $T_{xy}$ to $i+1'$ and $T_{\tilde x\tilde y}$ to $i$,
where $(\tilde x, \tilde y)$ is the first position  in the
 $(i+1)$-ribbon containing $(x,y)$
 that is southeast of $(x,y)$ with $T_{\tilde x \tilde y} = i+1'$ and
$T_{\tilde x-1, \tilde y} \notin \{i, i+1'\}$:
 \[ {\scriptsize
\begin{ytableau} 
T_{x,y-1} & \color{red} T_{xy} & \ddots\\
\none & T_{x-1,y}  & \ddots &\ddots \\
\none &\none  & \none & \color{red} T_{\tilde x \tilde y}
\end{ytableau} =\begin{ytableau} 
\substack{\text{not} \\ i+1'} & \color{red}i+1& \ddots \\
\none& \substack{i \\\text{or} \\ i+1'}  & \ddots &\ddots \\
\none&\none  & \none &\color{red} i+1'
\end{ytableau}
\ \mapsto\ 
\begin{ytableau} 
\substack{\text{not} \\ i+1'} &\color{red} i+1' & \ddots \\
\none & \substack{i \\\text{or} \\ i+1'}  & \ddots  & \ddots \\
\none & \none  & \none  &\color{red} i
\end{ytableau}}.
\]

\item[(R2)] Alternatively suppose $T_{xy} = i+1'$. Then we form $e_i(T)$ from $T$ as follows:
 \item[(a)] If $T_{x-1,y} = i$, then form $e_i(T)$   by changing $T_{xy}$ to $i$ and $T_{x-1,y} $ to $i'$:
\[ {\scriptsize
\ytableausetup{boxsize = 1.0cm,aligntableaux=center}
\begin{ytableau} 
T_{x,y-1} &\color{red} T_{xy} \\
\none & \color{red}T_{x-1,y}
\end{ytableau} = \begin{ytableau} 
? &\color{red} i+1'\\
\none & \color{red}i
\end{ytableau}
\ \mapsto\ 
 \begin{ytableau} 
?&\color{red} i\\
\none & \color{red}i'
\end{ytableau}}.
\] 

 \item[(b)] If $T_{x-1,y} \neq i$ and $T_{x,y-1} \notin \{i', i\}$, then form $e_i(T)$ by changing $T_{xy}$ to $i+1'$:
\[ {\scriptsize
\ytableausetup{boxsize = 1.0cm,aligntableaux=center}
\begin{ytableau} 
T_{x,y-1} &\color{red} T_{xy} \\
\none &T_{x-1,y}
\end{ytableau} = \begin{ytableau} 
\substack{\text{not }i' \\ \text{nor }i} &\color{red} i+1'\\
\none &  \substack{\text{not }i}
\end{ytableau}
\ \mapsto\ 
 \begin{ytableau} 
\substack{\text{not }i' \\ \text{nor }i} &\color{red} i'\\
\none & \substack{\text{not }i} 
\end{ytableau}}.
\] 

 \item[(c)] If $T_{x-1,y} \neq i$, $T_{x,y-1} \in \{i', i\}$,
 and 
 the position $(\tilde x, \tilde y)$ farthest northwest in the
 $i$-ribbon containing $(x,y-1)$ has $\tilde x \neq \tilde y$,
 then 
form $e_i(T)$ by
changing $T_{xy}$ to $i$ and $T_{\tilde x\tilde y}$ to $i'$:
      \[ {\scriptsize
\begin{ytableau} 
T_{\tilde x\tilde y}  &  \ddots & \none \\
\none & \ddots & \ddots\\
\none &\none &  T_{x,y-1} & T_{xy} \\
\none & \none & \none & T_{x-1,y} 
\end{ytableau}  =
 \begin{ytableau} 
\color{red}i &  \ddots & \none \\
\none & \ddots & \ddots\\
\none &\none &   \substack{i' \ \text{or} \ i} & \color{red}i+1' \\
\none & \none & \none & \substack{\text{not} \ i}
\end{ytableau} 
\ \mapsto\ 
 \begin{ytableau} 
\color{red}i'  &  \ddots & \none \\
\none & \ddots & \ddots\\
\none &\none &  \substack{i' \ \text{or} \ i} & \color{red}i \\
\none & \none &\none &  \substack{\text{not} \ i}
\end{ytableau} }.
\]

\item[(d)] If $T_{x-1,y} \neq i$, $T_{x,y-1} \in \{i', i\}$,
 and 
 the position $(\tilde x, \tilde y)$ farthest northwest in the
 $i$-ribbon containing $(x,y-1)$ has $\tilde x = \tilde y$,
 then form $e_i(T)$ by first changing $T_{xy}$ to $i$
and then
interchanging the primes on  $T_{\tilde x \tilde x}$ and $T_{\tilde x+1,\tilde x+1}$.
Thus we would have
 \[
\ba 
e_1\(
\ytab{
\none &\color{red} 2' & 2 \\
\color{red}1& 1 & \color{red}2' & 2 }
\) &= \ytab{
\none & \color{red}{ 2} & 2 \\
\color{red}1' & 1 & \color{red}{ 1} &2 }
\quad\text{while}\quad
e_1\( \ytab{
\none & 2 & 2 \\
1& 1 & \color{red}2' & 2 }\) =\ytab{
\none &  2 & 2 \\
1 & 1 & \color{red}{ 1} &2 }
\ea
\]
where in both examples $(x,y)=(1,3)$ and $(\tilde x, \tilde y)=(1,1)$.

\een
\end{definition}

\begin{remark}\label{shtab-e-remark}
When $T$ has no primed diagonal entries, the preceding definition coincides with the formula for $e_i(T)$ in 
\cite[Def. 3.9]{AssafOguz} (and also in \cite[\S4]{HPS} via \cite[Prop. 3.19]{AssafOguz}).
%
There are versions of the properties in Remark~\ref{shtab-f-remark} for the raising operators.
First, we have:
\begin{itemize}
\item[(i)] in case R2(a)  it always holds that $x \neq y$, 
\item[(ii)] in case R2(c) it always holds that $T_{\tilde x \tilde y} = i$, and
\item[(iii)] in case R1(a)  it always holds that $x \neq y-1$.
\end{itemize}
These hold since it is a straightforward exercise to show that any counterexample
leads to contradiction of the fact that $(x,y)$ is the first unpaired position with $T_{xy} \in \{i+1',i+1\}$. Next:
\begin{itemize}
\item[(iv)] in case R1(c)  it always holds that $x \neq y$.
\end{itemize}
This follows since if $x=y$ in case R1(c), then 
  $T_{x,y+1} =i$ as  $(x,y)$ is the last position in $\unpaired_i(T)$ with $T_{xy} \in \{i+1',i+1\}$,
but then adding a prime to $T_{xy}$  would give a counterexample to (i). Also:
\begin{itemize}
\item[(v)]  in case R1(c) there always exists a position $(\tilde x, \tilde y)$ as described.
\end{itemize}
This holds since any counterexample would remain so on removing all primes from the diagonal,
contradicting the fact that the raising operators in \cite[Def. 3.9]{AssafOguz} are well-defined.
Alternatively, (v) can be shown by mimicking the last paragraph of the proof of \cite[Thm. 3.8]{AssafOguz}.
Finally:
\begin{itemize}
\item[(vi)] in case R2(d) it must hold that $T_{\tilde x+1, \tilde x+1} \in \{i+1',i+1\}$ for $(x,y)$ to be in $\unpaired_i(T)$.
\end{itemize}
\end{remark}

An analogue of Lemma~\ref{shtab-primed-lem}  follows from this remark.

\begin{lemma}\label{shtab-primed-lem2}
Suppose $e_i(T) \neq 0$. 
Then $\primes_\diag(T) = \primes_\diag(e_i(T))$. If case R2(d) applies in Definition~\ref{shtab-e}, then 
$ \primes(e_i(T))=\primes(T)-1$. In all other 
cases $ \primes(e_i(T))=\primes(T)$ and the sets of the primed diagonal positions in $e_i(T)$ and $T$ coincide.
\end{lemma}



To define the remaining operators $f_{\bar 1}$, $e_{\bar 1}$, $f_0$, and $e_0$, we require that $T $ be a semistandard shifted tableau
(rather than a skew shifted tableau). 
When $T$ has no primed diagonal entries, the next two definitions reduce to the formulas
in \cite[Defs. 4.4 and 4.5]{AssafOguz} and \cite[Lems. 3.1 and 3.2]{Hiroshima2018}.

\begin{definition}\label{tab-qf-def}
If $T$ has a $2'$ in its first row, or no entries equal to $1'$ or $1$,
then set $f_{\bar 1}(T):=0$:
\[
f_{\bar 1}\(\ytab{\none& 3 \\ 1 & 2' & 2}\) = f_{\bar 1}\(\ytab{\none &3 \\ 2 & 2}\)=0.
\]
Otherwise form $f_{\bar 1}(T)$  by changing the last entry 
equal to $1'$ or $1$ in the first row of $T$ to $2'$, unless this entry is unprimed and on the diagonal,
in which case it changes to $2$:
\[
f_{\bar 1}\(\ytab{\none& 3 \\ 1' & {\color{red}1} & 2}\) = \ytab{\none& 3 \\ 1' & {\color{red}2'} & 2},
\quad
f_{\bar 1}\(\ytab{\none& 3 \\ {\color{red}1} & 2 }\) = \ytab{\none& 3 \\ {\color{red}2} & 2 },
\quand
f_{\bar 1}\(\ytab{\none& 3 \\ {\color{red}1'} & 2 }\) = \ytab{\none& 3 \\ {\color{red}2'} & 2 }.
\]
\end{definition}

\begin{definition}
If $T_{11} \in \{ 2',2\}$, then form $e_{\bar 1}(T)$ from $T$ by subtracting one from this entry: 
\[
e_{\bar 1}\(\ytab{\none& 3 \\ {\color{red}2'} & 2}\) = \ytab{\none &3 \\ {\color{red}1'} & 2}
\quand
e_{\bar 1}\(\ytab{\none &3 \\ {\color{red}2} & 2}\) = \ytab{\none &3 \\ {\color{red}1} & 2}.
\]
If $T_{11} \notin \{2', 2\}$ but the first row of $T$ has a (necessarily unique) entry $2'$,
then form $e_{\bar 1}(T)$ from $T$ by changing this entry to $1$;
otherwise  set $e_{\bar 1}(T):=0$:
\[
e_{\bar 1}\(\ytab{\none& 3 \\ 1 & {\color{red}2'} & 2}\) = \ytab{\none &3 \\ 1 & {\color{red}1} & 2}
\quand
e_{\bar 1}\(\ytab{\none &3 \\ 1 & 2 & 2}\) = 0.
\]
\end{definition}


\begin{definition}\label{shtab-f0}
 If $T_{11} \neq 1$ then set $f_0(T):=0$, and otherwise
form $f_0(T)$ from $T$ by changing $T_{11}$  to $1'$. 
If $T_{11} \neq 1'$ then set $e_0(T):=0$, and otherwise
form $e_0(T)$ from $T$ by changing $T_{11}$  to $1$. 
\end{definition}



Given a skew shifted tableau $T$, form $\unprime_\diag(T)$ by removing the primes
from all diagonal entries.
This operation commutes with the maps $e_i$ and $f_i$ in the following sense.

\begin{lemma}\label{tab-unprime-lem}
Let $T$ be a skew shifted tableau of shape $\lambda/\mu$ and suppose $i$ is a positive integer.
Then 
$e_i ( \unprime_\diag(T)) = \unprime_\diag ( e_i(T))$ and $ f_i ( \unprime_\diag(T))= \unprime_\diag ( f_i(T))$
on setting $\unprime_\diag(0) := 0$. 
When $\mu = \emptyset$ the same identities hold 
for $i=\bar 1$.
\end{lemma}

\begin{proof}
The desired identities are easily checked when $\mu = \emptyset$ and $i=\bar 1$. Assume $i \in [n-1]$.
Since $T$ and $\unprime_\diag(T)$ have the same shifted reading word, 
if follows that $f_i(T) =0$ if and only if $f_i(\unprime_\diag(T)) = 0$.
Assume this does not occur, and 
let $(x,y)$ be the unpaired position that arises in Definition~\ref{shtab-f} when applying $f_i$ to $T$.
This $(x,y)$ must also be the unpaired position that arises in Definition~\ref{shtab-f} when evaluating $f_i(\unprime_\diag(T))$.

The properties in Remark~\ref{shtab-f-remark} ensure that 
whichever case of Definition~\ref{shtab-f} applies when evaluating $f_i(T)$,
the same case applies when evaluating $f_i(\unprime_\diag(T))$, with one exception.
Outside this exception, it is evident that $f_i(\unprime_\diag(T)) = \unprime_\diag(f_i(T))$.
The exception is that if $x=y$ and
case L2(b) of Definition~\ref{shtab-f} is invoked when applying $f_i$ to $T$,
then case L1(b) applies when evaluating $f_i(\unprime_\diag(T))$.
However, in this situation it is clear from Definition~\ref{shtab-f} that we again have
 $f_i(\unprime_\diag(T))=\unprime_\diag(f_i(T))$.

Checking that $e_i ( \unprime_\diag(T)) = \unprime_\diag ( e_i(T))$ follows by a similar argument.
\end{proof}

Choose strict partitions $\mu\subset \lambda$ and let 
$\SShTab_n(\lambda/\mu)$  be the set of (semistandard) skew  shifted tableaux of shape $\lambda/\mu$ 
with all entries at most $n$.
Let $\ShTab_n(\lambda/\mu)$ be the subset of tableaux in $\SShTab_n(\lambda/\mu)$
with no primed diagonal entries.
The \defn{weight} of  $T\in \SShTab_n(\lambda/\mu)$ 
is the vector $\weight(T) := (a_1,a_2,\dots,a_n)$
where $a_i$ is the number of entries in $T$ equal to $i'$ or $i$.

Restricted to $\ShTab_n(\lambda/\mu)$,
the operators $e_i$ and $f_i$ for $i \in \{\bar 1,1,2,\dots,n-1\}$
coincide with the ones in \cite{AssafOguz,HPS,Hiroshima2018}.
From those papers (see in particular \cite[Thm 4.8]{AssafOguz}), we  know that
if $\lambda \in \NN^n$ is any strict partition, then these operators 
make $\ShTab_n(\lambda)$ into a connected normal $\q_n$-crystal.

\begin{corollary}\label{skew-norm-cor}
When nonempty, the set $\ShTab_n(\lambda/\mu)$ 
is a normal (but not necessarily connected) $\gl_n$-crystal relative  to the operators $e_i$ and $f_i$  defined above. 
\end{corollary}

\begin{proof}
Let $k := \ell(\mu)$. If $\cB$ is a $\gl_{k+n}$-crystal with weight map $\tilde \weight$
and crystal operators $\tilde e_i$ and $\tilde f_i$
then $\cB$ may be regarded as a $\gl_n$-crystal with weight map $\weight (b) := ( \tilde \weight(b)_{k+1},\dots, \tilde \weight(b)_{k+n} ) \in \ZZ^n$
and crystal operators $e_i := \tilde e_{k+i}$ and $f_i := \tilde f_{k+i}$ for $i \in [n-1]$.
This gives a functor $\sF$ from $\gl_{k+n}$-crystals to $\gl_n$-crystals.
It is straightforward to see that $\sF(\cB \otimes \cC) \cong \sF(\cB) \otimes \sF(\cC)$
and $\sF(\BB_{k+n}) \cong \One   \sqcup \cdots \sqcup \One \sqcup \BB_n$, so this functor sends normal $\gl_{k+n}$-crystals
to normal $\gl_n$-crystals.

Now consider the subset of $T\in \ShTab_{k+n}(\lambda)$ that have $i$ in all boxes in row $i$ of $\SD_\mu \subset \SD_\lambda$
and only entries greater than $k$ in $\SD_{\lambda/\mu}$.
This is a union of full subcrystals of the normal $\gl_n$-crystal
$\sF(\ShTab_{k+n}(\lambda))$, and is isomorphic to the 
 prospective crystal structure on $\ShTab_n(\lambda/\mu)$.
\end{proof}

 Theorem~\ref{sshtab-upgrade} and Corollary~\ref{sshtab-upgrade-cor}
 will give $\qq_n$-analogues of the above facts for sets 
 of shifted tableaux with  
primes allowed on the diagonal.
Here, we only prove this easier statement: 

\begin{proposition}
When nonempty, the set $\SShTab_n(\lambda)$ (respectively, $\SShTab_n(\lambda/\mu)$) is a $\qq_n$-crystal 
(respectively, $\gl_n$-crystal)
relative to the operators $e_i$ and $f_i$ defined above.
\end{proposition}

The set $\SShTab_n(\lambda)$ is empty if and only if $\ell(\lambda)> n$.
For an example of this crystal see Figure~\ref{tab-fig}.

\begin{proof}
Suppose  $T \in \SShTab_n(\lambda/\mu)$ and $i \in [n-1]$.
First assume $ f_i(T)\neq 0 $.
Then   
\be\label{uun-eq}\unprime_\diag(e_i(f_i(T))) = e_i(f_i(\unprime_\diag(T))) = \unprime_\diag(T)\ee by Lemma~\ref{tab-unprime-lem} and Corollary~\ref{skew-norm-cor}.
It follows that we have $\primes(e_i(f_i(T))) = \primes(T)$ by Lemmas~\ref{shtab-primed-lem} and \ref{shtab-primed-lem2}.

If $\primes(f_i(T)) = \primes(T)$, then Lemma~\ref{shtab-primed-lem} tells us
that
applying $f_i$ to $T$ must not invoke case L1(d) in Definition~\ref{shtab-f}
while 
Lemma~\ref{shtab-primed-lem2} tells us
that
 applying $e_i$ to $f_i(T)$ must not invoke case R2(d) in Definition~\ref{shtab-e}.
But this means that the sets of primed diagonal positions in $T$, $f_i(T)$, and $e_i(f_i(T))$ all coincide, so we have $e_i(f_i(T)) = T$ in view of \eqref{uun-eq}.

If $\primes(f_i(T)) \neq \primes(T)$, on the other hand, then Lemmas~\ref{shtab-primed-lem} and \ref{shtab-primed-lem2}
imply that we must be in case L1(d) of Definition~\ref{shtab-f} when applying $f_i$ to $T$
and in case L2(d) of Definition~\ref{shtab-e} when applying $e_i$ to $f_i(T)$.
The only way that \eqref{uun-eq} can hold in this situation is if the unpaired position $(x,y)$ arising in both definitions is the same, but then we again have $e_i(f_i(T)) = T$.

When $e_i(T) \neq 0$ a symmetric argument shows that $f_i(e_i(T)) = T$.
Since $\ShTab_n(\lambda/\mu)$ is a $\gl_n$-crystal and $\unprime_\diag$ is a weight-preserving map,
this suffices by Lemma~\ref{tab-unprime-lem}
to show that $\SShTab_n(\lambda/\mu)$ is a $\gl_n$-crystal.
Taking $\mu =\emptyset$, we conclude that $\SShTab_n(\lambda)$  is also a $\gl_n$-crystal.
Checking the remaining axioms to show that $\SShTab_n(\lambda)$ is a $\qq_n$-crystal is straightforward.
\end{proof}

The characters of 
$\SShTab_n(\lambda/\mu)$ and $\ShTab_n(\lambda/\mu)$ 
are the \defn{skew Schur $Q$- and $P$-polynomials}
$\ch(\SShTab_n(\lambda/\mu)) = Q_{\lambda/\mu}(x_1,x_2,\dots,x_n)$
and
$\ch(\ShTab_n(\lambda/\mu)) = P_{\lambda/\mu}(x_1,x_2,\dots,x_n)$,
defined by setting $x_{n+1}=x_{n+2}=\dots=0$ in the relevant power series discussed in \cite[\S8]{Stembridge1989}.
%
%

\begin{figure}[h]
\begin{center}
\input{qq3-shifted-tableau-crystal.tex}
\end{center}
\caption{Crystal graph of $\qq_3$-crystal $\SShTab_3(\lambda)$ for $\lambda=(2,1)$.
In this picture we draw styled edges without labels for clarity. Solid blue and red arrows are edges 
$b \xrightarrow{\ 1\ } c$ and $b \xrightarrow{\ 2\ } c$, respectively. Dotted green and dashed blue arrows 
are edges $b \xrightarrow{\ 0\ } c$ and $b \xrightarrow{\ \bar 1\ } c$, respectively.
}
\label{tab-fig}
\end{figure}

\subsection{Highest and lowest weights} \label{tab-highest-sect}

Let $\lambda \in \NN^n$ be a strict partition.
Results in \cite{Hiroshima2018} identify the 
unique $\q_n$-highest and $\q_n$-lowest weight elements in the connected normal $\q_n$-crystal $\ShTab_n(\lambda)$.
Define
$\Thighest_\lambda$
to be the shifted tableau of shape $\lambda$ whose entries in row $i$ are all $i$. Define $\Tlowest_\lambda$
to be the unique shifted tableau of shape $\lambda$ with no primed diagonal positions
whose entries along
the ribbon 
$\SD_{(\lambda_{i+1},\lambda_{i+2},\lambda_{i+3},\dots) / (\lambda_{i+2},\lambda_{i+3},\dots)}$
are each $(n-i)'$ or $n-i$, for $i=0,1,2,\dots,n-1$. 
The latter construction depends on $n$, although we suppress this in our notation.
If $n=5$, then 
\[
\Thighest_{(7,4,2)} = \ytab{ \none & \none & 3 &3 \\ \none & 2 & 2 & 2  & 2 \\ 1 & 1 & 1 & 1 & 1 &1 & 1 }
\quand
\Tlowest_{(7,4,2)} = \ytab{ \none & \none & 5 &5 \\ \none & 4 & 4 & 5' & 5 \\ 3 & 3 & 4' & 4 & 5' &5 & 5 }.
\]

\begin{theorem}[{\cite[Thm. 3.3]{Hiroshima2018}}]
\label{thigh-thm}
The $\q_n$-crystal $\ShTab_n(\lambda)$ is connected with unique $\q_n$-highest weight element $\Thighest_\lambda$
and unique $\q_n$-lowest weight element $\Tlowest_\lambda$.
\end{theorem}

Lemma~\ref{tab-unprime-lem} shows that $\unprime_\diag:\SShTab_n(\lambda)\to \ShTab_n(\lambda)$
is a weight-preserving map that commutes with all crystal operators, excluding $e_0$ and $f_0$.
It follows that $\unprime_\diag$ commutes with the involutions $\sigma_i$ for all $i \in [n-1]$ in \eqref{weyl-action-eq},
and hence also with the operators $e_{\bar i}$, $f_{\bar i}$, $e_{\bar i'}$, and $f_{\bar i'}$ for all $i \in [n-1]$
used in Definition~\ref{q-highest-def}.

Since  $\unprime_\diag(T) \neq 0$ if $T\neq 0$, the map
$\unprime_\diag$ must send $\qq_n$-highest and  $\qq_n$-lowest weight elements in $\SShTab_n(\lambda)$
to $\q_n$-highest and $\q_n$-lowest weight elements in $\ShTab_n(\lambda)$.
Consulting Definition~\ref{qq-highest-def}, we deduce that
$T \in \SShTab_n(\lambda)$
 is a $\qq_n$-highest  weight element if and only if 
 \be\label{useful-eq1}
 \unprime_\diag(T) = \Thighest_\lambda\quand e_{0}^{[i]}(T)  = 0\text{ for all $i \in [n]$},
 \ee
 and that $T \in \SShTab_n(\lambda)$ is a $\qq_n$-lowest  weight element
  if and only if 
 \be\label{useful-eq2}
 \unprime_\diag(T) =\Tlowest_\lambda\quand f_{0}^{[i]}(T)  = 0\text{ for all $i \in [n]$}.
 \ee
When $\unprime_\diag(T)$ is $ \Thighest_\lambda$ or $ \Tlowest_\lambda$
there are simple formulas for  $e_{0}^{[i]}(T)$ and $f_{0}^{[i]}(T)$.

\begin{lemma}\label{i,inotinT}
Let $i \in [n-1]$,  $U \in  \SShTab_n(\lambda)$, and $\alpha = \weight(U)$.
\ben
\item[(a)] If $\alpha_i = 0$, then $\sigma_i(U) $ is formed by subtracting $1$
from every entry of $U$ equal to $i+1'$ or $i+1$.

\item[(b)] If $ \alpha_{i+1}=0$, then $\sigma_i(U) $ is formed by adding $1$
to every entry of $U$ equal to $i'$ or $i$.
\een
\end{lemma}

\begin{proof}
Because $\sigma_i$ is an involution, it suffices to prove (a).
Assume $\alpha_i=0$.
 If $\alpha_{i+1} = 0$, then the desired identity is $\sigma_i(U) = U$, which holds since $f_i(U) = e_i(U) = 0$.
Further assume $\alpha_{i+1} > 0$. Then $\varphi_i(b) - \varepsilon_i(b) = \alpha_i - \alpha_{i+1} = -\alpha_{i+1}<0$
so by \eqref{weyl-action-eq}
 we have  $\sigma_i(U) = e_i^{\alpha_{i+1}}(U)$.

The shifted tableau $U$ has no boxes containing $i$ or $i'$ and a positive number of boxes 
containing $i+1$ or $i+1'$, which we call \defn{changeable boxes}. Consider the shifted reading word order restricted to the changeable boxes of $U$.
We claim that for each integer $0\leq j \leq \alpha_{i+1}$ the tableau $e_i^{j}(U)$
is formed from  $U$ by subtracting $1$ from the entries in the first $j$ changeable boxes in this order.

This certainly holds when $j=0$. Now suppose the claim is true for some integer $0 \leq j < \alpha_{i+1}$ and let $T := e_i^j(U)$. By  hypothesis all letters equal to $i+1$ in $\shword(T)$ occur after all letters equal to $i$ so the pairing in Definition~\ref{tab-pair-def} is trivial. Therefore when applying $e_i$ to $T$, the unpaired position $(x,y)$ in Definition~\ref{shtab-e} is the $(j+1)$th changeable box and we have $T_{xy} = U_{xy}  \in \{i+1',i+1\}$.

Consider the case of  Definition~\ref{shtab-e} that applies to compute $e_i(T) = e_i^{j+1}(U)$.
If $T_{xy}=i+1'$, then it is impossible to be in case R2(a) or R2(c), since then an $i$ would occur after $i+1$ in $\shword(T)$. 
Similarly if $T_{xy}=i+1$, then it is impossible to be in case R1(a) or R1(c), since by hypothesis $(x,y)$ 
contributes the first letter equal to $i+1$ or $i+1'$ in $\shword(T)$.
Thus we 
 must be in case R1(b) or R2(b), so $e_i(T)$ is formed from $T$ by subtracting one from  $T_{xy}$.
 This means that $e_i^{j+1}(U) = e_i(T)$ is formed from $U$ by subtracting one from the entries in the first $j+1$ changeable boxes,
which proves our claim by induction. Invoking the claim with $j= \alpha_{i+1}$ 
proves part (a) of the lemma.
\end{proof}

\begin{lemma} \label{s_i-highest}
Let  
$T$ be a semistandard shifted tableau with
$\unprime_\diag(T|_{[i,i+1]}) = \Thighest_\lambda|_{[i,i+1]}$ for some  $i \in [\ell(\lambda)-1]$.
Then $\sigma_i(T)$ is formed from $T$ by interchanging the primes on the entries in boxes $(i,i)$ and $(i+1,i+1)$
and changing the last $\lambda_i - \lambda_{i+1}$ boxes containing $i$ in row $i$ 
 from 
  \[\ytableausetup{boxsize = 1.1cm,aligntableaux=center}
\begin{ytableau} i & i& \dots & i
\end{ytableau}
\quad\text{to}\quad
\ytableausetup{boxsize = 1.1cm,aligntableaux=center}
\begin{ytableau} i+1' & i + 1& \dots & i+1
\end{ytableau}.\]
\end{lemma}

For example, 
$ \sigma_1 \( \ytab{   \none & 2' &2 \\   1 & 1 & 1  & 1  }\) = \ytab{  \none & 2 &2 \\   1' & 1 & 2'  & 2 }  $
and
$ \sigma_1 \( \ytab{   \none & 2' &2 \\   1' & 1 & 1  & 1  }\) = \ytab{  \none & 2' &2 \\   1' & 1 & 2'  & 2 }  $.

\begin{proof}
By hypothesis the integer $k := \varphi_i(T) - \varepsilon_i(T)  = \weight(T)_i - \weight(T)_{i+1}$ is equal to $\lambda_i-\lambda_{i+1}> 0$
so $\sigma_i(T) = f_i^k(T)$.
The skew tableau
$T|_{[i,i+1]}$ has the form
\[\ytableausetup{boxsize = 1.1cm,aligntableaux=center}
\begin{ytableau} \none &i+1^\ast& \dots &i+1 \\ i^\ast & i& \dots & i & i & \dots & i
\end{ytableau}\]
where there are two possibilities for each of the diagonal entries $i^\ast \in \{i',i\}$ and $i+1^\ast \in \{i+1',i+1\}$.
Therefore the sequence $\unpaired_i(T)$ from Definition~\ref{tab-pair-def} consists of the last $k$ boxes containing $i$ in row $i$.
Each time we apply $f_i$ to $f_i^{j-1}(T)$ for $j=1,2,\dots,k-1$,
case L1(b) in Definition~\ref{shtab-f} changes the $j$th of these boxes, ordered right to left, from $i$ to $i+1$.
When we finally apply $f_i$ to $f_i^{k-1}(T)$,
 case L1(d) in Definition~\ref{shtab-f} results in the next box changing from $i$ to $i+1'$ and 
the primes on the entries in boxes $(i,i)$ and $(i+1,i+1)$ being interchanged.
 Thus $\sigma_i(T) $ is as described.
\end{proof}

\begin{lemma}\label{lem-qn+highestwt}
Let $i \in [\ell(\lambda)]$ and 
suppose $T \in  \SShTab_n(\lambda)$ has $\unprime_\diag(T) = \Thighest_\lambda$.
\ben
\item[(a)] If $T_{ii} \in \ZZ'$, then $e_{0}^{[i]}(T)$ is formed  from $T$ by removing the prime from this entry. 

\item[(b)] If $T_{ii} \in \ZZ$, then $f_{0}^{[i]}(T)$ is formed  from $T$ by adding a prime to this entry. 

\item[(c)] If $j \in [n]\setminus[\ell(\lambda)]$, then $e_{0}^{[j]}(T)=f_{0}^{[j]}(T)=0$.
\een
\end{lemma}
 
\begin{proof}
Let $U$ be the shifted tableau formed from $T$ 
by reversing the prime on entry $T_{ii}$.
Recall the definitions of 
$e_{0}^{[i]}$
and 
$f_{0}^{[i]} $
from \eqref{ef0-eq}.
To prove parts (a) and (b) it suffices to check that $e_0$ (respectively, $f_0$) transforms
$\sigma_1\sigma_2\cdots \sigma_{i-1}(T) $ to $\sigma_1\sigma_2\cdots \sigma_{i-1}(U)$
when $T_{ii}$ is primed (respectively, unprimed).
This is straightforward using Lemma~\ref{s_i-highest}. 
In part (c) we have $(j,j)\notin T$, so Lemma~\ref{i,inotinT} implies
 that $\sigma_1\sigma_2\cdots \sigma_{j-1}(T) $ is formed from $T$ by adding $1$ to all of its entries,
 so $\sigma_1\sigma_2\cdots \sigma_{j-1}(T) $  has no boxes containing $1'$ or $1$ and therefore $e_{0}^{[j]}(T)=f_{0}^{[j]}(T)=0$.
\end{proof}

\begin{lemma} \label{s_i-lowest}
Let  
$T$ be a semistandard shifted tableau with
$\unprime_\diag(T|_{[i,i+1]}) = \Tlowest_\lambda|_{[i,i+1]}$ for some   $i \in [n-1] \setminus [n-\ell(\lambda)]$.
%
Let  $j=i + \ell(\lambda)   -n $.
Then $\sigma_i(T)$ is formed from $T$ by interchanging the primes on the entries in boxes $(j,j)$ and $(j+1,j+1)$
and changing all entries  equal to $i+1'$ or $i+1$ in the first row  to $i$.
\end{lemma}

For example, 
$ \sigma_4 \( \ytab{ \none & \none & 5' &5 \\ \none & 4 & 4 & 5'  & 5 \\ 3& 3 & 4' & 4 & 5' &5 & 5 }\) = \ytab{ \none & \none & 5 &5 \\ \none & 4' & 4 & 5'  & 5 \\ 3& 3 & 4' & 4 & 4 &4& 4 } . $

\begin{proof}
Recall that the domain of $\Tlowest_\lambda|_{[i,i+1]}$ consists of the two ribbons 
\[\SD_{(\lambda_{n-i+1},\lambda_{n-i+2},\dots) / (\lambda_{n-i+2},\lambda_{n-i+3},\dots)}
\quand
 \SD_{(\lambda_{n-i},\lambda_{n-i+1},\dots) / (\lambda_{n-i+1},\lambda_{n-i+2},\dots)}, \] 
which have all entries in
  $\{i',i\}$ and  $\{i+1' , i+1\}$, respectively.
  The integer $k := \varphi_i(T) - \varepsilon_i(T)  = \weight(T)_i - \weight(T)_{i+1}$ is therefore equal to $\lambda_{n-i+1}-\lambda_{n-i}< 0$
so $\sigma_i(T) = e_i^{-k}(T)$.

Compared to the proof of Lemma~\ref{s_i-highest},
it is less trivial but still straightforward to see that the sequence $\unpaired_i(T)$ from 
Definition~\ref{tab-pair-def} consists of all of the boxes with entries in $\{i+1',i+1\}$ in the first row of $T$.
There are $-k$ such boxes, only the first of which
 (going left to right) has a primed entry.
Applying $e_i$ to $T$ invokes case R2(d) in Definition~\ref{shtab-e},
causing this primed entry to change to $i$ and the primes on the entries in boxes $(j,j)$ and $(j+1,j+1)$ to be interchanged.
Successively applying $e_i$ to $e_i(T)$, $1-k$ additional times, changes the entries in the remaining unpaired boxes one by one from $i+1$ to $i$
via case R1(b) in Definition~\ref{shtab-e}. Thus $\sigma_i(T)$ is as described.
\end{proof}

\begin{lemma}\label{lem-qn+lowestwt}
Let $i \in [\ell(\lambda)]$ and 
suppose $T \in  \SShTab_n(\lambda)$ has $\unprime_\diag(T) = \Tlowest_\lambda$.
\ben

\item[(a)] If $T_{ii} \in \ZZ'$, then $e_{0}^{[i+n-\ell(\lambda)]}(T)$ is formed from $T$ by removing the prime from this entry. 

\item[(b)] If $T_{ii} \in \ZZ$, then $f_{0}^{[i+n-\ell(\lambda)]}(T)$ is formed from $T$ by adding a prime to this entry. 

\item[(c)] If $j \in [n-\ell(\lambda)]$, then $e_{0}^{[j]}(T)=f_{0}^{[j]}(T)=0$.
\een
\end{lemma}

\begin{proof}
The proof is  similar to the one given for Lemma~\ref{lem-qn+highestwt},
 just using Lemmas~\ref{i,inotinT} and \ref{s_i-lowest} in place of Lemma~\ref{s_i-highest}. We omit the details.
\end{proof}

Form $\TTlowest_\lambda$
 by adding a prime to each diagonal entry in $\Tlowest_\lambda$.
 
\begin{theorem}\label{highest-wt-thm}
The $\qq_n$-crystal $\SShTab_n(\lambda)$ is connected with unique $\qq_n$-highest 
and $\qq_n$-lowest weight elements given by 
 $\Thighest_\lambda$ and $\TTlowest_\lambda$ respectively.

\end{theorem}

\begin{proof}
It follows from \eqref{useful-eq1}  and  Lemma~\ref{lem-qn+highestwt}
that  $\Thighest_\lambda$ is the only element of $\SShTab_n(\lambda)$ that could be a $\qq_n$-highest weight. Since $e_0^{[i]}e_0^{[i]}=0$ for all $i \in [n]$ and since 
$\Thighest_\lambda = e_0^{[i]}f_0^{[i]}(\Thighest_\lambda)$ when $i \in [\ell(\lambda)]$,
it follows from Lemma~\ref{lem-qn+highestwt} that $e_0^{[i]}(\Thighest_\lambda) = 0$
for all $i \in [n]$, so $\Thighest_\lambda$ is the unique $\qq_n$-highest weight element and $\SShTab_n(\lambda)$ is connected.
A similar argument using \eqref{useful-eq2} and  Lemma~\ref{lem-qn+lowestwt} shows that 
$\TTlowest_\lambda$ is the unique $\qq_n$-lowest weight vector in $\SShTab_n(\lambda)$.
\end{proof}

\subsection{Dual equivalence operators}

Suppose $T$ is standard shifted tableau with $n$ boxes. For each $i \in [n]$, 
write $\square_i$ for the unique position 
of $T$ containing $i$ or $i'$, and define $\fks_i ( T)$ 
to be the shifted tableau formed from $T$ as follows:
\begin{itemize}

\item If $\square_i$ and $\square_{i+1}$ are in the same row or  column,
then reverse the primes on the entries of whichever of  these positions is off the diagonal;
then, if both $\square_{i-1}$ and $\square_{i+1}$ (respectively, $\square_i$ and $\square_{i+2}$) are on the diagonal when $i-1 \in [n]$
(respectively, $i+2 \in [n]$),
and their entries are not both primed or both unprimed, also 
reverse the primes on these entries.

\item Otherwise, swap $i$ with $i+1$ and  $i'$ with $i'+1$.
\end{itemize}
Thus we would have 
\[
\fks_6\(\ \ytab{
\none & \ & 6 & 7'  \\ 
\ & \ & \ & \
}\ \)
=
\ytab{
\none & \ & 6' & 7  \\ 
\ & \ & \ & \
}
\quand
\fks_5\(\ \ytab{
\none & \ & 5 &  \ \\ 
\ & \ & \ & 6'
}\ \)
=
\ytab{
\none & \ & 6 &  \ \\ 
\ & \ & \ & 5'
}
\]
as well as 
\[
\fks_4\(\ \ytab{
\none & \none &6' & \ \\ 
\none & 4' & 5 & \ \\ 
\ & \ & \ & \
}\ \)
=
\ytab{
\none & \none &6' & \ \\ 
\none & 4' & 5' & \ \\ 
\ & \ & \ & \
}
\quand
\fks_5\(\ \ytab{
\none & \none &6 & \ \\ 
\none & 4' & 5' & \ \\ 
\ & \ & \ & \
}\ \)
=
\ytab{
\none & \none &6' & \ \\ 
\none & 4 & 5 & \ \\ 
\ & \ & \ & \
}.
\]
Next, for each $i \in \ZZ$,  we construct a shifted tableau $\fkd_i(T)$ from $T$ as follows.
If $i\in\{-1,0\}$ and $i+2\in[n]$, then form
$ \fkd_i(T) $ from $T$ by swapping  $i+2$ with $i+2'$. If $i \in [n-2]$, then  set
\[\fkd_i(T) :=  \begin{cases}
\fks_{i} ( T) &\text{if $i+2$ is between $i$ and $i+1$ in $\shword(T)$}
\\
\fks_{i+1} ( T) &\text{if $i$ is between $i+1$ and $i+2$ in $\shword(T)$}
\\T &\text{if $i+1$ is between $i$ and $i+2$ in $\shword(T)$.}
\end{cases}\]
For integers $i$ with $i+2 \notin [n]$ define $\fkd_i(T) := T$. 
Here are a few properties of these operators:

\begin{proposition}[{See \cite[\S3.5]{M2021a}}] \label{primes-prop}
Let $T$ be a standard shifted tableau with $n$ boxes. 
For $j \in [n]$ let $\square_j$ be the unique box of $T$ containing $j$ or $j'$.
Fix $i \in [n-2]$. Then the following holds:
\ben
\item[(a)] $\fkd_i(\fkd_i(T)) =\fkd_{-1}(\fkd_{-1}(T)) =\fkd_0(\fkd_0(T)) =T$ and $\unprime_{\diag}(\fkd_i(T)) = \fkd_i(\unprime_{\diag}(T))$.

\item[(b)] $\fkd_i(T)$ only differs from $T$ in its entries in positions $\square_i$, $\square_{i+1}$, and $\square_{i+2}$.


\item[(c)] If $\square_i$ and $\square_{i+2}$ are not both on the diagonal, then $ \fkd_i(T) $ and $T$ have the same number of primed entries,
and the diagonal positions that are primed in $\fkd_i(T)$ are the same as in $T$.

\item[(d)] If $\square_i$ and $\square_{i+2}$ are both on the diagonal, then 
the number of primed entries in $ \fkd_i(T) $ and $T$ differ by one.
In this case, if the entries in $\square_i$ and $\square_{i+2}$ are both primed or both unprimed, then
$\fkd_i(T) $ is formed from $T$ by reversing the prime on the entry in just $\square_{i+1}$,
and otherwise 
$\fkd_i(T) $ is formed from $T$ 
by reversing the primes on the entries in $\square_i$, $\square_{i+1}$, and $\square_{i+2}$.

\een
\end{proposition}

The \defn{standardization}  of a semistandard shifted tableau $T$ is 
given as follows. List the boxes of $T$ in the order such that
one box comes before another if its entry is weakly smaller and the letter it contributes to $\shword(T)$
appears first going left to right. Then form $\std(T)$ from $T$ by changing the entry in the 
$i$th box to $i'$ if   primed and to $i$
otherwise. For example, \[
\std\(
\ytab{\none & \none & \none & 9' & 9  \\
\none & \none & 5 & 8 & 8 & 9 \\
\none & 4'  & 5' & 6 & 6 & 8 & 9' \\
2 &  4'  & 4 & 4 & 5 & 6' & 6}
\)
=
\ytab{\none & \none & \none & 17' & 18  \\
\none & \none & 7 & 13 & 14 & 19 \\
\none & 3'  & 6' & 10 & 11 & 15 & 16' \\
1 &  2'  & 4 & 5 & 8 & 9' & 12}.
\]
The operations  $\unprime_\diag$ and $\std$ commute.

When $T$ is standard with $n$ boxes,
a number $i \in [n-1]$ is a \defn{descent} of $T$ 
if $i+1$ is before $i$ in $\shword(T)$. 
Let $\Des(T)$ be the set of descents of $T$.
One can check that
 $i \in[n-1]$ is in $\Des(T)$ if and only if  
 (a) $i$ and $i+1$ both appear in $T$ with $i+1$ in a row strictly after $i$,
(b) $i'$ and $i+1'$ both appear in $T$ with $i+1'$ in a column strictly after $i'$,
or (c) $i$ and $i+1'$ both appear in $T$.

Below is another technical result to be used in Section~\ref{morphisms-subsect2};
compare with Lemma~\ref{f0-incr-lem}.

\begin{lemma}\label{f0-shtab-lem}
Suppose $T$ is a semistandard shifted tableau. Let $q := \weight(T)_1$ 
and 
\[U :=\begin{cases} \std(T) &\text{if }q \leq 1 \\  \fkd_{q-2}\cdots \fkd_1\fkd_0(\std(T)) &\text{if }q\geq 2.\end{cases}\]
 If  $q = 0$ or if $\weight(T)_2 \neq 0$ and $q \in \Des(U)$, then $f_{\bar 1}(T) = 0$,
and otherwise   $\std(f_{\bar 1}(T)) = U$.
\end{lemma}

\begin{proof}
If $q=0$, then there are no entries equal to $1'$ or $1$ in $T$, so $f_{\bar 1}(T) = 0$ by Definition~\ref{tab-qf-def}.
Suppose $q=1$. Then 
$T_{11}=U_{11} \in \{1',1\}$, and if $\weight(T)_2 \neq 0$, then 
  $T_{12} =U_{12} \in \{2',2\}$.
Thus if $\weight(T)_2 \neq 0$, then we can only have $q \in \Des(U)$ if $T_{12}  =U_{12}= 2'$
in which case $f_{\bar 1}(T) = 0$.
If $\weight(T)_2= 0$ or if $q \notin\Des(U)$, in which case $T_{12} = U_{12}=2$,
 then $f_{\bar 1}(T)$ is formed from $T$ by replacing entry $T_{11}$ with $1+T_{11}$,
 and this tableau also has standardization $U$ as claimed.

Suppose $q\geq 2$ so that $U =  \fkd_{q-2}\cdots \fkd_1\fkd_0(\std(T))$. The shifted tableau $T$ 
contains at most one entry equal to $1'$, which can only appear in the $(1,1)$ position, and all $1$'s in $T$ appear in the first row. Therefore  $\shword(\std(T))$ 
contains
$  12345\dots q$ as a consecutive subword. Applying $\fkd_0$ to $\std(T)$  changes the unique entry $2$ to $2'$, so the shifted reading word of $\fkd_0(\std(T))$
contains $  2 1345\dots q  $ as a (not necessarily consecutive) subword.
It follows that $\fkd_1$ applied to $\fkd_0(\std(T))$ acts as $\fks_{2}$, which reverses the primes on entries $2'$  and $3$. The shifted reading word of $\fkd_1\fkd_0(\std(T))$ therefore contains 
$  3 1245\dots q  $ as a subword, so $\fkd_2$ applied to $\fkd_1\fkd_0(\std(T))$ acts as $\fks_{3}$, which reverses the primes on  entries $3'$ and $4$.
Continuing in this way, we deduce that $U$ is formed from $\std(T)$ by simply adding a prime to entry $q$,
which is contained in the off-diagonal position $(1,q)$. 

Now we are ready to prove the last part of the lemma. 
If $\weight(T)_2 = 0$, then $f_{\bar 1}(T)$ is formed from $T$ by changing entry $T_{1q}$ from $1$ to $2'$
in which case we have $ \std(f_{\bar 1}(T)) = U$ as claimed.
Assume
$\weight(T)_2 \neq 0$.
Then 
$q \in \Des(U)$ if and only if $q+1$ appears before $q$ in $\shword(U)$,
which occurs if and only if the first row of $T$ contains an entry equal to $2'$, which would have to occur in position $(1,q+1)$.
Thus if $q \in \Des(U)$, then  $f_{\bar 1}(T) = 0$ by Definition~\ref{tab-qf-def},
while 
if $q \notin \Des(U)$, then 
applying $f_{\bar 1}$ to $T$ again
changes 
entry $T_{1q}$ from $1$ to $2'$, in which case $\std(f_{\bar 1}(T)) = U$.
\end{proof}

\section{Crystal morphisms}\label{morphisms-sect}

Continue to fix a positive integer $n$.
Below, we describe several morphisms between the families of $\qq_n$-crystals introduced above.
Specifically, we explain how  each crystal of words $\cW^+_n(m) \cong (\BB_n^+)^{\otimes m} $ may be embedded in a crystal of factorizations $\ORF_n(z)$
and how each crystal $\ORF_n(z)$ may be embedded in a union of shifted tableau crystals $\SShTab_n(\lambda)$.
This will allow us to prove Theorem~\ref{main-thm} from the introduction and to
 show that $\ORF_n(z)$ and $\SShTab_n(\lambda)$ are always normal $\qq_n$-crystals.
 
\subsection{From words to increasing factorizations}\label{morphisms-subsect1}

Let $p \in \ZZ$. Then $\iR^+(s_p) = \{p',p\}$ where $s_p = (p,p+1) \in S_\ZZ$.
The following is an easy exercise:

\begin{proposition}\label{word-factor-prop1}
The standard $\qq_n$-crystal $\BB_n^+ $ is isomorphic to $\ORF_n(s_p)$
via the map 
that sends 
$\boxed{i} \mapsto (\emptyset,\dots,\emptyset,p,\emptyset,\dots,\emptyset)$
and 
$\boxed{i'} \mapsto (\emptyset,\dots,\emptyset,p',\emptyset,\dots,\emptyset)$
where in both $n$-tuples all 
but the $i$th terms are empty words.
\end{proposition}

Fix positive integers $M$ and $N$, define $I_N := \{z \in I_\ZZ : z(i) = i\text{ for all }i \in \ZZ \setminus[N]\}$,
and choose involutions $y \in I_M$ and $z \in I_N$.
Let $y \oplus  z\in I_{M+N}$
 be the permutation mapping $i\mapsto y(i)$ for $i \in [M]$ and $i+M \mapsto z(i) + M$ for $i \in [N]$.
 In this setup, the following holds:
 
\begin{proposition}\label{word-factor-prop2}
The 
$\qq_n$-crystal $\ORF_n(y) \otimes \ORF_n(z)$ is isomorphic to 
$\ORF_n(y\oplus z)$
via the map $a \otimes b \mapsto (a^1\underline b^1,a^2\underline b^2,\cdots a^n\underline b^n)$
where $\underline b^i$ is the word formed by adding $M$ to every letter of $b^i$.
\end{proposition}

\begin{proof}
Denote the given map $\ORF_n(y) \otimes \ORF_n(z) \to \ORF_n(y\oplus z)$ by $\Phi$. Verifying  that $\Phi$ is a weight-preserving bijection is a standard exercise using
the discussion in Section~\ref{iw-sect}.
One can check that $\Phi$ commutes with the operators $e_{\bar 1}$, $f_{\bar 1}$, $e_0$, and $f_0$ by inspecting the relevant formulas in 
Theorem~\ref{qq-thmdef} and Section~\ref{incr-f-sect}.

Fix $i \in [n-1]$.
It remains to show that $\Phi$ commutes with $e_i$ and $f_i$.
We will just demonstrate that $\Phi \circ f_i = f_i\circ \Phi$ since the argument for $e_i$ is similar. 
Suppose $a \in \ORF_n(y)$ and $b \in \ORF_n(z)$.
First consider the unpaired letters in $a^i$ and $a^{i+1}$ 
relative to $\pair(a^i,a^{i+1})$ as well as the unpaired letters in
$b^i$ and $ b^{i+1}$ relative to $\pair(b^i,b^{i+1})$ as given in Definition~\ref{pair-def}. Notice that  the value of $\varphi_i(a) $ is the number of unpaired letters in $a^i$ while the value of $\varepsilon_i(a) $ is the number of unpaired letters in $a^{i+1}$. A similar description applies to $\varphi_i(b) $ and $\varepsilon_i(b) $.

We now turn to the unpaired letters in  $a^i\underline b^i$ and $a^{i+1}\underline b^{i+1}$ relative to $ \pair( a^i\underline b^i, a^{i+1}\underline b^{i+1})$.
Any letters in $a^{i+1}$ that were unpaired in the $(a^i,a^{i+1})$-pairing are 
now matched with letters in $\underline b^i$ that arise as shifts of unpaired letters in the $(b^i,b^{i+1})$-pairing.

It follows that if $\varepsilon_i(a) < \varphi_i(b)$ so that $f_i(a\otimes b)=a \otimes f_i(b)$, then the last unpaired letter $x \in a^i\underline b^i$ is just $M$ plus the last unpaired letter in $ b^i$.
If there is no such unpaired letter, then $\Phi(f_i(a\otimes b)) = f_i(\Phi(a\otimes b)) =0$.
Otherwise, applying $f_i$ to $\Phi(a\otimes b)$ 
removes $x$ from $a^i\underline b^i$ and adds a new letter $y\geq x$ to $a^{i+1} \underline b^{i+1}$, and it is easy to see that 
the result is the same as applying $\Phi$ to $a \otimes f_i(b) = f_i(a\otimes b)$.
Similarly, if $\varepsilon_i(a) \geq \varphi_i(b)$,
 then the last unpaired letter in $a^i\underline b^i$ is the same as the last unpaired letter in $a^i$,
and applying $f_i$ to $\Phi(a\otimes b)$ gives the same result as applying $\Phi$ to $f_i(a) \otimes b = f_i(a\otimes b) $.
Thus $\Phi \circ f_i = f_i\circ \Phi$ as needed.
 \end{proof}

\begin{corollary}\label{word-factor-cor}
There is a $\qq_n$-crystal isomorphism $\cW^+_n(m)  \cong \ORF_n(s_2s_4s_6\cdots s_{2m})$.
Thus each connected normal $\qq_n$-crystal is isomorphic to a full subcrystal of $\ORF_n(z)$ for some $z \in I_\ZZ$.
\end{corollary}

\begin{proof}
The first claim follows by induction from 
Proposition~\ref{word-factor-prop2}, taking $y =s_2s_4\cdots s_{2m-2}$, $z=s_1$, $M=2m-1$, and $N=2$,
with the $m=1$ base case provided by Proposition~\ref{word-factor-prop1}.
The second claim follows from the first claim
since
 $(\BB_n^+)^{\otimes m} \cong \cW^+_n(m)$ for all $m \in \NN$.
\end{proof}

\subsection{From increasing factorizations to shifted tableaux}\label{morphisms-subsect2}

Fix an involution $z \in I_\ZZ$.
We now wish to relate the $\qq_n$-crystals $\ORF_n(z)$ and $\SShTab_n(\lambda)$.
We will do this by making use of a correspondence between increasing factorizations
and pairs of shifted tableaux described in \cite[\S3]{M2021a}.
Recall that if $i \in \ZZ$, then $i' := i-\frac{1}{2}$ so $\lceil i'\rceil =\lceil i\rceil= i$.

\begin{definition}[{See \cite[\S3]{M2021a}}]
\label{oeg-def}
Suppose  $a \in  \Incr_n^+(z)$ and $w=w_1w_2\cdots w_m =\concat(a)$.
Let $\emptyset = T_0,T_1,\dots,T_m$ be the sequence of shifted tableaux
in which
 $T_i$ for $i \in [m]$ is formed by inserting $w_i$ into $T_{i-1}$ according to the following procedure:
\ben
\item Start by inserting $w_i$ into the first row.
At each stage, an entry $x$ is inserted into a row or column.
Let $y$ and $\tilde y$ be 
the first entries in the row or column with $\lceil x\rceil \leq \lceil y\rceil $
and $\lceil x\rceil < \lceil \tilde y \rceil $.

\item If no such entries $y$ and $\tilde y$  exist, then $ x $ is added to the end of the row or column,
with the exception that if $x$ is added to the main diagonal, then its value is changed to $\lceil x \rceil$.
We say the process to form $T_i$ \defn{ends in column insertion} if we are inserting into a column at this stage or 
if $ \lceil x \rceil \neq x$  is added to the main diagonal. Otherwise, the process \defn{ends in row insertion}.
 
\item If $y $ and $ \tilde y$ are distinct, then the primes on these entries are interchanged
and $x+1$ is inserted into the next row (if we were inserting into a row and $y$ is not on the main diagonal) or the next column (if we were  inserting into a column or $y$ is on the main diagonal).

\item If $y=\tilde y$ is off the main diagonal, then $x$ replaces $y$ and we 
insert $y$ into
the next row (if we were inserting into a row) or the next column (if we were inserting into a column).
 If $y=\tilde y$ is on the main diagonal, then $\lceil x \rceil $ replaces $y$ and
we insert $y - (\lceil x \rceil  - x)$ into the next column.
\een
Finally define $\PO(a) := T_m$ and construct
$\QO(a)$ as the shifted tableau with the same shape as $\PO(a)$
that contains $j$ (respectively, $j'$) in the box added to $T_{i-1}$ to form $T_i$
if $w_i$ is in the $j$th factor of $a$ and the insertion process ends in row insertion (respectively, column insertion).
\end{definition}

\begin{example}\label{oeg-ex1}
If $a = (4,1'35,\emptyset,4',\emptyset, 2)$, then $\PO(a)$ and $\QO(a)$ are computed as follows:
\[
\ba 
\emptyset \ \overset{4}{\leadsto}\  \ytab{4} \ \overset{1'}{\leadsto}\   \ytab{1 & 4'}   \ \overset{3}{\leadsto}\  \ytab{ \none & 4 \\ 1 & 3} \ \overset{5}{\leadsto}\  \ytab{ \none & 4 \\ 1 & 3 & 5}
\ \overset{4'}{\leadsto}\  \ytab{ \none & 4 & 5 \\ 1 & 3 & 4'} \ \overset{2}{\leadsto}\  \ytab{ \none & 3 & 5' \\ 1 & 2 & 4 & 5} &= \PO(a), \\
\emptyset \ \overset{\phantom 1}{\leadsto}\  \ytab{1} \ \overset{\phantom {2'}}{\leadsto}\  \ytab{1 & 2'}   \ \overset{\phantom {2'}}{\leadsto}\   \ytab{ \none & 2' \\ 1 & 2'} \ \overset{\phantom 2}{\leadsto}\   \ytab{  \none & 2' \\ 1 & 2' & 2}
\ \overset{\phantom 4}{\leadsto}\   \ytab{   \none & 2' & 4 \\ 1 & 2' & 2} \ \overset{\phantom {6'}}{\leadsto}\  \ytab{  \none & 2' & 4 \\ 1 & 2' & 2 & 6'} &= \QO(a).
\ea
\]
We make no distinction between 
$w_1w_2\cdots w_n$ and the  sequence of $1$-letter words $(w_1, w_2, \dots, w_n)$.
This lets us view each $w \in \iR^+(z)$  as an element of $\ORF_n(z)$ for $n=\ell(w)$
so we can evaluate $\PO(w)$ and $\QO(w)$. If $w =41'354'2$, then $\PO(w) = \PO(a)$ but 
$ \QO(w) =  \ytab{  \none & 3' & 5 \\ 1 & 2' & 4 & 6'} \neq \QO(a)$.
\end{example}

The map $a \mapsto (\PO(a), \QO(a))$ is called \defn{orthogonal Edelman-Greene insertion} in \cite{M2021a}.
It is a shifted version of the \defn{Edelman-Greene correspondence} from \cite{EG},
as well as the ``orthogonal'' counterpart
to a ``symplectic'' insertion algorithm
studied in \cite{Hiroshima,Marberg2019a,Marberg2019b}.
Restricted to the subset $\iR(z) \subsetneq \iR^+(z)$,
the map  is
a special case of \defn{shifted Hecke insertion} from 
\cite{PatPyl2014}. 

A (shifted) tableau is \defn{increasing} if its rows and columns are strictly increasing.
The \defn{row reading word} $\row(T)$ of a (shifted) tableau $T$ is formed by reading its rows from left to right,
but starting with the top row in French notation.

One important feature of 
the map
$a \mapsto (\PO(a), \QO(a))$
is that it is a bijection from
$ \Incr_n^+(z)$ to the set of
pairs $(P,Q)$ of shifted tableaux of the same shape, in which $Q$ is semistandard with all entries at most $n$, 
and $P$ is an increasing 
with no primes on the main diagonal and
$\row(P) \in \iR^+(z)$ \cite[Thm.~3.15]{M2021a}.
%
%
For our applications, we need a few other technical properties from \cite{M2021a}:

\begin{lemma}[{See \cite[\S3]{M2021a}}] \label{qo-lems}
The following holds for all $a \in \ORF_n(z)$ and $w \in \iR^+(z)$:
\ben
\item[(a)] $\weight(a) = \weight(\QO(a))$ and $\Des(w) = \Des(\QO(w))$.

\item[(b)] Box $(1,1)$ of $\QO(a)$ is primed if and only if the first letter of $\concat(a)$ is primed.

\item[(c)] Each $T \in \bigsqcup_{\text{strict partitions } \lambda} \SShTab_n(\lambda)$ occurs as 
$\QO(a)$ for some $z \in I_\ZZ$ and $ a \in \ORF_n(z)$.

\item[(d)] $\PO(\unprime(a)) = \unprime(\PO(a))$ and $ \QO(\unprime(a)) = \unprime_\diag(\QO(a)).$

\item[(e)] $\PO(\concat(a)) = \PO(a)$ and $ \QO(\concat(a)) = \std(\QO(a)).$

\item[(f)] $\PO(\ck_i(w)) = \PO(w)$ and $ \QO(\ck_i(w)) = \fkd_i(\QO(w))$ for all $i \in \ZZ$.
\een
\end{lemma}

\begin{proof}
   Properties (b) and (e) and the identity $\weight(a) = \weight(\QO(a))$ are clear from Definition~\ref{oeg-def}.
 The claim that $\Des(w) = \Des(\QO(w))$ is \cite[Prop.~3.13]{M2021a}.
Properties (d) and (f) are  \cite[Prop.~3.8]{M2021a} and \cite[Thm.~3.24]{M2021a},
while property (c) follows from \cite[Thm.~3.15 and Lem.~3.17]{M2021a}.
\end{proof}


\begin{example}\label{oeg-ex}
Let $a = (4,1'35, \emptyset, 4',\emptyset, 2)$ and $w =  \concat(a) = 41'354'2$ as in Example~\ref{oeg-ex1}.
Then $a \in \ORF_n(z)$ and $w \in \iR^+(z)$ for $z = (1,3)(2,6)(4,5) \in I_\ZZ$ and $n=6$.
\ben
\item[(a)] One has $\weight(\QO(a)) = (1,3,0,1,0,1) = \weight(a)$ and $\Des(\QO(w)) = \{1,4,5\} = \Des(w)$.
\item[(b)] Box $(1,1)$ of $\QO(a)$ is not primed since the first letter of $w$ is not primed.
\item[(d)] One has $\PO( ( 4,135, \emptyset, 4,\emptyset, 2)) =  \ytab{ \none & 3 & 5 \\ 1 & 2 &4 &5}$ and $\QO((4,135, \emptyset, 4,\emptyset, 2)) = \ytab{\none & 2 & 4 \\ 1 & 2' & 2& 6'} $.
\item[(e)] It holds that $\PO(w) =\PO(a)$ and $\QO(w) = \std(\QO(a))$.
\item[(f)] If $u = 14'354'2 = \ck_0(w)$ and $v = 1'4354'2 = \ck_1(w)$ then $\PO(u) =\PO(v) =\PO(w)$ while 
\[\QO(u) = \ytab{\none & 3' & 5 \\ 1 & 2 & 4& 6'} =\fkd_0(\QO(w))
\quand \QO(v) = \ytab{\none & 3 & 5 \\ 1' & 2 & 4& 6'} =\fkd_1(\QO(w)).\]
\een
\end{example}

Most of the subtlety of orthogonal Edelman-Greene insertion has to do with the distribution of primed entries in the output tableaux.
When primes are disallowed, 
we already have some strong results that relate this algorithm to 
the relevant $\q_n$-crystal structures:

\begin{theorem}[{\cite[Thm. 3.32]{Marberg2019b}}]
\label{qo-morphism-thm}
The map 
$a\mapsto \QO(a)$ is  a quasi-isomorphism of $\q_n$-crystals
$
\bigsqcup_{z \in I_\ZZ} \Incr_n(z) \to \bigsqcup_{\text{strict partitions }\lambda \in \NN^n}  \ShTab_n(\lambda),$ 
and the full $\q_n$-subcrystals of $\bigsqcup_{z \in I_\ZZ} \Incr_n(z) $ are the subsets on which $a \mapsto \PO(a)$ is constant.
\end{theorem}

We will show that this statement extends to $\qq_n$-crystals. The proof requires a lemma.

\begin{lemma}\label{key-lem}
Let $T\in \ShTab^+_n(\lambda)$.
Suppose $k \in [n-1]$ and $f_k(T) \neq 0$. Define 
\[ M := \weight(T)_1 + \weight(T)_2 + \dots + \weight(T)_{k-1} + 1
\quand N := \weight(T)_1 + \weight(T)_2 + \dots + \weight(T)_{k+1} .\]
Then there are indices $j_1,\dots,j_l \in [M,N-2]$ with 
$ \std(f_k(T)) = \fkd_{j_l} \cdots  \fkd_{j_1}(\std(T))$.
\end{lemma}

\begin{proof}
Let $\underline T := \unprime_\diag(T)$. 
Then $f_k(\underline T) = \unprime_\diag(f_k(T)) \neq 0$ by Lemma~\ref{tab-unprime-lem}.
By properties  (c) and (d) in Lemma~\ref{qo-lems},
there exists  $z \in \I_\ZZ$ and $ a \in \Incr_n^+(z)$ with $\QO(a) = T$, and if $\underline a := \unprime(a)$,
then  $\QO(\underline a) = \underline T$. We must have  $f_k(\underline a) \neq 0$ since $f_k(\underline T) \neq 0$
by Theorem~\ref{qo-morphism-thm},
so Lemma~\ref{ock-key-lem} implies that there are integers $j_1,j_2,\dots,j_l \in [M,N-2]$ with
\[ \concat(f_k(\underline a)) = \ck_{j_l} \cdots \ck_{j_2} \ck_{j_1}(\concat(\underline a)).\]
 Define $U:= \fkd_{j_l} \cdots \fkd_{j_2} \fkd_{j_1}(\std(T))$.
We argue that $ \std(f_k(T)) =U$. Applying $\QO$ to the left side of the previous 
displayed equation gives
\[\ba \QO(\concat(f_k(\underline a))) &= \std(\QO(f_k(\underline a))) 		&\quad\text{by Lemma~\ref{qo-lems}(e),}
\\&=\std(f_k(\underline T)) 										&\quad\text{by Theorem~\ref{qo-morphism-thm},}
\\&= \std(\unprime_\diag(f_k(T)))									&\quad\text{by Lemma~\ref{tab-unprime-lem},}
\\&= \unprime_\diag(\std(f_k(T)))									&\quad\text{by definition},\ea
\]
while applying $\QO$ to the right side gives
\[\ba 
\QO(\ck_{j_l} \cdots \ck_{j_2} \ck_{j_1}(\concat(\underline a)))
 &= \fkd_{j_l} \cdots \fkd_{j_2} \fkd_{j_1}(\std(\underline T)) 
 =  \unprime_\diag(U)
 \ea\]
 by parts (e) and (f) of Lemma~\ref{qo-lems} and part (a) of Proposition~\ref{primes-prop}.
 Thus
\[
 \unprime_\diag(\std(f_k(T))) =  \unprime_\diag( U),\]
  so to prove that $ \std(f_k(T))=U$ it suffices to show that $\std(f_k(T))$ and $U$ share the same set of primed diagonal positions.
 Since 
 the same positions in $f_k(T)$ and $\std(f_k(T))$ have primed entries, it is enough to check
$f_k(T)$ and $U$ have the same primed diagonal positions.

For each $j \in [n]$ let $\square_j$ be the unique position containing $j$ or $j'$ in $\std(T)$.
 Then the domain of 
 the skew shifted tableau $T|_{[i,i+1]}$, which is a union of two rims,
 consists of precisely the boxes $\square_M,\square_{M+1},\dots,\square_N$.
 These boxes therefore contain at most two diagonal positions,
which must occur in consecutive rows.

Suppose there are less than two diagonal positions among  $\square_M,\square_{M+1},\dots,\square_N$.
Then it is clear from Definition~\ref{shtab-f}
that $
 f_k(T)$
has the same set of primed diagonal positions as $T$,
and it follows from Proposition~\ref{primes-prop}
that  $ U$
 has the same set of primed diagonal positions as  $\std(T)$.
As $\std(T)$ and $T $ have identical sets of primed positions,
the same diagonal positions in $f_k(T)$ and $U$ are primed as desired.
 This reasoning also applies if 
 $\square_M,\square_{M+1},\dots,\square_N$ include two diagonal positions
 but the entries of $T$ in these positions are both primed or both unprimed.
 
 The case left to consider is the following:
 assume  $\square_M,\square_{M+1},\dots,\square_N$ involve exactly two diagonal positions,
 say in rows $r-1$ and $r$, and exactly one of these positions has a primed entry in $T$.
 Let $\cD(T) := \{(i,i) \in T : T_{ii} \in \ZZ'\}$ be the set of primed diagonal positions in $T$ and set $\mathcal{S} := \{(r-1,r-1),(r,r)\}$.  
 Lemma~\ref{shtab-primed-lem} implies that
 \[ 
 \cD(f_k(T)) = \begin{cases} \cD(T) \mathbin{\triangle} \mathcal{S} &\text{if }\primes(f_k(T)) \not\equiv \primes(T) \modu 2), \\
 \cD(T) &\text{otherwise},
 \end{cases}
 \]
 where $\triangle$ denotes the symmetric set difference.
 On the other hand, it follows from Proposition~\ref{primes-prop} 
 that as we apply   $\fkd_{j}$ for $j=j_1,j_2,\dots,j_l$ successively
 to $\std(T)$, only the entries in positions $\square_M,\square_{M+1},\dots,\square_N$ vary.
In particular, the parity of the number of primed positions changes precisely when 
 $j$ or $j'$ appears in box $(r-1,r-1)$  
 and $j+2$ or $j+2'$
 appears in box $(r,r)$.
Proposition~\ref{primes-prop}(e)
tells us that if this happens, then applying $\fkd_j$ interchanges the primes on these diagonal boxes,
and otherwise  the set of primed diagonal positions is unchanged by $\fkd_j$. Thus
 \[ 
 \cD(U) = \begin{cases} \cD(\std(T)) \mathbin{\triangle}\mathcal{S} &\text{if }\primes( U) \not\equiv \primes(\std(T)) \modu 2), \\
 \cD(\std(T)) &\text{otherwise}.
 \end{cases}
 \]
 As $\cD(T) = \cD(\std(T))$ and $\primes(T) = \primes(\std(T))$, these formulas for $\cD(f_k(T))$ and $ \cD(U)$ give the same result,
  so $\cD(f_k(T)) = \cD(U)$ as needed.
 This lets us conclude that
$ \std(f_k(T)) =  U = \fkd_{j_l} \cdots  \fkd_{j_2} \fkd_{j_1}(\std(T))$.
 \end{proof}


\begin{theorem}\label{qqo-morphism-thm}
The map $a\mapsto \QO(a)$ is  a quasi-isomorphism of $\qq_n$-crystals
$
\bigsqcup_{z \in I_\ZZ} \Incr^+_n(z) \to \bigsqcup_{\text{strict partitions }\lambda \in \NN^n}  \SShTab_n(\lambda),$
and the full $\qq_n$-subcrystals of $\bigsqcup_{z \in I_\ZZ} \Incr^+_n(z) $ are the subsets on which $a \mapsto \PO(a)$ is constant.
\end{theorem}

This statement is a strict generalization of Theorem~\ref{qo-morphism-thm}.
Although  $ \Incr^+_n(z)$ and $ \SShTab_n(\lambda)$, viewed as $\q_n$-crystals, contain  $\Incr_n(z)$ and  $\ShTab_n(\lambda)$ as subcrystals,
the sets $ \Incr^+_n(z) \setminus  \Incr_n(z)$ 
and
$ \SShTab_n(\lambda) \setminus  \ShTab_n(\lambda)$
are not normal as $\q_n$-crystals.
Therefore even the weaker statement that $\QO$ gives a $\q_n$-crystal morphism in Theorem~\ref{qqo-morphism-thm}
cannot be deduced from Theorem~\ref{qo-morphism-thm}.

\begin{proof}
Let $z \in I_\ZZ$ and $a =(a^1,a^2,\dots,a^n) \in \Incr_n^+(z)$.
Property (b) in Lemma~\ref{qo-lems} implies that position $(1,1)$ of $\QO(a)$ contains $1$ 
(respectively, $1'$) if and only if the word $a^1$ is nonempty and begins with an unprimed (respectively, primed) letter.
Comparing Definitions~\ref{orf-def3} and \ref{shtab-f0},
we conclude that   $\QO(f_{0}(a)) = f_{0}(\QO(a)),$ interpreting $\QO(0) := 0$ and $f_i(0):=0$.

Properties (a) and (f) in Lemma~\ref{qo-lems} tell us that $\QO$ preserves weights and descent sets  and 
that $\QO\circ \ck_j = \fkd_j\circ \QO$.
Comparing Lemmas~\ref{f0-incr-lem} and \ref{f0-shtab-lem}, we see that
  $f_{\bar 1}(a) \neq 0$ if and only if $f_{\bar 1}(\QO(a)) \neq 0$,
in which case $\QO(f_{\bar 1}(a)) = f_{\bar 1}(\QO(a))$ since both tableaux have the same weight and same standardization.

Now suppose $i \in [n-1]$ and let $T = \QO(a)$.
As above, 
we abbreviate by setting $\underline a := \unprime(a)$ and $\underline T := \unprime_\diag(T)$
so
that $\QO(\underline a) = \underline T$
by Lemma~\ref{qo-lems}(d).
We have 
$f_i(a) = 0$ if and only if $f_i(T) = 0$
since 
Lemma~\ref{ef-ick-lem}, Theorem~\ref{qo-morphism-thm}, and Lemma~\ref{tab-unprime-lem}
imply the respective equivalences
$f_i(a) = 0$ $\Leftrightarrow$ $f_i(\underline a) = 0$ $\Leftrightarrow$ $ f_i(\underline T) = 0$ $\Leftrightarrow$
 $f_i(T) = 0$.

Assume $ f_i(a) \neq 0$ so that $ f_i(T) \neq 0$. 
Since orthogonal Edelman-Greene insertion is a weight-preserving bijection,
there is a unique $b \in  \Incr^+_n(z)$ with $\PO(a) = \PO(b)$ and $f_i(T) = \QO(b)$.
To show that $\QO(f_i(a)) = f_i(T)= f_i(\QO(a)) $ it suffices to prove that $f_i(a) = b$.
Lemma~\ref{ock-key-lem} and properties (e) and (f) of Lemma~\ref{qo-lems} imply
that $\PO(f_i(a)) = \PO(a)=\PO(b)$, so Lemma~\ref{qo-lems}(d) gives
$\PO(\unprime(f_i(a)))  = \PO(\unprime(b)).$ Likewise, we have
\[
\ba \QO(\unprime(f_i(a))) 
&= \QO(f_i(\underline a)) 			&\quad\text{by Lemma~\ref{ef-ick-lem},}
\\&= f_i (\QO(\underline a)) 				
								&\quad\text{by Theorem~\ref{qo-morphism-thm},}
\\&= f_i(\underline T) 			&\quad\text{since $\QO(\underline a)=\underline T$,}
\\&= \unprime_\diag(f_i(T) )  				&\quad\text{by Lemma~\ref{tab-unprime-lem},}
\\&= \QO(\unprime(b))  				&\quad\text{by Lemma~\ref{qo-lems}(d).}
\ea\]
We conclude that   $\unprime(f_i(a)) = \unprime(b)$ since $\bullet \mapsto (\PO(\bullet),\QO(\bullet))$ is a bijection. Now to prove that $f_i(a)=b$,
it is enough by Lemma~\ref{marked-lem0} to show that $\marked(f_i(a)) = \marked(b)$, i.e.,
that $f_i(a)$ and $b$ have the same set of marked cycles as defined at the start of Section~\ref{ck-sect}.

We know that $\marked(f_i(a)) = \marked(a)$ by Corollary~\ref{ock-key-cor}.
Let $v := \concat(a)$ and $w: = \concat(b)$. By Lemma~\ref{qo-lems}(e) we have
$  \QO(v) = \std(T)$ and $ \QO(w) = \std(f_i(T)).$
By Lemma~\ref{key-lem}, there are indices $j_1,j_2,\dots,j_l>0$ such that 
 \[\QO(w) = \std(f_i(T)) =  \fkd_{j_l} \cdots \fkd_{j_2} \fkd_{j_1}(\std(T))= \fkd_{j_l} \cdots \fkd_{j_2} \fkd_{j_1}(\QO(v)).\]
By Lemma~\ref{qo-lems}(f), it follows that 
 $ \QO(w) = \QO(\ck_{j_l} \cdots \ck_{j_2} \ck_{j_1}(v))$ and 
 \[  \PO(w)  = \PO(b) = \PO(a) = \PO(v) = \PO(\ck_{j_l} \cdots \ck_{j_2} \ck_{j_1}(v)).\]
 Thus $w=\ck_{j_l} \cdots \ck_{j_2} \ck_{j_1}(v) $, so by Lemma~\ref{marked-lem} and Corollary~\ref{ock-key-cor} we have
  $ \marked(f_i(a)) =\marked(a) = \marked(v) = \marked(w) = \marked(b).$
We conclude that  $f_i(a) = b$, so we have $\QO(f_i(a)) = f_i(\QO(a))$ as desired.

Thus  $\QO(f_i(a)) = f_i(\QO(a))$ for all $i \in \{ \bar 1, 0,1,\dots,n-1\}$ and $a \in \ORF_n(z)$.
It follows that if $e_i(a)\neq 0$, then $f_i(\QO(e_i(a))) = \QO(a)$ so $\QO(e_i(a)) = e_i(\QO(a))$.
Likewise, if $e_i(\QO(a)) \neq 0$, then  $e_i(\QO(a)) = \QO(b) $ for a unique $b \in \ORF_n(z)$ with $\PO(a) =\PO(b)$,
and then we have  \[\QO(a) = f_i(e_i(\QO(a))) = f_i(\QO(b)) = \QO(f_i(b)). \] 
As $\PO(f_i(b)) = \PO(b)$,
this can only hold if $a=f_i(b)$, in which case $e_i(a) = b \neq 0$.
Taking contrapositives, we deduce that if $e_i(a) = 0$, then  $e_i(\QO(a)) = 0$.
Hence $\QO(e_i(a)) = e_i(\QO(a))$ for all $i$ and $a$, interpreting $e_i(0):=0$,
so $a \mapsto \QO(a)$ is at least a $\qq_n$-crystal morphism.

Now 
let $P = \PO(a)$.
 If $\lambda$ is the shape of $P$,  
then $\QO$ defines a weight-preserving bijection $\{ b \in \ORF_n(z) : \PO(b) =P\} \to \SShTab_n(\lambda)$
that commutes with all crystal operators.
Since $\SShTab_n(\lambda)$ is a connected $\qq_n$-crystal by Theorem~\ref{highest-wt-thm},
all full $\q_n$-subcrystals of $\ORF_n(z)$ must be analogous subsets on which $\PO$ is constant,
so 
$\QO$ is a quasi-isomorphism of $\qq_n$-crystals.
 \end{proof}

\begin{corollary}
Let $\mu \subset \lambda$ be strict partitions such that $\SShTab_n(\lambda/\mu)$ is nonempty. Then $\unprime_\diag : \SShTab_n(\lambda/\mu) \to \ShTab_n(\lambda/\mu)$ is a quasi-isomorphism of $\gl_n$-crystals. 
\end{corollary}

\begin{proof}
Let $\Lambda_n$ be the set of strict partitions in $\NN^n$.
 The diagram
\[
\begin{tikzcd}
 \bigsqcup_{z \in I_\ZZ} \Incr^+_n(z) \arrow[d, "\QO"'] \arrow[rr, "\unprime"] &&  \bigsqcup_{z\in I_\ZZ} \Incr_n(z)  \arrow[d, "\QO"] \\
\bigsqcup_{\lambda\in\Lambda_n}  \SShTab_n(\lambda) \arrow[rr, "\unprime_\diag"] && \bigsqcup_{\lambda\in\Lambda_n}  \ShTab_n(\lambda)
\end{tikzcd}
\]
commutes and all four arrows are surjective maps. Since the top horizontal arrow and both vertical arrows 
are quasi-isomorphisms of $\gl_n$-crystals  by Corollary~\ref{unprime-cor2} and Theorems~\ref{qo-morphism-thm} and \ref{qqo-morphism-thm},
the bottom horizontal arrow must also be a quasi-isomorphism of $\gl_n$-crystals.

This prove the result when $\mu =\emptyset$. This suffices for the general case, since applying the functor $\sF$ from Corollary~\ref{skew-norm-cor} to a map
preserves the property of being a quasi-isomorphism,
and $\SShTab_n(\lambda/\mu)$ and $\ShTab_n(\lambda/\mu)$ may be identified with $\gl_n$-subcrystals 
of $\sF(\SShTab_{k+n}(\lambda))$ and $\sF(\ShTab_{k+n}(\lambda))$ for $k=\ell(\mu)$ with the former mapped onto the latter by $\unprime_\diag$.
\end{proof}

It was easy to describe a coarse decomposition of $\cW^+_n(m)$ and $\ORF_n(z)$ into $\gl_n$-subcrystals
on which the unpriming maps $\cW^+_n(m)\to \cW_n(m)$ and $\ORF_n(z)\to \Incr_n(z)$
are injective.
It does not seem
to be straightforward to do the same for $\unprime_\diag:\SShTab_n(\lambda)\to \ShTab_n(\lambda)$.
A harder open problem is to describe the full $\gl_n$-subcrystals of $\SShTab_n(\lambda)$.\footnote{The \defn{rectification} process in \cite[\S5]{AssafOguzAbstract}
would characterize 
the full $\gl_n$-subcrystals of the smaller object $\ShTab_n(\lambda)$. 
}

\subsection{From words to shifted tableaux}\label{last-sect}

Combining Sections~\ref{morphisms-subsect1} and \ref{morphisms-subsect2}
shows how to embed each crystal of words $\cW^+_n(m)$ in a union of 
shifted tableau crystals $\SShTab_n(\lambda)$. This will lead to a proof of Theorem~\ref{main-thm}.

Fix  $m,n \in \NN$ and
consider a word   $w=w_1w_2\cdots w_m \in \cW^+_n(m)$.
 Define $w^\top$ to be the $n$-tuple of strictly increasing primed words $a=(a^1,a^2,\dots,a^n)$
in which 
the unprimed letters in $a^i$ are the indices $j \in [m]$ with $w_j =i$,
and
 the primed letters in $a^i$ are 
given by adding primes to each $j \in [m]$  with $w_j =i'$. 
For example, if $n=3$ and $w=2'211'2'$, then $w^\top = (34',1'25',\emptyset)$.

Form $\double(w) $ by applying the map with $i\mapsto 2i$ and $i' \mapsto (2i)'  $ for $i \in \ZZ$ to the letters of $w$. 
 For a  tableau $T$, define $\double(T)$ by applying the same map to every entry.
For an $n$-tuple of primed words $a=(a^1,\dots,a^n)$,
let $\double(a) := (\double(a^1),\dots,\double(a^n)).$

One always has $\double(w^\top)\in \ORF_n(s_2s_4\cdots s_{2m})$.
In fact, $w \mapsto \double(w^\top)$ is precisely the isomorphism 
$\cW^+_n(m) \xrightarrow{\sim} \ORF_n(s_2s_4\cdots s_{2m})$
 in the proof of Corollary~\ref{word-factor-cor}. Hence if we define 
\be 
\PHM(w) := \QO(\double(w^\top)),
\ee
then $w  \mapsto \QO(\double(w^\top))$ is  
a quasi-isomorphism of $\qq_n$-crystals. 
We can make a more precise statement.
Each entry of $\PO(\double(w^\top))$ is in $\{ 2' < 2 < 4' < 4 < \dots < 2m' < 2m\}$
so there exists a unique shifted tableau $\QHM(w)$ such that 
\be
\double ( \QHM(w) )=  \PO( \double(w^\top)).
\ee
\begin{example} If $n=3$ and $w = 3'311'3$, then
 $\double(w^\top)= (6\ 8',\emptyset, 2'\ 4\ 10')$ and
  \[
  \PO( \double(w^\top)) = \ytab{\none & 6 \\ 2 & 4 & 8' & 10'}
  \quand
 \QO( \double(w^\top))= \ytab{\none & 3' \\ 1 & 1  & 3' &3} = \PHM(w).\]
 \end{example}
 
 The correspondence $w \mapsto (\PHM(w), \QHM(w))$
 extends Haiman's \defn{shifted mixed insertion algorithm} from  \cite{HaimanMixed}
and  is called \defn{orthogonal mixed insertion} in \cite{M2021a}.
It can be 
defined in a self-contained way via a certain insertion process (see \cite[Def.~5.17]{M2021a})
and gives a bijection from
$\cW^+_n(m)$
 to the set of
pairs $(P,Q)$ of shifted tableaux of size $m$ with the same shape, in which $P$ is semistandard with all entries at most $n$
and $Q$ is standard
with no primes on the diagonal \cite[Prop.~5.4 and Thm.~5.21]{M2021a}.
The following is clear from 
the observations above and Theorem~\ref{qqo-morphism-thm}: 

\begin{corollary}\label{qq-o-morphism-cor}
The map $w\mapsto \PHM(w)$ is  a quasi-isomorphism of $\qq_n$-crystals
$
\bigsqcup_{m \in \NN} \cW^+_n(m) \to \bigsqcup_{\text{strict partitions }\lambda \in \NN^n} \SShTab_n(\lambda),$
and the full $\qq_n$-subcrystals of $\bigsqcup_{m \in \NN} \cW^+_n(m) $ are the subsets on which $w \mapsto \QHM(w)$ is constant.
\end{corollary}

The map $w\mapsto \PHM(w)$ restricts to a quasi-isomorphism 
 $\bigsqcup_{m \in \NN} \cW_n(m) \to \bigsqcup_{\lambda} \ShTab_n(\lambda)$ of $\q_n$-crystals
 by results in \cite{AssafOguz,HPS,Hiroshima2018}; this is explained in the proof sketch for \cite[Thm.-Def. 2.12]{Marberg2019b}.

\begin{corollary}\label{any-normal-cor}
A connected normal $\qq_n$-crystal 
has a unique $\qq_n$-highest weight element. The weight of this element is a strict 
partition $\lambda \in \NN^n$ and  the crystal 
is isomorphic to $\SShTab_n(\lambda)$.
\end{corollary}

\begin{proof}
Corollary~\ref{qq-o-morphism-cor} tells us that any connected normal $\qq_n$-crystal
is isomorphic to a full $\qq_n$-subcrystal of $\SShTab_n(\lambda)$ for some $\lambda \in \NN^n$,
so the result follows from Theorem~\ref{highest-wt-thm}.
\end{proof}

Suppose $\cB$ is a $\qq_n$-crystal. Recall the formula for $\sigma_{w_0}$ from \eqref{w0-eq} and
define $\sigma_{w^+_0} : \cB \to \cB$ by
\be \sigma_{w^+_0} :=  (\sigma_0) (\sigma_1\sigma_0)(\sigma_2\sigma_1\sigma_0) \cdots  (\sigma_{n-1} \cdots \sigma_1\sigma_0).\ee
Like $\sigma_{w_0}$, this operator is weight-reversing.
If $\cB$ is normal, then Theorem~\ref{bc-action-prop} implies that $\sigma_{w^+_0}$ is an involution which gives the action on $\cB$ of the
 element $w^+_0 \in W^{\mathsf{BC}}_n$ sending $i \in [n]$ to $i-1-n$.

\begin{proposition}\label{wplus-prop}
Suppose $\cB$ is a normal $\qq_n$-crystal. Then 
$b\in\cB$ is a $\qq_n$-lowest weight element if and only if $\sigma_{w^+_0}(b) \in \cB$ is a $\qq_n$-highest weight element.
\end{proposition}

\begin{proof}
By Corollary~\ref{any-normal-cor} we may assume that $\cB = \SShTab_n(\lambda)$ for some strict partition $\lambda \in \NN^n$.
Because $\sigma_{w^+_0}$ is an involution and because 
$ \SShTab_n(\lambda)$ has  unique $\qq_n$-highest and $\qq_n$-lowest weight elements $\Thighest_\lambda$ and
$\TTlowest_\lambda$ by Theorem~\ref{highest-wt-thm},
it suffices to show that $\sigma_{w^+_0}(\TTlowest_\lambda) = \Thighest_\lambda$.

Let
$ \sigma_{w^\mathsf{BC}_0} := (\sigma_0) (\sigma_1\sigma_0\sigma_1)(\sigma_2\sigma_1\sigma_0\sigma_1\sigma_2) \cdots  (\sigma_{n-1} \cdots \sigma_1\sigma_0\sigma_1\cdots \sigma_{n-1}).$ This operator gives the action of the longest element $w^{\mathsf{BC}}_0 \in W^{\mathsf{BC}}_n$ on $\cB$, 
and $\sigma_{w^+_0} = \sigma_{w_0} \sigma_{w^\mathsf{BC}_0}$.
Lemma~\ref{lem-qn+lowestwt} implies that 
$ \sigma_{w^\mathsf{BC}_0}(\TTlowest_\lambda) = e_0^{[n-\ell(\lambda)+1]}\cdots e_0^{[n-2]}e_0^{[n-1]}e_0^{[n]}(\TTlowest_\lambda)
= \Tlowest_\lambda \in  \ShTab_n(\lambda).$

Recall from Sections~\ref{tab-op-sect} and \ref{tab-highest-sect} that 
the raising and lowering operators of $ \SShTab_n(\lambda)$ give a normal $\q_n$-crystal structure on $ \ShTab_n(\lambda)$
with unique $\q_n$-highest weight element $ \Thighest_\lambda$. Thus, it follows from 
Proposition~\ref{low-prop} that $\sigma_{w^+_0}(\TTlowest_\lambda) = \sigma_{w_0}(\Tlowest_\lambda) = \Thighest_\lambda$ as needed.
\end{proof}


The only claim in Theorem~\ref{main-thm} left to prove is that for each strict partition $\lambda \in \NN^n$
there is a connected normal $\qq_n$-crystal with highest weight $\lambda$.
Thus it suffices to check the following:

 \begin{theorem}\label{sshtab-upgrade}
If $\lambda \in \NN^n$ is a strict partition, then $\SShTab_n(\lambda)$ is a connected normal $\qq_n$-crystal.
\end{theorem}

\begin{proof}
Theorem~\ref{highest-wt-thm} shows that  the $\qq_n$-crystal $\SShTab_n(\lambda)$ is connected.
Let $\cB_m := \SShTab_n((m))$ by the $\qq_n$-crystal of semistandard shifted tableaux with $m$ boxes all in the first row.
Consider the set of  words $w \in \cW^+_n(m)$
for which there exists  $i \in [m]$
such that
 $w_1  > \dots > w_i   \leq  \dots \leq w_{m}$
and
$w_j \in \ZZ$ for all $j \in [m]\setminus\{i\}$.
This set is a $\qq_n$-subcrystal 
isomorphic to $\cB_m$ via the map that produces a one-row tableau from $w$ by adding primes to the letters $w_1,w_2,\dots,w_{i-1}$
and then sorting the modified word, so that $532'2234 \mapsto \ytab{2' & 2 & 2& 3' & 3 & 4 & 5'}$
for example.
Thus $\cB_m$ is normal.

The character of $\cB_m$ is $q_m := Q_{(m)}(x_1,\dots,x_n)$,
so for any strict partition $\lambda \in \NN^n$
the character of  $\cB_\lambda :=\cB_{\lambda_1}\otimes \cB_{\lambda_2}\otimes \cdots \otimes \cB_{\lambda_n}$
is $q_{\lambda} := q_{\lambda_1} q_{\lambda_2}\cdots q_{\lambda_n}$. 
Each full $\q_n$-subcrystal of the normal $\qq_n$-crystal
$\cB_\lambda $
is isomorphic to  $\SShTab_n(\mu)$ for some strict partition $\mu \in \NN^n$ by Corollary~\ref{any-normal-cor}.
For a given $\mu$, consider the number $g_{\lambda\mu}$ of full subcrystals of $\cB_\lambda $ isomorphic to  $\SShTab_n(\mu)$.
To prove that $\SShTab_n(\lambda)$ is a normal $\qq_n$-crystal, it suffices to show that $g_{\lambda\lambda} > 0$.

As $\ch(\SShTab_n(\mu))=Q_\mu(x_1,\dots,x_n)$ and since the Schur $Q$-polynomials 
are a $\ZZ$-basis 
for $\SymQ(x_1,\dots,x_n)$,
the values of $g_{\lambda\mu}$ are the unique integers such that $q_\lambda = \sum_\mu g_{\lambda\mu} Q_\mu(x_1,\dots,x_n)$.
But it follows from 
 basic properties of Schur $Q$-polynomials
 that $ Q_\lambda(x_1,\dots,x_n)$ appears with coefficient $1$ in $q_\lambda$ (see the equations directly after \cite[Chapter III, (8.8$'$)]{Macdonald})
so  $g_{\lambda\lambda} =1$.
\end{proof}

\begin{corollary}\label{sshtab-upgrade-cor}
Let $\mu \subset \lambda$ be strict partitions.
When nonempty, the set $\SShTab_n(\lambda/\mu)$ 
is a normal (but not necessarily connected) $\gl_n$-crystal relative  to the operators defined in Section~\ref{tab-op-sect}.
\end{corollary}

\begin{proof}
If $\sF$ is the functor defined in the proof of Corollary~\ref{skew-norm-cor}, then
 $\SShTab_n(\lambda/\mu)$ is isomorphic to a subcrystal of 
$\sF(\SShTab_{k+n}(\lambda))$, which is a normal $\gl_{n}$-crystal by Theorem~\ref{sshtab-upgrade}.
\end{proof}

We finally obtain a stronger form of Proposition~\ref{qq-orf-prop}.

 \begin{corollary}\label{orf-upgrade}
If $z \in I_\ZZ$ and $\ORF_n(z)$ nonempty, then $\ORF_n(z)$  is a normal $\qq_n$-crystal.
\end{corollary}

Recall that $\ORF_n(z)\neq \varnothing$ if and only if the partition  $\mu(z)$  in
Proposition~\ref{mu-prop} has $\ell(\mu(z)) \leq n$.

\begin{proof}
This follows from Theorems~\ref{qqo-morphism-thm} and \ref{sshtab-upgrade}.
\end{proof}

The character of $\ORF_n(z)$ for $z \in I_\ZZ$ is the polynomial $ \hat G_z(x_1,x_2,\dots,x_n)$
where $\hat G_z $ is the (rescaled) \defn{involution Stanley symmetric function} studied in \cite[\S4.5]{HMP4}.
The definition of $\hat G_z$ in \cite[\S4.5]{HMP4} is as a generating function for certain decreasing rather than increasing factorizations,
so this identification is not obvious, but follows from \cite[Remark 3.10]{Marberg2019b}.

Each $\hat G_z$ is a finite linear combination of Schur $Q$-functions with nonnegative integer coefficients \cite[Cor. 4.62]{HMP4}.
Corollary~\ref{orf-upgrade} leads to another interpretation of these coefficients via  Theorem~\ref{main-thm}.

\begin{corollary}\label{q-pos-cor}
Suppose $z \in I_\ZZ$ and $n$ is the common length of all words in $\iR(z)$. 
Then $\hat G_z = \sum_\lambda c_{z\lambda} Q_\lambda$ where $c_{z\lambda}$ is the number of $\qq_n$-highest weight
elements in $\ORF_n(z)$ with weight $\lambda$. 
\end{corollary}

One says that $\pi \in S_\ZZ$ is \defn{vexillary} if it is $2143$-avoiding in the sense that  $\pi(i_2) < \pi(i_1) < \pi(i_4) < \pi(i_3)$ 
never holds for $i_1<i_2<i_3<i_4$.
The symmetric function $\hat G_z$ is equal to a Schur $Q$-function if and only $z \in I_\ZZ$
is vexillary \cite[Thm. 1.15]{HMP4}. 
In this case $\hat G_z =  Q_{\mu(z)}$  \cite[Thm. 1.13]{HMP4}.

 \begin{corollary}
 Let $z \in I_\ZZ$.
The normal $\qq_n$-crystal $\ORF_n(z)$ is connected for all $n \geq \ell(\mu(z))$ if and only if $z$ is vexillary,
in which case $\ORF_n(z) \cong \SShTab_n( \mu(z))$.
\end{corollary}


\begin{proof}
 When nonempty, $\ORF_n(z)$ is a connected $\qq_n$-crystal if and only if its character is a
 single Schur $Q$-polynomial, 
 which occurs for all $n$ precisely when $\hat G_z$ is equal to a Schur $Q$-function.
\end{proof}

If $z \in I_\ZZ$ is vexillary, then the normal $\q_n$-crystal
$\Incr_n(z)$ is also connected, but the latter may be connected without $z$ being vexillary;
see \cite[Cor. 4.56]{HMP4}, which via  \cite[Cor. 3.34]{Marberg2019b} gives
a more complicated pattern avoidance condition characterizing when this happens.

\end{document}